\documentclass[a4paper,10pt]{amsart}
\usepackage[utf8]{inputenc}

\usepackage{amsmath}
\usepackage{amstext}
\usepackage{amssymb}
\usepackage{amsthm}
\usepackage{enumerate}
\usepackage{hyperref}
\usepackage[shortlabels]{enumitem}
\usepackage{pgf,tikz}
\usepackage{xcolor}
\usepackage{bbm}

    \DeclareMathOperator{\term}{term}

\DeclareMathOperator{\dom}{dom}
\DeclareMathOperator{\lex}{lex}
\DeclareMathOperator{\otp}{otp}
\DeclareMathOperator{\supp}{supp}
\DeclareMathOperator{\splt}{split}
\DeclareMathOperator{\tp}{tp}
\newcommand{\card}[1]{\vert #1\vert}
\newcommand{\seq}[1]{\langle #1 \rangle }
    
    \makeatletter
\newcommand{\oast}{\mathbin{\mathpalette\make@circled\ast}}
\newcommand{\make@circled}[2]{%
  \ooalign{$\m@th#1\smallbigcirc{#1}$\cr\hidewidth$\m@th#1#2$\hidewidth\cr}%
}
\newcommand{\smallbigcirc}[1]{%
  \vcenter{\hbox{\scalebox{0.77778}{$\m@th#1\bigcirc$}}}%
}
\makeatother

\theoremstyle{plain}
\newtheorem{thm}{Theorem}[section]
\newtheorem{lemma}[thm]{Lemma}
\newtheorem{mainlemma}[thm]{Main Lemma}
\newtheorem{prop}[thm]{Proposition}
\newtheorem{cor}[thm]{Corollary}
\newtheorem{claim}[thm]{Claim}
\newtheorem*{fact}{Fact}
\theoremstyle{definition}
\newtheorem{definition}[thm]{Definition}
\theoremstyle{remark}

\newtheorem{remark}[thm]{Remark}

\newtheorem{quest}{Question}

\title[Tree forcing and hypergraphs]{Tree forcing and definable maximal independent sets in hypergraphs}
\author{Jonathan Schilhan}
\address{School of Mathematics, University of Leeds, Leeds, LS2 9JT, United Kingdom}
\email{j.schilhan@leeds.ac.uk}
\thanks{The author would like to thank the Austrian Science Fund, FWF, for generous support through START Project
	Y1012-N35. We would also like to thank Jörg Brendle, David Choudounský and David Schrittesser for many fruitful discussions.}

\begin{document}

\begin{abstract}
	We show that after forcing with a countable support iteration or a finite product of Sacks or splitting forcing over $L$, every analytic hypergraph on a Polish space admits a $\mathbf{\Delta}^1_2$ maximal independent set. This extends an earlier result by Schrittesser (see \cite{Schrittesser2016}). As a main application we get the consistency of $\mathfrak{r} = \mathfrak{u} = \mathfrak{i} = \omega_2$ together with the existence of a $\Delta^1_2$ ultrafilter,  a $\Pi^1_1$ maximal independent family and a $\Delta^1_2$ Hamel basis. This solves open problems of Brendle, Fischer and Khomskii \cite{Brendle2019} and the author \cite{Schilhan2019}. We also show in ZFC that $\mathfrak{d} \leq \mathfrak{i}_{cl}$, adressing another question from \cite{Brendle2019}. 
\end{abstract}

\maketitle

\section{Introduction}

Throughout mathematics, the existence of various kinds of maximal sets can typically only be obtained by an appeal to the \textit{Axiom of Choice} or one of its popular forms, such as \emph{Zorn's Lemma}. Under certain circumstances, it is possible though, to explicitly define such objects. The earliest result in this direction is probably due to Gödel who noted in \cite[p. 67]{Goedel1940} that in the constructible universe $L$, there is a $\Delta^1_2$ well-order of the reals (see \cite[25]{Jech2013} for a modern treatment). Using similar ideas, many other special sets of reals, such as \emph{Vitali sets}, \emph{Hamel bases} or \emph{mad families}, just to name a few, can be constructed in $L$ in a $\Delta^1_2$ way. This has become by now a standard set theoretic technique. In many cases, these results also give an optimal bound for the complexity of such a set. For example, a Vitali set cannot be Lebesgue measurable and in particular cannot have a $\Sigma^1_1$ or $\Pi^1_1$ definition. In other cases, one can get stronger results by constructing $\Pi^1_1$ witnesses. This is typically done using a coding technique, originally developped by Erdős, Kunen and Mauldin in \cite{Erdoes1981}, later streamlined by Miller (see \cite{Miller1989}) and further generalized by Vidnyánszky (see \cite{Vidnyanszky2014}). For example, Miller showed that there are $\Pi^1_1$ Hamel bases and mad families in $L$. Other results of this type can be found e.g. in \cite{Fischer2010}, \cite{Fischer2013} or \cite{Fischer2019}.
Since the assumption $V=L$ is quite restrictive, it is interesting to know in what forcing extensions of $L$, definable witnesses for the above mentioned kinds of sets still exist. Various such results exist in the literature, e.g. in \cite{Brendle2013}, \cite{Fischer2011}, \cite{Fischer2017}, \cite{Schrittesser2018} or \cite{Fischer2020}.

The starting observation for this paper is that almost all of these examples can be treated in the same framework, as \emph{maximal independent sets in hypergraphs}.

\begin{definition}
	A \emph{hypergraph} $E$ on a set $X$ is a collection of finite non-empty subsets of $X$, i.e. $E \subseteq [X]^{<\omega} \setminus \{ \emptyset\}$. Whenever $Y \subseteq X$, we say that $Y$ is \emph{$E$-independent} if $[Y]^{<\omega} \cap E = \emptyset$. Moreover, we say that $Y$ is \emph{maximal $E$-independent} if $Y$ is maximal under inclusion in the collection of $E$-independent subsets of $X$. 
\end{definition}

Whenever $X$ is a topological space, $[X]^{<\omega}$ is the disjoint sum of the spaces $[X]^n$ for $n \in \omega$. Here, as usual, $[X]^n$, the set of subsets of $X$ of size $n$ becomes a topological space by identification with the quotient of $X^n$ under the equivalence relation $(x_0, \dots, x_{n-1}) \sim (y_0, \dots, y_{n-1})$ iff $\{x_0, \dots, x_{n-1}\} = \{y_0, \dots, y_{n-1} \}$. Whenever $X$ is Polish, $[X]^{<\omega}$ is Polish as well and we can study its definable subsets. In particular, we can study definable hypergraphs on Polish spaces. 

The main result of this paper is the following theorem. 

\begin{thm}\label{thm:maintheorem}
After forcing with the $\omega_2$-length countable support iteration (csi) of Sacks or splitting forcing over $L$, every analytic hypergraph on a Polish space has a $\mathbf{\Delta}^1_2$ maximal independent set.
\end{thm}

This extends a result by Schrittesser \cite{Schrittesser2016}, who proved the above for Sacks forcing, which we denote by $\mathbb{S}$, and ordinary $2$-dimensional graphs (see also \cite{Schrittesser2018}). For equivalence relations this was already known by Budinas \cite{Budinas84}. We will also prove the case of finite products but our main focus will be on the countable support iteration. \emph{Splitting forcing} $\mathbb{SP}$ (Definition~\ref{def:splittingforcing}) is a less-known forcing notion that was originally introduced by Shelah in \cite{Shelah1992} and has been studied in more detail recently (\cite{Spinas2004}, \cite{Spinas2007}, \cite{HeinSpinas2019} and \cite{LaguzziMildenbergerStuber-Rousselle2020}). Although it is very natural and gives a minimal way to add a \emph{splitting real} (see more below), it has not been exploited a lot and to our knowledge, there is no major set theoretic text treating it in more detail.

Our three guiding examples for Theorem~\ref{thm:maintheorem} will be \emph{ultrafilters}, \emph{maximal independent families} and \emph{Hamel bases}. 

Recall that an \emph{ultrafilter} on $\omega$ is a maximal subset $\mathcal{U}$ of $\mathcal{P}(\omega)$ with the \emph{strong finite intersection property}, i.e. the property that for any $\mathcal{A} \in [\mathcal{U}]^{<\omega}$, $\vert \bigcap \mathcal{A} \vert = \omega$.\footnote{In this article, all ultrafilters are considered non-principal.} Thus, letting $E_u := \{ \mathcal{A} \in [\mathcal{P}(\omega)]^{<\omega} : \vert \bigcap \mathcal{A} \vert < \omega  \}$, an ultrafilter is a maximal $E_u$-independent set. In \cite{Schilhan2019}, we studied the projective definability of ultrafilters and introduced the cardinal invariant $\mathfrak{u}_B$, which is the smallest size of a collection of Borel subsets of $\mathcal{P}(\omega)$ whose union is an ultrafilter. If there is a $\mathbf{\Sigma}^1_2$ ultrafilter, then $\mathfrak{u}_B = \omega_1$, since every $\mathbf{\Sigma}^1_2$ set is the union of $\omega_1$ many Borel sets. Recall that the classical ultrafilter number $\mathfrak{u}$ is the smallest size of an ultrafilter base. We showed in \cite{Schilhan2019}, that $\mathfrak{u}_B \leq \mathfrak{u}$ and asked whether it is consistent that $\mathfrak{u}_B < \mathfrak{u}$ or even whether a $\Delta^1_2$ ultrafilter can exist while $\omega_1 <\mathfrak{u}$. The difficulty is that we have to preserve a definition for an ultrafilter, while its interpretation in $L$ must be destroyed. This has been achieved before for mad families (see \cite{Brendle2013}).

An \emph{independent family} is a subset $\mathcal{I}$ of $\mathcal{P}(\omega)$ so that for any disjoint $\mathcal{A}_0, \mathcal{A}_1 \in [\mathcal{I}]^{<\omega}$, $\vert \bigcap_{x \in \mathcal{A}_0} x \cap \bigcap_{x \in \mathcal{A}_1} \omega \setminus x \vert = \omega$. It is called \emph{maximal independent family} if it is additionally maximal under inclusion. Thus, letting $E_i = \{ \mathcal{A}_0 \dot\cup \mathcal{A}_1 \in [\mathcal{P}(\omega)]^{<\omega} :\vert \bigcap_{x \in \mathcal{A}_0} x \cap \bigcap_{x \in \mathcal{A}_1} \omega \setminus x \vert < \omega \}$, a maximal independent family is a maximal $E_i$-independent set. The definability of maximal independent families was studied by Miller in \cite{Miller1989}, who showed that they cannot be analytic, and recently by Brendle, Fischer and Khomskii in \cite{Brendle2019}, where they introduced the invariant $\mathfrak{i}_B$, the least size of a collection of Borel sets whose union is a maximal independent family. The classical independence number $\mathfrak{i}$ is simply the smallest size of a maximal independent family. In \cite{Brendle2019}, it was asked whether $\mathfrak{i}_B < \mathfrak{i}$ is consistent and whether there can be a $\Pi^1_1$ maximal independent family while $\omega_1 < \mathfrak{i}$. In the same article, it was shown that the existence of a $\Delta^1_2$ maximal independent family is equivalent to that of a $\Pi^1_1$ such family. The difficulty in the problem is similar to that before.

A \emph{Hamel basis} is a vector-space basis of $\mathbb{R}$ over the field of rationals $\mathbb{Q}$. Thus, letting $E_h := \{ \mathcal{A} \in [\mathbb{R}]^{<\omega} : \mathcal{A} \text{ is linearly dependent over } \mathbb{Q} \}$, a Hamel basis is a maximal $E_h$-independent set. A Hamel basis must be as large as the continuum itself. This is reflected in the fact that, when adding a real, every ground-model Hamel basis is destroyed. But still it makes sense to ask how many Borel sets are needed to get one. Miller, also in \cite{Miller1989}, showed that a Hamel basis can never be analytic. As before, we may ask whether there can be a $\Delta^1_2$ Hamel basis while CH fails. Again, destroying ground-model Hamel bases, seems to pose a major obstruction. 

The most natural way to increase $\mathfrak{u}$ and $\mathfrak{i}$ is by iteratively adding \emph{splitting reals}. Recall that for $x,y \in \mathcal{P}(\omega)$, we say that $x$ \emph{splits} $y$ iff $ \vert x \cap y \vert = \omega$ and $\vert y \setminus x \vert = \omega$. A real $x$ is called \emph{splitting over $V$} iff for every $y \in \mathcal{P}(\omega) \cap V$, $x$ splits $y$. The classical forcing notions adding splitting reals are \emph{Cohen}, \emph{Random} and \emph{Silver forcing} and all forcings that add so called \emph{dominating reals}. It was shown though, in \cite{Schilhan2019}, that after forcing with any of these, a $\mathbf\Sigma^1_2$ definition with ground model parameters will not define an ultrafilter and the same argument can be applied to independent families. For this reason, we are going to use the forcing notion $\mathbb{SP}$ that we mentioned above. As an immediate corollary of Theorem~\ref{thm:maintheorem}, we get the following. 

\begin{thm}\label{thm:mainthm2}
It is consistent that $\mathfrak{r} = \mathfrak{u} = \mathfrak{i} = \omega_2$ while there is a $\Delta^1_2$ ultrafilter, a $\Pi^1_1$ maximal independent family and a $\Delta^1_2$ Hamel basis. In particular, we get the consistency of $\mathfrak{i}_B, \mathfrak{u}_B < \mathfrak{r}, \mathfrak{i}, \mathfrak{u}$. 
\end{thm}

Here, $\mathfrak{r}$ is the reaping number, the least size of a set $\mathcal{S} \subseteq \mathcal{P}(\omega)$ so that there is no splitting real over $\mathcal{S}$. This solves the above mentioned questions from \cite{Schilhan2019} and \cite{Brendle2019}. Moreover, Theorem~\ref{thm:maintheorem} gives a ``black-box" way to get many results, saying that certain definable families exists in the Sacks model.

In \cite{Brendle2019}, another cardinal invariant $\mathfrak{i}_{cl}$ is introduced, which is the smallest size of a collection of closed sets, whose union is a maximal independent family. Similarly, one can define a closed version of the ultrafilter number, $\mathfrak{u}_{cl}$. Here, it is irrelevant whether we consider closed subsets of $[\omega]^\omega$ or $\mathcal{P}(\omega)$, since every closed subset of $[\omega]^\omega$ with the strong finite intersection property is $\sigma$-compact (see Lemma~\ref{lem:sigmacompind}).  In the model of Theorem~\ref{thm:mainthm2}, we have that $\mathfrak{i}_{cl} = \mathfrak{i}_B$ and $\mathfrak{u}_{cl} = \mathfrak{u}_B$, further answering the questions of Brendle, Fischer and Khomskii. On the other hand we show that $\mathfrak{d} \leq \mathfrak{i}_{cl}$, mirroring Shelah's result that $\mathfrak{d} \leq \mathfrak{i}$ (see \cite{Vaughan1990}). Here, $\mathfrak{d}$ is the dominating number, the least size of a dominating family in $(\omega^\omega, <^*)$.

\begin{thm}\label{thm:icleqd}
(ZFC) $\mathfrak{d} \leq \mathfrak{i}_{cl}$. 
\end{thm}

The paper is organized as follows. In Section 2, we will consider basic results concerning iterations of tree forcings. This section is interesting in its own right and can be read independently from the rest. More specifically, we prove a version of continuous reading of names for countable support iterations that is widely applicable (Lemma~\ref{lem:nicemaster}). In Section 3, we prove our main combinatorial lemma (Main Lemma~\ref{thm:mainlemma} and \ref{lem:mainlemmainf}) which is at the heart of Theorem~\ref{thm:maintheorem}. As for Section 2, Section 3 can be read independently of the rest, since our result is purely descriptive set theoretical. In Section 4, we introduce splitting and Sacks forcing and place it in bigger class of forcings to which we can apply the main lemma. This combines the results from Section 2 and 3. In Section 4, we bring everything together and prove Theorem~\ref{thm:maintheorem}, \ref{thm:mainthm2} and \ref{thm:icleqd}. We end with concluding remarks concerning the further outlook of our technique and pose some questions. 
 
\section{Tree forcing}
Let $A$ be a fixed countable set, usually $\omega$ or $2$.
\begin{enumerate}[label=(\alph*)]
	\item A \textit{tree} $T$ on $A$ is a subset of $A^{<\omega}$ so that for every $t \in T$ and $n < \vert t \vert$, $t \restriction n \in  T$, where $\vert t \vert$ denotes the length of $t$. For $s_0, s_1 \in A^{<\omega}$, we write $s_0 \perp s_1$ whenever $s_0 \not\subseteq s_1$ and $s_1 \not\subseteq s_0$.
	\item $T$ is \textit{perfect} if for every $t \in T$ there are $s_0, s_1 \in T$ so that $s_0, s_1\supseteq t$ and $s_0 \perp s_1$.
	\item A node $t \in T$ is called a \textit{splitting node}, if there are $i \neq j \in A$ so that $t^\frown i, t^\frown j \in T$. The set of splitting nodes in $T$ is denoted $\splt(T)$. We define $\splt_n(T)$ to be the set of $t \in \splt(T)$ such that there are exactly $n$ splitting nodes below $t$ in $T$. The finite subtree of $T$ generated by $\splt_n(T)$ is denoted $\splt_{\leq n}(T)$. 
	\item For any $t \in T$ we define the restriction of $T$ to $t$ as $T_t = \{ s \in T : s \not\perp t  \}$.
	\item The set of branches through $T$ is denoted by $[T] = \{ x \in A^\omega : \forall n \in \omega (x \restriction n \in T) \}$.
	\item $A^{\omega}$ carries a natural Polish topology generated by the clopen sets $[t] = \{ x \in A^\omega : t \subseteq x \}$ for $t \in A^{<\omega}$. Then $[T]$ is closed in $A^\omega$.
	\item Whenever $X \subseteq A^{\omega}$ is closed, there is a continuous \textit{retraction} $\varphi \colon A^\omega \to X$, i.e. $\varphi''A^\omega = X$ and $\varphi \restriction X$ is the identity.  
	\item A \textit{tree forcing} is a collection $\mathbb{P}$ of perfect trees ordered by inclusion. 
	\item By convention, all tree forcings are closed under restrictions, i.e. if $T \in \mathbb{P}$ and $t \in T$, then $T_t \in \mathbb{P}$, and the trivial condition is $A^{<\omega}$. 
	\item The set $\mathcal{T}$ of perfect subtrees of $A^{<\omega}$ is a $G_\delta$ subset of $\mathcal{P}(A^{<\omega}) \cong \mathcal{P}(\omega)$, where we identify $A^{<\omega}$ with $\omega$, and thus carries a natural Polish topology. It is not hard to see that it is homeomorphic to $\omega^\omega$, when $\vert A \vert \geq 2$. 
	\item Often times, we will use a bar above a variable, as in ``$\bar x$", to indicate that it denotes a sequence. In that case, we either write $x(\alpha)$ or $x_\alpha$ to denote the $\alpha$'th element of that sequence, depending on the context.  
	\item Let $\langle T_i : i < \alpha \rangle$ be a sequence of trees where $\alpha$ is an arbitrary ordinal. Then we write $\bigotimes_{i < \alpha} T_i$ for the set of finite partial sequences $\bar s$ where $\dom \bar s \in [\alpha]^{<\omega}$ and for every $i \in \dom \bar s$, $s(i) \in T_i$.
	\item $(A^\omega)^\alpha$ carries a topology generated by the sets $[\bar s] = \{ \bar x \in (A^\omega)^\alpha : \forall i \in \dom \bar s (x(i) \in [s(i)]) \}$ for $\bar s \in \bigotimes_{i < \alpha} A^{<\omega}$. 
	\item Whenever $X \subseteq (A^\omega)^\alpha$ and $C \subseteq \alpha$, we define the \textit{projection of $X$ to $C$} as $X \restriction C = \{ \bar x \restriction C : \bar x \in X \}$.
\end{enumerate}

\begin{fact}
	Let $\mathbb{P}$ be a tree forcing and $G$ a $\mathbb{P}$-generic filter over $V$. Then $\mathbb{P}$ adds a real $ x_G := \bigcup \{ s \in A^{<\omega} : \forall T \in G (s \in T) \} \in A^\omega$. Moreover, $V[G] = V[x_G]$.
\end{fact}

\begin{definition}\label{def:axiomA}
	We say that $(\mathbb{P}, \leq)$ is \textit{Axiom A} if there is a decreasing sequence of partial orders $\langle \leq_n : n \in \omega \rangle$ refining $\leq$ on $\mathbb{P}$ so that
	
	\begin{enumerate}
		\item for any $n \in \omega$ and $T, S \in \mathbb{P}$, if $S \leq_n T$, then $S \cap A^{<n} = T \cap A^{<n}$,
		\item for any fusion sequence, i.e. a sequence $\langle p_n : n \in \omega \rangle$ where $p_{n+1} \leq_n p_n$ for every $n$, $p = \bigcap_{n \in \omega} p_n \in \mathbb{P}$ and $p \leq_n p_n$ for every $n$, 
		\item and for any maximal antichain $D \subseteq \mathbb{P}$, $p \in \mathbb{P}$, $n \in \omega$, there is $q \leq_n p$ so that $\{ r \in D : r \not\perp q \}$ is countable. 
	\end{enumerate}
	Moreover we say that $(\mathbb{P}, \leq)$ is \textit{Axiom A with continuous reading of names} (\textit{crn}) if there is such a sequence of partial orders so that additionally,
	\begin{enumerate}
		\setcounter{enumi}{3}
		\item for every $p \in \mathbb{P}$, $n \in \omega$ and $\dot y$ a $\mathbb{P}$-name for an element of a Polish space\footnote{In the generic extension $V[G]$ we reinterpret $X$ as the completion of $(X)^V$. Similarly, we reinterpret spaces $(A^\omega)^\alpha$, continuous functions, open and closed sets on these spaces. This should be standard.} $X$, there is $q \leq_n p$ and a continuous function $f \colon [q] \to X$ so that $$ q \Vdash \dot y [G] = f(x_G).$$
	\end{enumerate} 
\end{definition}

Although (1) is typically not part of the definition of Axiom A, we include it for technical reasons. The only classical example that we are aware of, in which it is not clear whether (1)-(4) can be realized simultaneously, is Mathias forcing. 

Let $\langle \mathbb{P}_\beta, \dot{\mathbb{Q}}_\beta : \beta < \alpha \rangle$ be a countable support iteration of tree forcings that are Axiom A with crn, where for each $\beta < \alpha$, $$\Vdash_{\mathbb{P}_\beta} \text{``}\langle \dot{\leq}_{\beta,n} : n \in \omega \rangle \text{ witnesses that } \dot{\mathbb{Q}}_{\beta} \text{ is Axiom A with crn''}.$$

\begin{enumerate}[label=(\alph*)]
	\setcounter{enumi}{13}
	\item For each $n \in \omega, a \subseteq \alpha$, we define $\leq_{n,a}$ on $\mathbb{P}_\alpha$, where $$\bar q \leq_{n,a} \bar p \leftrightarrow \left(\bar q \leq \bar p \wedge \forall \beta \in a ( \bar q \restriction \beta \Vdash_{\mathbb{P}_\beta} \dot q(\beta) \dot{\leq}_{\beta,n} \dot p(\beta))\right).$$
	\item  The support of $\bar p \in \mathbb{P}_\alpha$ is the set $\supp (\bar p) = \{ \beta < \alpha : \bar p \Vdash \dot p(\beta) \neq \mathbbm{1} \}$.
\end{enumerate}

Recall that a condition $q$ is called a master condition over a model $M$ if for any maximal antichain $D \in M$, $\{ p \in D : q \not\perp p \} \subseteq M$. Equivalently, it means that for every generic filter $G$ over $V$ containing $q$, $G$ is generic over $M$ as well. Throughout this paper, when we say that \emph{$M$ is elementary}, we mean that it is elementary in a large enough model of the form $H(\theta)$. Sometimes, we will say that \emph{$M$ is a model of set theory} or just that \emph{$M$ is a model}. In most generality, this just mean that $(M,\in)$ satisfies a strong enough fragment of ZFC. But this is a way too general notion for our purposes. For instance, such $M$ may not even be correct about what $\omega$ is. Thus, let us clarify that in all our instances this will mean, that $M$ is either elementary or an extension of an elementary model by a countable (in $M$) forcing. In particular, some basic absoluteness (e.g. for $\Sigma^1_1$ or $\Pi^1_1$ formulas) holds true between $M$ and $V$, $M$ is transitive below $\omega_1$ and $\omega_1$ is computed correctly.

\begin{fact}[{Fusion Lemma, see e.g. \cite[Lemma 1.2, 2.3]{Baumgartner1979}}]
If $\langle a_n : n \in \omega \rangle$ is $\subseteq$-increasing, $\langle \bar p_n : n \in \omega \rangle$ is such that $\forall n \in \omega (\bar p_{n+1} \leq_{n,a_n} \bar p_n)$ and $\bigcup_{n \in \omega} \supp(\bar p_n) \subseteq \bigcup_{n \in \omega} a_n \subseteq \alpha$, then there is a condition $\bar p \in \mathbb{P}_\alpha$ so that for every $n \in \omega$, $\bar p \leq_{n,a_n} \bar p_n$; in fact, for every $\beta < \alpha$, $\bar p \restriction \beta \Vdash \dot p(\beta) = \bigcap_{n \in \omega} \dot p_n(\beta)$.

Moreover, let $M$ be a countable elementary model, $\bar p \in M \cap \mathbb{P}_\alpha$, $n \in \omega$, $a \subseteq M\cap\alpha$ finite and $\langle \alpha_i : i \in \omega \rangle$ a cofinal increasing sequence in $M\cap \alpha$. Then there is $\bar q \leq_{n,a} \bar p$ a master condition over $M$ so that for every name $\dot y \in M$ for an element of $\omega^\omega$ and $j \in \omega$, there is $i \in \omega$ so that below $\bar q$, the value of $\dot y \restriction j$ only depends on the $\mathbb{P}_{\alpha_i}$-generic.  
\end{fact}

\begin{enumerate}[label=(\alph*)]
	\setcounter{enumi}{15}
	\item For $G$ a $\mathbb{P}_\alpha$-generic, we write $\bar x_G$ for the generic element of $\prod_{\beta <\alpha}A^\omega$ added by $\mathbb{P}_\alpha$.
\end{enumerate}

Let us from now on assume that for each $\beta < \alpha$ and $n \in \omega$, $\mathbb{Q}_\beta$ and $\leq_{\beta,n}$ are fixed analytic subsets subsets of $\mathcal{T}$ and $\mathcal{T}^2$ respectively, coded in $V$. Although the theory that we develop below can be extended to a large extent to non-definable iterands, we will only focus on this case, since we need stronger results later on. 

\begin{lemma}
		\label{lem:nicemaster}
		For any $\bar p \in \mathbb{P}_\alpha$, $M$ a countable elementary model so that ${\mathbb{P}_\alpha, \bar p \in M}$ and $n \in \omega, a \subseteq M\cap \alpha$ finite, there is $\bar q \leq_{n,a} \bar p$ a master condition over $M$ and a closed set $[\bar q] \subseteq (A^\omega)^\alpha$ so that
		\begin{enumerate}
			\item $\bar q \Vdash \bar x_G \in [\bar q]$,
			
		\end{enumerate}
		
		for every $\beta < \alpha$,

		\begin{enumerate}
			\setcounter{enumi}{1}
			
			\item $\bar q \Vdash \dot q(\beta) = \{s \in A^{<\omega} : \exists \bar z \in [\bar q] (\bar z \restriction \beta = \bar x_{G} \restriction \beta \wedge s \subseteq z(\beta))  \}  $,

			\item the map sending $\bar x \in [\bar q] \restriction \beta$ to $\{s \in A^{<\omega} : \exists \bar z \in [\bar q] (\bar z \restriction \beta = \bar x \wedge s \subseteq z(\beta))  \}$ is continuous and maps to $\mathbb{Q}_\beta$,
			
			\item $[\bar q] \restriction \beta \subseteq (A^\omega)^\beta$ is closed,
		\end{enumerate}
		
		and for every name $\dot y \in M$ for an element of a Polish space $X$,
		
		\begin{enumerate}
			\setcounter{enumi}{4}
			\item there is a continuous function $f \colon [\bar q] \to X$ so that $\bar q \Vdash \dot y = f(\bar x_G).$
		\end{enumerate}		
	\end{lemma}

	\begin{enumerate}[label=(\alph*)]
		\setcounter{enumi}{16}
		\item We call such $\bar q$ as in Lemma~\ref{lem:nicemaster} a \textit{good master condition over $M$}.
	\end{enumerate}
Before we prove Lemma~\ref{lem:nicemaster}, let us draw some consequences from the definition of a good master condition. 

\begin{lemma}\label{lem:intermediategoodanalytic}\label{lem:goodmaster2}
Let $\bar q \in \mathbb{P}_\alpha$ be a good master condition over a model $M$ and $\dot y \in M$ a name for an element of a Polish space $X$. 
\begin{enumerate}[label=(\roman*)]
\item Then $[\bar q]$ is unique, in fact it is the closure of $\{ \bar x_G : G \ni \bar q \text{ is generic over } V\}$ in any forcing extension $W$ of $V$ where $\left(\beth_{\omega}(\vert \mathbb{P}_\alpha \vert)\right)^V$ is countable.
\item The continuous map $f \colon [\bar q] \to X$ given by (5) is unique and
\item whenever $Y \in M$ is an analytic subset of $X$ and $\bar q \Vdash \dot y \in Y$, then $f''[\bar q] \subseteq Y$. 
\end{enumerate}
Moreover, there is a countable set $C \subseteq \alpha$, not depending on $\dot y$, so that
\begin{enumerate}[label=(\roman*)]
\setcounter{enumi}{3}
\item $[\bar q] \restriction C$ is a closed subset of the Polish space $(A^{\omega})^C$ and $[\bar q] = ([\bar q] \restriction C) \times (A^\omega)^{\alpha \setminus C}$,
\item for every $\beta \in C$, there is a continuous function $g \colon [\bar q] \restriction (C \cap \beta)\to \mathbb{Q}_\beta$, so that for every $\bar x \in [\bar q]$, $$g(\bar x \restriction (C \cap \beta)) = \{s \in A^{<\omega} : \exists \bar z \in [\bar q] (\bar z \restriction \beta = \bar x \restriction \beta \wedge s \subseteq z(\beta))  \} ,$$
\item there is a continuous function $f \colon [\bar q] \restriction C \to X$, so that $$\bar q \Vdash \dot y = f(\bar x_G \restriction C).$$  
\end{enumerate}\end{lemma}

\begin{proof}
Let us write, for every $\beta < \alpha$ and $\bar x \in [\bar q] \restriction \beta$, $$T_{\bar x} := \{s \in A^{<\omega} : \exists \bar z \in [\bar q] (\bar z \restriction \beta = \bar x \wedge s \subseteq z(\beta))  \}.$$

For (i), let $W$ be an extension in which $\left(\beth_{\omega}(\vert \mathbb{P}_\alpha \vert)\right)^V$ is countable and let $\bar s \in \bigotimes_{i < \alpha} A^{<\omega}$ be arbitrary so that $[\bar s] \cap [\bar q]$ is non-empty. We claim that there is a generic $G$ over $V$ containing $\bar q$ so that $\bar x_G \in [\bar s]$.  This is shown by induction on $\max(\dom (\bar s))$. For $\bar s = \emptyset$ the claim is obvious. Now assume $\max(\dom (\bar s)) = \beta$, for $\beta < \alpha$. Then, by (3), $O := \{ \bar x \in [\bar q] : s(\beta) \in T_{\bar x \restriction \beta}\}$ is open and it is non-empty since $[\bar s] \cap [\bar q] \neq \emptyset$. Applying the inductive hypothesis, there is a generic $G \ni \bar q$ so that $\bar x_G \in O$. In $V[G \restriction \beta]$ we have, by (2), that $T_{\bar x_G \restriction \beta} = \dot q(\beta)[G]$.  Moreover, since $\bar x_G \in O$, we have that $s(\beta) \in \dot q(\beta)[G]$. Then it is easy to force over $V[G \restriction \beta]$, to get a full $\mathbb{P}_\alpha$ generic $H \supseteq G \restriction \beta$ containing $\bar q$ so that $\bar x_{H} \restriction \beta = \bar x_G \restriction \beta $ and $s(\beta) \subseteq \bar x_H(\beta)$. By (1), for every generic $G$ over $V$ containing $\bar q$, $\bar x_G \in [\bar q]$. Thus we have shown that the set of such $\bar x_G$ is dense in $[\bar q]$. Uniqueness follows from $[\bar q]$ being closed and the fact that if two closed sets coded in $V$ agree in $W$, then they agree in $V$. This follows easily from $\mathbf\Pi^1_1$ absoluteness.

Now (ii) follows easily since any two continuous functions given by (5) have to agree on a dense set in an extension $W$ and thus they agree in $V$. Again this is an easy consequence of $\mathbf\Pi^1_1$ absoluteness.

For (iii), let us consider the analytic space $Z = \{0\} \times X \cup \{1\} \times Y$, which is the disjoint union of the spaces $X$ and $Y$. Then there is a continuous surjection $F \colon \omega^\omega \to Z$ and by elementarity we can assume it is in $M$. Let us find in $M$ a name $\dot z$ for an element of $\omega^\omega$ so that in $V[G]$, if $\dot y[G] \in Y$, then $F(\dot z[G]) = (1,\dot y[G])$, and if $\dot y[G] \notin Y$, then $F(\dot z[G]) = (0,\dot y[G])$. By (5), there is a continuous function $g \colon [\bar q] \to \omega^\omega$ so that $\bar q \Vdash \dot z = g(\bar x_G)$. Since $\bar q \Vdash \dot y \in Y$, we have that for any generic $G$ containing $\bar q$, $F(g(\bar x_G)) = (1,f(\bar x_G))$. By density, for every $\bar x \in [\bar q]$, $F(g(\bar x)) = (1,f(\bar x))$ and in particular $f(\bar x) \in Y$. 

Now let us say that the support of a function $g \colon [\bar q] \to X$ is the smallest set $C_g \subseteq \alpha$ so that the value of $g(\bar x)$ only depends on $\bar x \restriction C_g$. The results of \cite{Bockstein1948} imply that if $g$ is continuous, then $g$ has countable support. 
Note that for all $\beta \notin \supp(\bar q)$, the map in (3) is constant on the set of generics and by continuity it is constant everywhere. Thus it has empty support. Let $C$ be the union of $\supp(\bar q)$ with all the countable supports given by instances of (3) and (5). Then $C$ is a countable set. For (iv), (v) and (vi), note that $[\bar q] \restriction C = \{ \bar y \in (A^\omega)^C : \bar y^\frown (\bar x \restriction \alpha \setminus C) \in [\bar q] \}$ for $\bar x \in [\bar q]$ arbitrary, and recall that in a product, sections of closed sets are closed and continuous functions are coordinate-wise continuous.
\end{proof}

\begin{proof}[Proof of Lemma~\ref{lem:nicemaster}]
Let us fix for each $\beta < \alpha$ a continuous surjection $F_\beta \colon \omega^\omega \to \mathbb{Q}_\beta$. 
The proof is by induction on $\alpha$. If $\alpha = \beta +1$, then $\mathbb{P}_\alpha = \mathbb{P}_\beta * \dot{\mathbb{Q}}_{\beta}$. Let $\bar q_0 \leq_{n,a} \bar p \restriction \beta$ be a master condition over $M$ and $H \ni \bar q_0$ a $\mathbb{P}_\beta$ generic over $V$. Then, applying a standard fusion argument using Axiom A with continuous reading of names in $V[H]$ to $\mathbb{Q}_\beta$, we find $q(\beta) \leq_{\beta,n} p(\beta)$ a master condition over $M[H]$ (note that $H$ is also $M$ generic since $\bar q_0$ is a master condition over $M$) so that for each name $\dot y \in M[H]$ for an element of a Polish space $X$ there is a continuous function $f \colon [q(\beta)] \to X$ so that $q(\beta) \Vdash \dot y = f(\dot{x}_G)$. Thus we find in $V$, a $\mathbb{P}_\beta$-name $\dot q(\beta)$ so that $\bar q_0$ forces that it is such a condition. Let $M^+ \ni M$ be a countable elementary model containing $\dot q(\beta)$ and $\bar q_0$, and let $\bar q_{1/2} \leq_{n,a} \bar q_{0}$ be a master condition over $M^+$. Again let $M^{++}\ni M^+$  be a countable elementary model containing $\bar q_{1/2}$. By the induction hypothesis we find $\bar q_1 \leq_{n,a} \bar q_{1/2}$ a good master condition over $M^{++}$. Finally, let $\bar q = \bar q_1 ^\frown \dot{q}(\beta)$. Then $\bar q \leq_{n,a} \bar p$ and $\bar q$ is a master condition over $M$. Since $\dot q(\beta) \in M^+\subseteq M^{++}$, there is a continuous function $f \colon [\bar q_1] \to \omega^\omega$, so that $\bar q_1 \Vdash_\beta F_\beta(f(\bar x_{H})) = \dot q(\beta)$. Here note that $F_\beta$ is in $M$ by elementarity and we indeed find a name $\dot z$ in $M^+$ so that $\bar q_0 \Vdash F_\beta(\dot z) = \dot q(\beta)$. Let $[ \bar q] = \{ \bar x \in (A^{\omega})^\alpha : \bar x \restriction \beta \in [\bar q_1] \wedge x(\beta) \in [F_\beta(f(\bar x \restriction \beta))]  \}$. Then $[\bar q]$ is closed and (1), (2), (3), (4) hold true. To see that $[\bar q]$ is closed, note that the graph of a continuous function is always closed, when the codomain is a Hausdorff space. For (5), let $\dot y \in M$ be a $\mathbb{P}_\alpha$-name for an element of a Polish space $X$. If $H \ni \bar q_1$ is $V$-generic, then there is a continuous function $g \colon [q(\beta)] \to X$ in $V[H]$ so that $V[H] \models q(\beta) \Vdash g(\dot x_G) = \dot y$, where we view $\dot y$ as a $\mathbb{Q}_\beta$-name in $M[H]$. Moreover there is a continuous retraction $\varphi \colon A^\omega \to [q(\beta)]$ in $V[H]$. Since $M^+$ was chosen elementary enough, we find names $\dot g$ and $\dot \varphi$ for $g$ and $\varphi$ in $M^+$. The function $g \circ \varphi$ is an element of the space\footnote{The topology is such that for any continuous $h$ mapping to $C(A^\omega, X)$, $(x,y) \mapsto h(x)(y)$ is continuous.} $C(A^\omega, X)$, but this is not a Polish space when $A$ is infinite, i.e. when $A^\omega$ is not compact. It is though, always a coanalytic space (consult e.g. \cite[12, 2.6]{Kechris1995} to see how $C(A^\omega, X)$ is a coanalytic subspace of a suitable Polish space). Thus there is an increasing sequence $\langle Y_\xi : \xi < \omega_1 \rangle$ of analytic subspaces such that $\bigcup_{\xi < \omega_1} Y_\xi = C(A^\omega, X)$ and the same equality holds in any $\omega_1$-preserving extension. Since $\bar q_{1/2}$ is a master condition over $M^+$, we have that $\bar q_{1/2} \Vdash \dot g \circ \dot \varphi \in Y_\xi$, where $\xi = M^+ \cap \omega_1$. Since $\bar q_1$ is a good master condition over $M^{++}$ and $Y_\xi \in M^{++}$, by Lemma~\ref{lem:intermediategoodanalytic}, there is a continuous function $g' \in V$, $g' \colon [\bar q_1] \to Y_\xi$, so that $\bar q_1 \Vdash g'(\bar x_H) = \dot g \circ \dot \varphi$. Altogether we have that $\bar q \Vdash \dot y = g'(\bar x_G \restriction \beta)(x_G(\beta))$.

For $\alpha$ limit, let $\langle \alpha_i : i \in \omega \rangle$ be a strictly increasing sequence cofinal in $M \cap \alpha$ and let $\bar q_0 \leq_{n,a} \bar p$ be a master condition over $M$ so that for every name $\dot y \in M$ for an element of $\omega^\omega$, $j \in \omega$, the value of $\dot y \restriction j$ only depends on the generic restricted to $\mathbb{P}_{\alpha_i}$ for some $i \in \omega$. Let us fix a ``big" countable elementary model $N$, with $\bar q_0, M \in N$. Let $\langle a_i : i \in \omega\rangle$ be an increasing sequence of finite subsets of $N \cap \alpha$ so that $a_0 = a$ and $\bigcup_{i \in \omega} a_i = N \cap \alpha$. Now inductively define sequences $\langle M_i : i \in \omega\rangle$, $\langle \bar r_i : i \in \omega \rangle$, initial segments lying in $N$, so that for every $i \in \omega$,
		
		\begin{enumerate}[label=-]
			\item $M_0 = M$, $\bar r_0 = \bar q_0 \restriction \alpha_0$,
			\item $M_{i+1} \ni \bar q_0$ is a countable model, 
			\item $M_{i}, \bar r_i, a_i \in M_{i+1}$
			\item $\bar r_i$ is a good $\mathbb{P}_{\alpha_i}$ master condition over $M_i$, 
			\item $r_{i+1} \leq_{n + i,a_i\cap \alpha_i} r_i ^\frown \bar q_0 \restriction [\alpha_i, \alpha_{i+1})$.
		\end{enumerate}
		
		Define for each $i \in \omega$, $\bar q_i = \bar r_i^\frown \bar q_0 \restriction [\alpha_i, \alpha) $. Then $\langle \bar q_i : i \in \omega \rangle$ is a fusion sequence in $\mathbb{P}_\alpha$ and we can find a condition $\bar q \leq_{n,a} \bar q_0 \leq_{n,a} \bar p$, where for each $\beta < \alpha$, $\bar q \restriction \beta \Vdash \dot q(\beta) = \bigcap_{i \in \omega} \dot q_i(\beta)$. Finally let $[\bar q] := \bigcap_{i \in \omega} ([\bar r_i] \times (A^\omega)^{[\alpha_i, \alpha)})$. Then (1) is easy to check. For (5), we can assume without loss of generality that $\dot y$ is a name for an element of $\omega^\omega$ since for any Polish space $X$, there is a continuous surjection from $\omega^\omega$ to $X$. Now let $(i_j)_{j \in \omega}$ be increasing so that $\dot y \restriction j$ is determined on $\mathbb{P}_{\alpha_{i_j}}$ for every $j \in \omega$. Since $\bar r_{i_j}$ is a good master condition over $M$, there is a continuous function $f_j \colon [\bar r_{i_j}] \to \omega^j$ so that $\bar r_{i_j} \Vdash \dot y \restriction j = f_j(\bar x_{G_{\alpha_{i_j}}})$ for every $j \in \omega$. It is easy to put these functions together to a continuous function $f \colon [\bar q] \to 2^\omega$, so that $f(\bar x) \restriction j = f_j(\bar x \restriction \alpha_{i_j})$. Then we obviously have that $\bar q \Vdash \dot y = f(\bar x_G)$. 
		
		Now let us fix for each $i \in \omega$, $C_i \subseteq \alpha_i$ a countable set as given by Lemma~\ref{lem:intermediategoodanalytic} applied to $\bar r_i$, $M_i$, which by elementarity exists in $N$. Let $C = \bigcup_{i \in \omega} C_i$. Then $[\bar q] = [\bar q] \restriction C \times (A^\omega)^{\alpha \setminus C}$ and $[\bar q] \restriction C$ is closed. For every $\beta \in \alpha \setminus C$, the map given in (3) is constant and maps to $\mathbb{Q}_\beta$, as $A^{<\omega}$ is the trivial condition. Thus we may restrict our attention to $\beta \in C$. Let us write $X_i = ([\bar r_i] \times (A^\omega)^{[\alpha_i, \alpha)}) \restriction C$ for every $i \in \omega$ and note that $\bigcap_{i \in \omega} X_i = [\bar q] \restriction C$. For every $\beta \in C$, $\bar x \in [\bar q] \restriction (C \cap \beta)$ and $i \in \omega$, we write $$T_{\bar x} := \{s \in A^{<\omega} : \exists \bar z \in [\bar q] \restriction C (\bar z \restriction \beta = \bar x \wedge s \subseteq z(\beta))  \}$$ and $$T_{\bar x}^i = \{s \in A^{<\omega} : \exists \bar z \in X_i (\bar z \restriction \beta = \bar x \wedge s \subseteq z(\beta))\}.$$ 
		
		\begin{claim}
			For every $i \in \omega$, where $\beta \in a_i$, $T^{i+1}_{\bar x} \leq_{\beta, i} T^i_{\bar x}$. In particular, $\bigcap_{i \in \omega} T^i_{\bar x} \in \mathbb{Q}_\beta$.
		\end{claim}
	
	\begin{proof}
	If $\alpha_{i+1} \leq \beta$, then $T^{i+1}_{\bar x} = T^i_{\bar x} = A^{<\omega}$. Else consider a $\mathbb{P}_{\alpha_{i+2}}$-name for $(T^{i+1}_{\bar y}, T^i_{\bar y}) \in \mathcal{T}^2$, where $\bar y = \bar x_G \restriction (C\cap \beta)$. Such a name exists in $M_{i+2}$ and $\beta \in a_i \subseteq M_{i+2}$. Thus $\leq_{\beta, i}\in M_{i+2}$ and by Lemma~\ref{lem:intermediategoodanalytic}, we have that for every $\bar y \in [\bar r_{i+2}] \restriction (C\cap \beta)$, $(T^{i+1}_{\bar y}, T^i_{\bar y}) \in \leq_{\beta,i}$, thus also for $\bar y = \bar x$. The rest follows from the fact that the statement, that for any fusion sequence in $\mathbb{Q}_\beta$, its intersection is in $\mathbb{Q}_\beta$, is $\mathbf{\Pi}^1_2$ and thus absolute. 
	\end{proof}
	
\begin{claim}\label{claim:closedset}
For every $\gamma$, $\bigcap_{i \in \omega} (X_i \restriction \gamma) = (\bigcap_{i \in \omega} X_i) \restriction \gamma$.
\end{claim}	

\begin{proof}
That $\bigcap_{i \in \omega} (X_i \restriction \gamma) \supseteq (\bigcap_{i \in \omega} X_i) \restriction \gamma$ is obvious. Let us show by induction on $\delta \in C$, that for any $\delta' \in C \cap \delta$, $$\bigcap_{i \in \omega} (X_i \restriction \delta') \subseteq \left(\bigcap_{i \in \omega} (X_i \restriction \delta)\right) \restriction \delta'.$$ The base case $\delta = \min C$ is clear. For the limit case, let $\delta' \in C \cap \delta$ be given and let $(\delta_n)_{n \in  \omega} $ be increasing cofinal in $(C \cap \delta) \setminus \delta'$. Whenever $\bar y \in \bigcap_{i \in \omega} (X_i \restriction \delta')$, by the inductive hypothesis, there is $\bar y_0 \in \bigcap_{i \in \omega} (X_i \restriction \delta_0)$ extending $\bar y$. In particular, there is $\bar z_0 \in X_0 \restriction \delta$ extending $\bar y_0$. Next, there is $\bar y_1 \in \bigcap_{i \in \omega} (X_i \restriction \delta_1)$ extending $\bar y_0$ and $\bar z_1 \in X_1 \restriction \delta$ extending $\bar y_1$. Continuing like this, we find a sequence $\langle\bar z_n : n \in \omega \rangle$ that converges to $\bar z \in (A^\omega)^{C \cap \delta}$. Since $\langle\bar z_n : n \geq m \rangle$ is contained within the closed set $X_m \restriction \delta$ for each $m \in \omega$, $\bar z \in \bigcap_{i \in \omega} (X_i \restriction \delta)$. Since $\bar z \restriction \delta' = \bar y$, this proves the limit case. Now assume $\delta = \xi +1$. Let $\delta' < \delta$ be given and let $\bar y \in \bigcap_{i \in \omega} (X_i \restriction \delta')$. Then there is $\bar z \in \bigcap_{i \in \omega} (X_i \restriction \xi)$ extending $\bar y$ by the inductive hypothesis. Since $\bigcap_{i \in \omega} T^i_{\bar z} \in \mathbb{Q}_\delta$, there is $u \in [\bigcap_{i \in \omega}T^i_{\bar z}]$ and $\bar z {}^\frown u \in \bigcap_{i \in \omega} (X_i \restriction \delta)$. To finish the proof apply the induction step one more time to $\delta = \sup \{ \xi +1 : \xi \in C\}$ and $\delta' = \gamma$. 
\end{proof}

Claim~\ref{claim:closedset} shows that (4) holds as $[\bar q] \restriction \beta = (\bigcap_{i \in \omega} X_i) \restriction \beta \times (A^\omega)^{\beta \setminus C} = \bigcap_{i \in \omega} (X_i \restriction \beta) \times (A^\omega)^{\beta \setminus C}$ and $\bigcap_{i \in \omega} (X_i \restriction \beta)$ is closed, being an intersection of closed sets.

\begin{claim}
	$T_{\bar x} = \bigcap_{i \in \omega} T^i_{\bar x}$.
\end{claim}

\begin{proof}

That $T_{\bar x} \subseteq \bigcap_{i \in \omega} T^i_{\bar x}$ is clear from the definitions. Thus let $s \in \bigcap_{i \in \omega} T^i_{\bar x}$. As $\bigcap_{i \in \omega} T^i_{\bar x} \in \mathbb{Q}_\beta$, there is $y \in [\bigcap_{i \in \omega} T^i_{\bar x}]$ with $s \subseteq y$. In particular, $\bar x {}^\frown y \in \bigcap_{i \in \omega} (X_i \restriction (\beta +1)) = (\bigcap_{i \in \omega} X_i )\restriction (\beta +1)$. So there is $\bar z \in \bigcap_{i \in \omega} X_i$ with $\bar z \restriction \beta = \bar x$ and $z(\beta) = y \supseteq s$. Thus $s \in T_{\bar x}$.\end{proof}
Now (2) follows easily. For the continuity of $\bar x \mapsto T_{\bar x}$, let $t \in A^{<\omega}$ be arbitrary and $j$ large enough so that $\vert t \vert \leq j$ and $\beta \in a_j$. Then $\{ \bar x \in [\bar q] \restriction \beta : t \notin T_{\bar x} \} = \{ \bar x \in [\bar q] \restriction \beta : t \notin T^j_{\bar x} \}$ and $\{ \bar x \in [\bar q] \restriction \beta : t \in T_{\bar x} \} = \{ \bar x \in [\bar q] \restriction \beta : t \in T^j_{\bar x} \}$ which are both open.\footnote{Here we use clause (1) in the definition of Axiom A.} Thus we have shown (3).\end{proof}

\begin{lemma}
\label{lem:gettingacondition}
Let $C \subseteq \alpha$ be countable and $X \subseteq (A^{\omega})^C$ be a closed set so that for every $\beta \in C$ and $\bar x \in X \restriction \beta$, $$\{s \in A^{<\omega} : \exists \bar z \in X (\bar z \restriction \beta = \bar x \wedge s \subseteq z(\beta))  \} \in \mathbb{Q}_\beta.$$ Let $M \ni X$ be countable elementary. Then there is a good master condition $\bar r$ over $M$ so that $[\bar r] \restriction C \subseteq X$.
\end{lemma}

\begin{proof}
	It is easy to construct $\bar q \in M$ recursively so that $\bar q \Vdash \bar x_G \restriction C \in X$. By Lemma~\ref{lem:nicemaster}, we can extend $\bar q$ to a good master condition $\bar r$ over $M$. The unique continuous function $f \colon [\bar r] \to (A^{\omega})^C$ so that for generic $G$, $f(\bar x_G) = \bar x_G \restriction C$, is so that $f(\bar x) = \bar x \restriction C$ for every $\bar x \in [\bar r]$. Since $f$ maps to $X$, $[\bar r] \restriction C \subseteq X$.  
\end{proof}

\section{The Main Lemma}

\subsection{Mutual Cohen Genericity}

Let $X$ be a Polish space and $M$ a model of set theory with $X \in M$. Recall that $x \in X$ is \emph{Cohen generic in} $X$ over $M$ if for any open dense $O \subseteq X$, such that $O \in M$, $x \in O$. 

Let $x_0, \dots, x_{n-1} \in X$. Then we say that $x_0, \dots x_{n-1}$ are ($X$-)\emph{mutually Cohen generic (mCg) over $M$} if $(y_0, \dots, y_{K-1})$ is a Cohen generic real over $M$ in the Polish space $X^K$, where $\langle y_i : i < K \rangle$ is some, equivalently any, enumeration of $\{ x_0, \dots, x_{n-1} \}$. In particular, we allow for repetition in the definition of mutual genericity. 

\begin{definition}\label{def:mcgfinite}
Let $\langle X_l : l < k \rangle \in M$ be Polish spaces. Then we say that $\bar x_0, \dots, \bar x_{n-1} \in \prod_{l<k} X_l$ are $\langle X_l : l < k \rangle$-\emph{mutually Cohen generic (mCg) over $M$}, if their components are mutually added Cohen generics, i.e.  $$(y_0^0, \dots, y_0^{K_0}, \dots,y_{k-1}^0, \dots, y_{k-1}^{K_{k-1}}) \text{ is Cohen generic in } \prod_{l<k }X_l^{K_l} \text{ over } M,$$ where $\langle y_l^{i} : i < K_l \rangle$ is some, equivalently any, enumeration of $\{ x_i(l) : i<n  \}$ for each $l< k$. 
\end{definition}


\begin{definition}
	Let $X$ be a Polish space with a fixed countable basis $\mathcal{B}$. Then we define the forcing poset $\mathbb{C}(2^\omega, X)$ consisting of functions $h \colon 2^{\leq n} \to \mathcal{B} \setminus \{\emptyset\}$ for some $n \in \omega$ such that $\forall \sigma \subseteq \tau \in 2^{\leq n} (h(\sigma) \supseteq h(\tau))$. The poset is ordered by function extension. 
\end{definition}

The poset $\mathbb{C}(2^\omega, X)$ adds generically a continuous function $\chi \colon 2^\omega \to X$, given by $\chi(x) = y$ where $\bigcap_{n \in \omega} h(x \restriction n ) = \{ y\}$ and $h = \bigcup G$ for $G$ the generic filter. This forcing will be used in this section several times to obtain ZFC results. Note for instance that if $G$ is generic over $M$, then for any $x \in 2^\omega$, $\chi(x)$ is Cohen generic in $X$ over $M$, and moreover, for any $x_0, \dots, x_{n-1} \in 2^\omega$, $\chi(x_0), \dots, \chi(x_{n-1})$ are $X$-mutually Cohen generic over $M$. Sometimes we will use $\mathbb{C}(2^\omega, X)$ to force over a countable model a continuous function from a space homeomorphic to $2^\omega$, such as $(2^\omega)^\alpha$ for $\alpha < \omega_1$.

\begin{lemma}
	\label{lem:forcingcont}
Let $M$ be a model of set theory, $K, n \in \omega$, $X_j \in M$ a Polish space for every $j < n$ and $G$ a $\prod_{j < n} \mathbb{C}(2^\omega, X_j)$-generic over $M$ yielding $\chi_j \colon 2^\omega \to X_j$ for every $j <n$. Then, whenever $\bar x$ is Cohen generic in $(2^\omega)^K$ over $M[G]$ and $u_0, \dots, u_{n-1} \in 2^\omega \cap M[\bar x]$ are pairwise distinct, $$\bar x^\frown \langle \chi_j(u_i) : i < n, j < n \rangle $$ is Cohen generic in $$(2^\omega)^K \times \prod_{i < n} X_i $$ over $M$. 
\end{lemma}

\begin{proof}
	Since $\bar x$ is generic over $M$ it suffices to show that $\langle \chi_j(u_i) : i < n, j \rangle$ is generic over $M[\bar x]$. Let $\dot O \in M$ be a $(2^{<\omega})^K$-name for a dense open subset of $\prod_{j < n} (X_j)^n$ and $\dot u_i$ a $(2^{<\omega})^K$-name for $u_i$, $i < n$, such that the trivial condition forces that the $\dot u_i$ are pairwise distinct. Then consider the set \begin{multline*}
		D := \{ (\bar h,\bar s) \in \prod_{i<n} \mathbb{C}(2^\omega, X_i) \times (2^{<\omega})^K : \exists t_0, \dots, t_{n-1}\in 2^{<\omega} \\ ( \forall i < n(\bar s \Vdash t_i \subseteq \dot u_i) \wedge \bar s \Vdash \prod_{i,j<n} h_{j}(t_i) \subseteq \dot O) \}.
	\end{multline*}
	
	We claim that this set is dense in $\prod_{i<n} \mathbb{C}(2^\omega, X_i) \times (2^{<\omega})^K$ which finishes the proof. Namely let $(\bar h, \bar s)$ be arbitrary, wlog $\dom h_{j} = 2^{\leq n_0}$ for every $j < n$. Then we can extend $\bar s$ to $\bar s'$ so that there are incompatible $t_i$, with $\vert t_i \vert \geq n_0$, so that $\bar s' \Vdash t_i \subseteq \dot u_i$ and there are $U_{i,j} \subseteq h_j(t_i \restriction n_0)$ basic open subsets of $X_j$ in $M$ for every $i < n$ and $j<n$, so that $\bar s' \Vdash \prod_{i,j < n} U_{i,j} \subseteq \dot O$. Then we can extend $\bar h$ to $\bar h'$ so that $h'_j(t_i) = U_{i,j}$ for every $i,j < n$. We see that $(\bar h', \bar s') \in D$.
\end{proof}

\subsection{Finite products}

This subsection can be skipped entirely if one is only interested in the results for the countable support iteration. The lemma that we will prove below is relevant to finite products instead (see Theorem~\ref{thm:finiteprod}). It is very similar to Main Lemma~\ref{lem:mainlemmainf} in the next subsection, but the proofs are completely different. Main Lemma~\ref{thm:mainlemma} is based on a forcing-theoretic proof of the Halpern-Läuchli theorem that is commonly attributed to L. Harrington (see e.g. \cite[Lemma~4.2.4]{Golshani2015} as a reference). On the other hand, Main Lemma~\ref{lem:mainlemmainf} uses an inductive argument. 

\begin{mainlemma}
	\label{thm:mainlemma}
	Let $k \in \omega$ and $E \subseteq [(2^\omega)^k]^{<\omega}$ an analytic hypergraph on $(2^\omega)^k$. Then there is a countable model $M$ so that either 

	\begin{enumerate}
		\item for any $n \in \omega$ and $\bar x_0,\dots, \bar x_{n-1} \in (2^\omega)^k$ that are $\langle 2^\omega : l < k\rangle$-mCg over $M$, $$\{\bar x_0,\dots, \bar x_{n-1}\} \text{ is } E \text{-independent}$$
	\end{enumerate}
	
	or for some $N \in \omega$,
	\begin{enumerate}
		\setcounter{enumi}{1}
		\item there are $\phi_0, \dots, \phi_{N-1} \colon (2^\omega)^k \to (2^\omega)^k$ continuous, $\bar s \in \bigotimes_{l<k} 2^{<\omega}$ so that for any $n \in \omega$ and $\bar x_0,\dots, \bar x_{n-1} \in (2^\omega)^k \cap [\bar s]$, that are $\langle 2^\omega : l < k\rangle$-mCg over $M$,  $$\{\phi_j(\bar x_i) : j < N, i<n\} \text{ is } E \text{-independent but } \{ \bar x_0\} \cup \{ \phi_j(\bar x_0) : j < N \} \in E.$$
			\end{enumerate}
	\end{mainlemma}

\begin{remark}
	Note that $N = 0$ is possible in the second option.  For example whenever $[(2^\omega)^k]^1 \subseteq E$, then $\emptyset$ is the only $E$-independent set. In this case the last line simplifies to ``$\{ \bar x_0 \} \in E$".
\end{remark}

\begin{proof}
	Let $\kappa = \beth_{2k -1}(\aleph_0)^+$. Recall that by Erdős-Rado (see \cite[Thm 9.6]{Jech2013}), for any ${c \colon [\kappa]^{2k} \to H(\omega)}$, there is $B \in [\kappa]^{\aleph_1}$ which is monochromatic for $c$, i.e. $c \restriction [B]^{2k}$ is constant. Let $\mathbb{Q}$ be the forcing adding $\kappa$ many Cohen reals $$\langle z_{(l,\alpha)} : \alpha < \kappa\rangle \text{ in } 2^{\omega} \text{ for each } l<k$$ with finite conditions, i.e. $\mathbb{Q} = \prod_\kappa^{<\omega} (2^{<\omega})^k$. We will use the notational convention that elements of $[\kappa]^d$, for $d \in \omega$, are sequences $\bar \alpha = (\alpha_0, \dots, \alpha_{d-1})$ ordered increasingly. For any $\bar \alpha \in [\kappa]^k$ we define $\bar z_{\bar \alpha} := (z_{(0,\alpha_0)}, \dots, z_{(k-1,\alpha_{k-1})}) \in (2^\omega)^k$.

	Let $\dot{\mathcal{A}}$ be a $\mathbb{Q}$-name for a maximal $E$-independent subset of $\{ \bar z_{\bar \alpha} : \bar \alpha \in [\kappa]^k \}$, reinterpreting $E$ in the extension by $\mathbb{Q}$. For any $\bar \alpha \in [\kappa]^k$, we fix $p_{\bar\alpha} \in \mathbb{Q}$ so that either \begin{equation}
	p_{\bar \alpha} = \mathbbm{1} \wedge p_{\bar \alpha} \Vdash \bar z_{\bar \alpha} \in \dot{\mathcal{A}}                                                                                                                
	\tag{1}                                                                                       \end{equation}
	or 
	\begin{equation}
	p_{\bar\alpha} \Vdash \bar z_{\bar \alpha} \not\in \dot{\mathcal{A}}.\tag{2}
	\end{equation}
	In case (2) we additionally fix $N_{\bar\alpha} < \omega$ and $(\bar \beta^i)_{i < N_{\bar\alpha}} = (\bar \beta^i(\bar \alpha))_{i < N_{\bar\alpha}}$, and we assume that $$p_{\bar\alpha} \Vdash \{ \bar z_{\bar \beta^i} : i < N_{\bar\alpha} \} \subseteq \dot{\mathcal{A}} \wedge \{ \bar z_{\bar \alpha} \} \cup \{ \bar z_{\bar \beta^i} : i < N_{\bar\alpha} \} \in E .$$ We also define $H_l(\bar \alpha) = \{ \beta^i_l : i < N_{\bar\alpha} \} \cup \{ \alpha_l \} \in [\kappa]^{<\omega}$ for each $l < k$. 
	
	Now for $\bar \alpha \in [\kappa]^{2k}$ we collect the following information: 
	
	\begin{enumerate}[label=(\roman*)]
		\item whether $p_{\bar \alpha \restriction k } = p_{\alpha_0,\dots, \alpha_{k-1}}\Vdash \bar z_{\bar \alpha \restriction k} \in \dot{\mathcal{A}}$ or not,
		\item $\bar s = (p_{\bar \alpha \restriction k }(0,\alpha_0), \dots,  p_{\bar \alpha \restriction k }(k-1,\alpha_{k-1})) \in (2^{<\omega})^k$,
		\item the relative position of the $p_{\bar \gamma}$ for $\bar \gamma \in \Gamma := \prod_{l<k} \{\alpha_{2l}, \alpha_{2l+1} \}$ to each other. More precisely, consider $ \bigcup_{\bar \gamma \in\Gamma  } \dom p_{\bar \gamma} = \{0\} \times d_0 \cup \dots \cup \{k-1\} \times d_{k-1}$ where $d_0,\dots, d_{k-1} \subseteq \kappa$. Let $M_l = \vert d_l \vert$ for $l<k$ and for each $\bar j \in \prod_{l<k} \{2l, 2l+1 \}$, collect a function $r_{\bar j}$ with $\dom r_{\bar j} \subseteq \{0\} \times M_0 \cup \dots \cup \{k-1\} \times M_{k-1}$ that is a copy of $p_{\bar \gamma}$, where $\bar \gamma = (\alpha_{j_0}, \dots, \alpha_{j_{k-1}})$, $\bar j =  (j_0, \dots, j_{k-1})$. Namely, $r_{\bar j}(l,m) = p_{\bar \gamma}(l,\beta)$, whenever $\beta$ is the $m$'th element of $d_l$. 
	\end{enumerate}
	
	In case $p_{\bar \alpha \restriction k }\Vdash \bar z_{\bar \alpha \restriction k} \notin \dot{\mathcal{A}}$ we additionally remember  
	
	\begin{enumerate}[label=(\roman*)]
		\setcounter{enumi}{3}
		
		\item $N = N_{\bar \alpha \restriction k}$, 
		\item $N_l = \vert H_l(\bar \alpha \restriction k) \vert$, for each $l < k$, 
		\item $\bar b^i \in \prod_{l <k} N_l$ so that $\beta^i_l$ is the $b^i_l$'th element of $H_l(\bar \alpha \restriction k)$, for each $i<N$, 
		\item $\bar a \in \prod_{l <k} N_l$ so that $\alpha_l$ is the $a_l$'th member of $H_l(\bar \alpha \restriction k)$, 
		\item the partial function $r$ with domain a subset of $\bigcup_{l <k} \{l\} \times N_l$, so that $r(l,m)= t \in 2^{<\omega}$ iff $p_{\bar \alpha \restriction k}(l,\beta) = t$ where $\beta$ is the $m$'th element of $H_l(\bar \alpha \restriction k)$. 
	\end{enumerate}

	And finally we also remember 
	
	\begin{enumerate}[label=(\roman*)]
		\setcounter{enumi}{8}
		
		\item for each pair $\bar\gamma, \bar\delta \in \prod_{l<k} \{\alpha_{2l}, \alpha_{2l +1} \}$, where $\bar\gamma = (\alpha_{j_l})_{l < k}$ and $\bar\delta = (\alpha_{j'_l})_{l < k}$, finite partial injections $e_{l, \bar j, \bar j'} \colon N_l \to N_l$ so that $e_{l, \bar j, \bar j'}(m) = m'$ iff the $m$'th element of $H_l(\bar\gamma)$ equals the $m'$'th element of $H_l(\bar\delta)$. 
		\end{enumerate}
	
	This information is finite and defines a coloring $c \colon [\kappa]^{2k} \to H(\omega)$. Let $B \in [\kappa]^{\omega_1}$ be monochromatic for $c$. Let $M \preccurlyeq H(\theta)$ be countable for $\theta$ large enough so that $\kappa, c, B, \langle p_{\bar \alpha} : \bar \alpha \in [\kappa]^k \rangle, E, \dot{\mathcal{A}} \in M$.	
	
	\begin{claim}
		If for every $\bar \alpha \in [B]^k$, $p_{\bar{\alpha}} \Vdash \bar z_{\bar \alpha} \in \dot{\mathcal{A}}$, then (1) of the main lemma holds true.
	\end{claim}

	\begin{proof}
		Let $\bar x_0, \dots, \bar x_{n-1}$ be arbitrary mCg over $M$. Say $\{x_i(l) : i < n\}$ is enumerated by $\langle y^i_l : i < K_l \rangle$ for every $l<k$. Now find $$\alpha^0_0 < \dots < \alpha^{K_0 -1}_0 < \dots < \alpha^{0}_{k-1} < \dots < \alpha^{K_{k-1} -1}_{k-1}$$ in $M \cap B$. Then there is a $\mathbb{Q}$-generic $G$ over $M$ so that for any $\bar j \in \prod_{l < k} K_l$,  $$\bar z_{\bar \beta}[G] = (y^{j_0}_0, \dots, y^{j_{k-1}}_{k-1}),$$ where $\bar \beta = (\alpha^{j_0}_0, \dots, \alpha^{j_{k-1}}_{k-1})$. In particular, for each $i<n$, there is $\bar \beta_i \in [B \cap M]^k$ so that $\bar z_{\bar \beta_i}[G]= \bar x_{i}$. Since $p_{\bar \beta_i} =\mathbbm{1} \in G$ for every $\bar \beta_i$ we have that $$M[G] \models \bar x_i \in \dot{\mathcal{A}}[G]$$ for every $i< n$ and in particular $$M[G] \models \{ \bar x_i : i < n \} \text{ is } E \text{-independent}.$$ By absoluteness $\{ \bar x_i : i < n \}$ is indeed $E$-independent.
		\end{proof}

	 	Assume from now on that $p_{\bar \alpha}\Vdash \bar z_{\bar \alpha} \notin \dot{\mathcal{A}}$ for every $\bar \alpha \in [B]^k$. Then we may fix $\bar s$, $N$, $(N_l)_{l<k}$, $\bar b^i$ for $i<N$, $\bar a$, $r$ and $e_{l,\bar j, \bar j'}$ for all $l<k$ and $\bar j,\bar j' \in \prod_{l'<k} \{2l', 2l'+1 \}$ corresponding to the coloring on $[B]^{2k}$. 
	 
	 \begin{claim}
	 	 For any $\bar \alpha \in [B]^{2k}$ and $\bar\gamma, \bar\delta \in \prod_{l<k} \{\alpha_{2l}, \alpha_{2l +1} \}$, $$p_{\bar\gamma} \restriction (\dom p_{\bar\gamma} \cap \dom p_{\bar\delta}) = p_{\bar\delta} \restriction (\dom p_{\bar\gamma} \cap \dom p_{\bar\delta}).$$
	 \end{claim}
 
	 \begin{proof}
	 	 Suppose not. By homogeneity we find a counterexample $\bar \alpha$, $\bar\gamma$, $\bar\delta$ where $B \cap (\alpha_{2l'},\alpha_{2l' +1})$ is non-empty for every $l' <k$. So let $(l,\beta) \in \dom p_{\bar\gamma} \cap \dom p_{\bar\delta}$ such that $p_{\bar\gamma}(l,\beta) = u \neq v = p_{\bar\delta}(l,\beta)$. Let $\bar \rho \in [B]^k$ be such that for every $l' < k$, $$\begin{cases}
	 	 \rho_{l'} \in (\gamma_{l'}, \delta_{l'} ) & \text{if } \gamma_{l'} < \delta_{l'} \\
	 	 \rho_{l'} \in (\delta_{l'}, \gamma_{l'} ) & \text{if }  \delta_{l'} < \gamma_{l'} \\
	 	 \rho_{l'} = \gamma_{l'} & \text{if } \gamma_{l'} = \delta_{l'}.
	 	 \end{cases}$$
	 	 Now note that $\bar\rho$'s relative position to $\bar\gamma$ is the same as that of $\bar\delta$ to $\bar\gamma$. More precisely, let $\bar j, \bar j' \in \prod_{l'<k} \{2l', 2l'+1 \}$ so that $\bar \gamma = (\alpha_{j_0}, \dots, \alpha_{j_{k-1}})$, $\bar \delta = (\alpha_{j'_0}, \dots, \alpha_{j'_{k-1}})$. Then there is $\bar \beta \in [B]^{2k}$ so that $\bar \gamma = (\beta_{j_0}, \dots, \beta_{j_{k-1}})$ and $\bar \rho = (\beta_{j'_0}, \dots, \beta_{j'_{k-1}})$. Thus by homogeneity of $[B]^{2k}$ via $c$, $p_{\bar\rho}(l,\beta) = v$. Similarly $\bar\delta$ is in the same position relative to $\bar\rho$ as to $\bar\gamma$. Thus also $p_{\bar\rho}(l,\beta) = u$ and we find that $v = u$ -- we get a contradiction.
	 \end{proof}

	\begin{claim}
		For any $l<k$ and $\bar j,\bar j' \in \prod_{l'<k} \{2l', 2l'+1 \}$, $e_{l,\bar j, \bar j'}(m) = m$ for every $m \in \dom e_{l,\bar j, \bar j'} $. 
	\end{claim}
	
	\begin{proof}
		Let $\alpha_0 < \dots < \alpha_{2k} \in B$ so that $(\alpha_{2l'},\alpha_{2l'+1}) \cap B \neq \emptyset$ for every $l' <k$. Consider $\bar \gamma = (\alpha_{j_{l'}})_{l' < k}$, $\bar \delta = (\alpha_{j'_{l'}})_{l' < k}$ and again we find $\bar \rho \in [B]^k$ so that $\rho_{l'}$ is between (possibly equal to) $\alpha_{j_{l'}}$ and $\alpha_{j'_{l'}}$. If $e_{l,\bar j, \bar j'}(m) = m'$, then if $\beta$ is the $m$'th element of $H_l(\bar \gamma)$, then $\beta$ is $m'$'th element of $H_l(\bar \delta)$ aswell as of $H_l(\bar\rho)$. But also $\beta$ is the $m$'th element of $H_l(\bar\rho)$, thus $m =m'$.  
\end{proof}
	
	Note that by the above claim $e_{l,\bar j, \bar j'} = (e_{l,\bar j', \bar j})^{-1} = e_{l,\bar j', \bar j}$ and the essential information given by $e_{l,\bar j, \bar j'}$ is it's domain.
	
	Next let us introduce some notation. Whenever $x,y \in 2^\omega$, we write $x < y$ to say that $x$ is lexicographically below $y$, i.e. $x(n) < y(n)$, where $n = \min\{m \in \omega : x(m) \neq x(m) \}$. For any $g \in \{-1,0,1 \}^{k}$ we naturally define a relation $\tilde R_g$ between $k$-length sequences $\bar \nu$ and $\bar \mu$, either of elements of $2^\omega$, or of ordinals $<\kappa$, as follows: $$\bar \nu \tilde R_g \bar \mu \leftrightarrow \forall l<k \begin{cases}
	\nu_l < \mu_l &\text{ if } g(l)=-1\\
	\nu_l = \mu_l &\text{ if } g(l)=0\\
	\nu_l > \mu_l &\text{ if } g(l)=1.
	\end{cases}
	$$
	
	\noindent Further we write $\bar \nu R_g \bar \mu$ iff $\bar \nu \tilde R_g \bar \mu$ or $\bar \mu \tilde R_g \bar \nu$. Enumerate $\{R_g : g \in \{-1,0,1 \}^{k} \}$ without repetition as $\langle R_i : i < K \rangle$ (it is easy to see that $K = \frac{3^k +1}{2}$). Note that for any $\bar \nu, \bar \mu$ there is a unique $i <K$ so that $\bar \nu R_i \bar \mu$. Now for each $l < k$ and $i < K$, we let $$I_{l,i} := \dom e_{l,\bar j, \bar j'} \subseteq N_l,$$ 
	where $\bar j R_i \bar j'$. By homogeneity of $[B]^{2k}$ and the observation that $e_{l,\bar j, \bar j'} = e_{l,\bar j', \bar j}$, we see that $I_{l,i}$ does not depend on the particular choice of $\bar j, \bar j'$, such that $\bar j R_i \bar j'$.

For each $l< k$ and $m < N_l$, we define a relation $E_{l,m}$ on $(2^\omega)^k$ as follows: 
	
	$$ \bar x E_{l,m} \bar y \leftrightarrow m \in I_{l,i} \text{ where } i \text{ is such that } \bar x R_i \bar y.$$
	
	\begin{claim}
		$E_{l,m}$ is an equivalence relation. 
	\end{claim}
	
	\begin{proof}
		The reflexivity and symmetry of $E_{l,m}$ is obvious. Assume that $\bar x_0 E_{l,m} \bar x_1$ and $\bar x_1 E_{l,m} \bar x_2$, and say $\bar x_0 R_{i_0} \bar x_1$, $\bar x_1 R_{i_1} \bar x_2$ and $\bar x_0 R_{i_2} \bar x_2$. Find $\bar \gamma^0, \bar \gamma^1, \bar \gamma^2 \in [B]^k$ so that $$\{\gamma^i_0 : i <3 \} < \dots < \{\gamma^i_{k-1} : i <3 \} $$ and $$\bar \gamma^0 R_{i_0} \bar \gamma^1, \bar \gamma^1 R_{i_1} \bar \gamma^2, \bar \gamma^0 R_{i_2} \bar \gamma^2.$$ 
		
		If $\beta$ is the $m$'th element of $H_l(\bar \gamma^0)$, then $\beta$ is also the $m$'th element of $H_l(\bar \gamma^1)$, since we can find an appropriate $\bar \alpha \in [B]^{2k}$ and $\bar j$, $\bar j'$ so that $\bar \gamma^0 = (\alpha_{j_l})_{l<k}$ and $\bar \gamma^1 = (\alpha_{j'_l})_{l<k}$, $\bar j R_{i_0} \bar j'$ and we have that $m \in I_{l,i_0}$. Similarly $\beta$ is the $m$'th element of  $H_l(\bar \gamma^2)$. 
		
		But now we find again $\bar \alpha \in [B]^{2k}$ and $\bar j$, $\bar j'$ so that $\bar \gamma^0 = (\alpha_{j_l})_{l<k}$ and $\bar \gamma^2 = (\alpha_{j'_l})_{l<k}$. Thus $m \in I_{l,i_2}$, as $e_{l,\bar j, \bar j'}(m) = m$ and $\bar x_0 E_{l,m} \bar x_2$. 
		\end{proof}
	
	\begin{claim}
		$E_{l,m}$ is smooth as witnessed by a continuous function, i.e. there is a continuous map $\varphi_{l,m} \colon (2^{\omega})^k \to 2^{\omega}$ so that $\bar x E_{l,m} \bar y$ iff $\varphi_{l,m}(\bar x) = \varphi_{l,m}(\bar y)$.  
	\end{claim}
	
	\begin{proof}
		We will check the following: 
		
		\begin{enumerate}[label=(\alph*)]
			\item For every open $O \subseteq (2^\omega)^k$, the $E_{l,m}$ saturation of $O$ is Borel,
			\item every $E_{l,m}$ equivalence class is $G_\delta$. 
		\end{enumerate}
		
		By a theorem of Srivastava (\cite[Thm 4.1.]{Srivastava1979}), (a) and (b) imply that $E_{l,m}$ is smooth, i.e. we can find $\varphi_{l,m}$ Borel. 
		
		\begin{enumerate}[label=(\alph*)]
			\item The $E_{l,m}$ saturation of $O$ is the set $\{\bar x : \exists \bar y \in O (\bar x E_{l,m} \bar y) \}$. It suffices to check for each $g \in \{-1,0,1\}^k$ that the set $X = \{\bar x : \exists \bar y \in O (\bar x \tilde R_g \bar y) \}$ is Borel. Let $\mathcal{S} = \{\bar \sigma \in (2^{<\omega})^k : [\sigma_0] \times \dots \times [\sigma_{k-1}] \subseteq O \}$. Consider $$\varphi(\bar x) :\leftrightarrow \exists \bar \sigma \in \mathcal{S} \forall l' < k \begin{cases} x_{l'} <_{\lex} {\sigma_{l'}}^\frown 0^\omega &\text{ if } g(l') = -1 \\
			x_{l'} \in [\sigma_{l'}] &\text{ if } g(l') = 0 \\
			{\sigma_{l'}}^\frown 1^\omega <_{\lex} x_{l'} &\text{ if } g(l') = 1\end{cases}.$$
			If $\varphi(\bar x)$ holds true then let $\bar \sigma$ witness this. We then see that there is $\bar y \in [\sigma_0] \times \dots \times [\sigma_{k-1}]$ with $\bar x \tilde R_g \bar y$. On the other hand, if $\bar y \in O$ is such that $\bar x R_g \bar y$, then we find $\bar \sigma \in \mathcal{S}$ defining a neighborhood of $\bar y$ witnessing $\varphi(\bar x)$. Thus $X$ is defined by $\varphi$ and is thus Borel. 
			\item Since finite unions of $G_\delta$'s are $G_\delta$ it suffices to check that $\{ \bar x : \bar x \tilde R_g \bar y \}$ is $G_\delta$ for every $\bar y$ and $g \in \{-1,0,1\}^k$. But this is obvious from the definition. 
		\end{enumerate}
		
		Now note that given $\varphi_{l,m}$ Borel, we can find perfect $X_0,\dots X_{k-1} \subseteq 2^{\omega}$ so that $\varphi_{l,m}$ is continuous on $X_0 \times \dots \times X_{k-1}$ ($\varphi_{l,m}$ is continuous on a dense $G_\delta$). But there is a $<_{\lex}$ preserving homeomorphism from $X_l$ to $2^{\omega}$ for each $l<k$ so we may simply assume $X_l = 2^\omega$. 
	\end{proof}

	Fix such $\varphi_{l,m}$ for every $l<k$, $m<N_l$, so that $\varphi_{l,a_l}(\bar x) = x_l$ (note that $\bar x E_{l,a_l} \bar y$ iff $x_l = y_l$). Now let $M_0$ be countable elementary, containing all relevant information and such that $\varphi_{l,m} \in M_0$ for every $l < k$, $m < N_l$. Let $\chi_{l,m} \colon 2^\omega \to [r(l,m)]$ for $l < k$ and $m \neq a_l$ be generic continuous functions over $M_0$, i.e. the sequence $(\chi_{l,m})_{l < k, m \in N_l \setminus \{a_l\} }$ is $\prod_{l < k, m \in N_l \setminus \{a_l\}} \mathbb{C}(2^\omega, [r(l,m)])$ generic over $M_0$. Let us denote with $M$ the generic extension of $M_0$. Also let $\chi_{l,m}$ for $m = a_l$ be the identity. Finally we set $$\phi_i(\bar x) = ((\chi_{l,b^i_l}\circ \varphi_{l,b^i_l})(\bar x))_{l < k}$$ for each $i < N$. 
	
	\begin{claim}
		(2) of the main lemma holds true with $M$, $\bar s$ and $\phi_i$, $i <N$, that we just defined. 
	\end{claim}
	
	\begin{proof}
		Let $\bar x_0. \dots, \bar x_{n-1} \in [\bar s]$ be $\langle 2^\omega : l < k\rangle$-mCg over $M$. Let us write $\{ \bar x_i(l) : i <n \} = \{ y_l^i : i < K_l \}$ for every $l < k$, where $y_l^0 <_{\lex} \dots <_{\lex} y_l^{K_l -1} $. Now find $$\alpha^0_0 < \dots < \alpha_0^{K_0 -1} < \dots < \alpha_{k-1}^0 < \dots < \alpha_{k-1}^{K_{k-1}-1}$$ in $B \cap M$. For every $\bar j \in \prod_{l < k} K_l$, define $\bar y_{\bar j} := (y_0^{j(0)}, \dots, y_{k-1}^{j(k-1)})$ and $\bar \alpha_{\bar j} := (\alpha_0^{j(0)}, \dots, \alpha_{k-1}^{j(k-1)})$. Then, for each $i < n$, we have $\bar j_i \in \prod_{l < k} K_l$ so that $\bar x_i = \bar y_{\bar j_i}$. For each $i < n$ define the function $g_i \colon \bigcup_{l < k} \{ l\} \times H_l(\bar \alpha_{\bar j_i}) \to 2^\omega$, setting $$g_i(l,\beta) = \chi_{l,m}(\varphi_{l,m}(\bar x_{i})), $$ whenever $\beta$ is the $m$'th element of $H_l(\bar \alpha_{\bar j_i})$.
		
		Now we have that the $g_i$ agree on their common domain. Namely let $i_0, i_1 < n$ and $(l,\beta) \in \dom g_{i_0} \cap \dom g_{i_1}$. Then if we set $i$ to be so that $\bar x_{i_0} R_i \bar x_{i_1}$, we have that $m \in I_{l,i}$, where $\beta$ is the $m$'th element of $H_{l}(\bar \alpha_{\bar j_{i_0}})$ and of $H_{l}(\bar \alpha_{\bar j_{i_1}})$. In particular $\bar x_{i_0} E_{l,m} \bar x_{i_1}$ and $\varphi_{l,m}(\bar x_{i_0}) = \varphi_{l,m}(\bar x_{i_1})$ and thus $$g_{i_0}(l,\beta) = \chi_{l,m}(\varphi_{l,m}(\bar x_{i_0})) = \chi_{l,m}(\varphi_{l,m}(\bar x_{i_1})) = g_{i_1}(l,\beta).$$
		
		Let $g := \bigcup_{i < n} g_i$. Then we see by Lemma~\ref{lem:forcingcont}, that $g$ is Cohen generic in $\prod_{(l,\beta) \in \dom g} 2^\omega$ over $M_0$. Namely consider $K = \sum_{l < k} K_l$ and $(y^0_0, \dots, y_{k-1}^{K_{k-1}-1})$ as a $(2^{<\omega})^K$-generic over $M$. Then, if $\langle u_i : i < n'\rangle$ enumerates $\{ \varphi_{l,m}(\bar x_i) : i < n, l<k, m < N_l \}$, we have that every value of $g$ is contained in $\{ \chi_{l,m}(u_i) : i < n', l <k, m<N_l \}$. Also note that by construction for every $i < n$, $p_{\bar \alpha_{\bar j_i}} \restriction \dom g$ is in the generic filter defined by $g$. Since $\{ p_{\bar \alpha_{\bar j_i}} : i < n \}$ is centered we can extend the generic filter of $g$ to a $\mathbb{Q}$-generic $G$ over $M_0$ so that $p_{\bar \alpha_{j_i}} \in G$ for every $i < n$. 
		
	   Now we have that $$ \bar z_{\bar \alpha_{\bar j_i}}[G] = \bar x_i \text{ and } \bar z_{\bar \beta^{j}(\bar \alpha_{\bar j_i})}[G] = \phi_j(\bar x_{i})$$ for every $i < n$ and $j < N$. Thus we get that $$M_0[G] \models \bigcup_{i < n} \{ \phi_j(\bar x_i) : j < N \} \subseteq \dot{\mathcal{A}}[G] \wedge \{ \bar x_0 \} \cup \{\phi_j(\bar x_0) : j < N \} \in E.$$ Again, by absoluteness, we get the required result.
\end{proof}
\vspace{-10pt}
\end{proof}

\subsection{Infinite products}

\begin{definition}\label{def:mcginfinite}
Let $\langle X_i : i < \alpha \rangle \in M$ be Polish spaces indexed by a countable ordinal $\alpha$. Then we say that $\bar x_0, \dots, \bar x_{n-1} \in \prod_{i<\alpha} X_i$ are $\langle X_i : i < \alpha \rangle$-\emph{mutually Cohen generic (mCg) over} $M$ if there are $\xi_0 = 0 <  \dots< \xi_k = \alpha$ for some $k \in \omega$ so that $$\bar x_0, \dots, \bar x_{n-1} \text{ are } \langle Y_l : l < k\rangle\text{-mutually Cohen generic over } M,$$ where $Y_l = \prod_{i \in [\xi_l, \xi_{l+1})} X_i$ for every $l <k$ and we identify $\bar x_i$ with $$(\bar  x_i \restriction [\xi_0, \xi_1), \dots, \bar x_i \restriction [\xi_{k-1}, \xi_k)) \in \prod_{l <k} Y_l ,$$ for every $i < n$.
\end{definition}
   
The identification of a sequence $\bar x$ with $(\bar x \restriction [\xi_0, \xi_1), \dots, \bar x \restriction [\xi_{k-1}, \xi_k))$ for a given $\langle \xi_l : l < k\rangle$ will be implicitely made throughout the rest of the paper in order to reduce the notational load. 

Note that whenever $\bar x_0, \dots, \bar x_{n-1}$ are $\langle X_i : i < \alpha\rangle$-mCg over $M$ and $\beta \leq \alpha$, then ${\bar x_0 \restriction \beta}, \dots, {\bar x_{n-1} \restriction \beta}$ are $\langle X_i : i < \beta\rangle$-mCg over $M$. Also note that Definition~\ref{def:mcginfinite} agrees with the notion of mCg for finite $\alpha$.

\begin{definition}
	We say that $\bar x_0, \dots, \bar x_{n-1} \in \prod_{i<\alpha} X_i$ are \emph{strongly $\langle X_i : i < \alpha \rangle$-mCg} over $M$ if they are $\langle X_i : i < \alpha \rangle$-mCg over $M$ and for any $i,j < n$, if $\xi = \min\{ \beta < \alpha : x_i(\beta) \neq x_j(\beta) \}$, then $x_i(\beta) \neq x_j(\beta)$ for all $\beta \geq \xi$.
\end{definition}

\begin{mainlemma}
	\label{lem:mainlemmainf}
	Let $\alpha < \omega_1$ and $E \subseteq [(2^\omega)^\alpha]^{<\omega}$ be an analytic hypergraph. Then there is a countable model $M$, $\alpha +1 \subseteq M$, so that either
	
	\begin{enumerate}
		\item for any $n \in \omega$ and $\bar x_0,\dots, \bar x_{n-1} \in (2^\omega)^\alpha$ that are strongly $\langle 2^\omega : i < \alpha\rangle$-mCg over $M$, $$\{\bar x_0,\dots, \bar x_{n-1}\} \text{ is } E \text{-independent}$$
	\end{enumerate}
	
	or for some $N\in \omega$,
	\begin{enumerate}
		\setcounter{enumi}{1}
		\item there are $\phi_0, \dots, \phi_{N-1} \colon (2^\omega)^\alpha \to (2^\omega)^\alpha$ continuous, $\bar s \in \bigotimes_{i<\alpha} 2^{<\omega}$ so that for any $n \in \omega$ and $\bar x_0,\dots, \bar x_{n-1} \in (2^\omega)^\alpha \cap [\bar s]$ that are strongly mCg over $M$,  $$\{\phi_j(\bar x_i) : j < N, i<n\} \text{ is } E \text{-independent but } \{ \bar x_0\} \cup \{ \phi_j(\bar x_0) : j < N \} \in E.$$
		\end{enumerate}
	
\end{mainlemma}

\begin{proof}

We are going to show something slightly stronger. Let $R$ be an analytic hypergraph on $(2^\omega)^\alpha \times \omega$, $M$ a countable model with $R \in M, \alpha +1 \subseteq M$ and $k \in \omega$. Then consider the following two statements.

\begin{enumerate}[label={$(\arabic*)_{R,M,k}$:},align=left]

    \item For any pairwise distinct $\bar x_0, \dots, \bar x_{n-1}$ that are strongly $\langle 2^\omega : i < \alpha\rangle$-mCg over $M$, and any $k_0, \dots, k_{n-1} < k$, $$\{(\bar x_0{},k_0),\dots, (\bar x_{n-1}, k_{n-1})\} \text{ is } R \text{-independent}.$$

\item There is $N \in \omega$, there are $\phi_0, \dots, \phi_{N-1} \colon (2^\omega)^\alpha \to (2^\omega)^\alpha$ continuous, such that for every $\bar x \in (2^\omega)^\alpha$ and $j_0 < j_1 < N$, $\phi_{j_0}(\bar x) \neq \phi_{j_1}(\bar x)$ and $\phi_{j_0}(\bar x) \neq \bar x$, there are $k_0, \dots, k_{N-1} \leq k$ and $\bar s \in \bigotimes_{i<\alpha} 2^{<\omega}$, so that for any pairwise distinct $\bar x_0,\dots, \bar x_{n-1} \in (2^\omega)^\alpha \cap [\bar s]$ that are strongly $\langle 2^\omega : i < \alpha\rangle$-mCg over $M$,  $$\{(\phi_j(\bar x_i),k_{j}) : j < N, i<n\} \text{ is } R \text{-independent, but}$$ $$\{ (\bar x_0, k)\} \cup \{(\phi_j(\bar x_0), k_{j}) : j < N \} \in R.$$ In fact, if $k > 0$, $$\{ (\bar x_i, {k-1}) : i < n \} \cup \{(\phi_j(\bar x_i), k_{j}) : j < N, i<n\} \text{ is } R \text{-independent.}$$
\end{enumerate}

\noindent We are going to show by induction on $\alpha$ that for any $R, M, k$, $(1)_{R,M,k}$ implies that either $(1)_{R,M,k+1}$ or there is a countable model $M^+ \supseteq M$ so that $(2)_{R,M^+,k}$. From this we easily follow the statement of the main lemma. Namely, whenever $E$ is a hypergraph on $(2^\omega)^\alpha$, consider the hypergraph $R$ on $(2^\omega)^\alpha \times \omega$ where $\{ (\bar x_0, k_0), \dots, (\bar x_{n-1}, k_{n-1}) \} \in R$ iff $\{ \bar x_0, \dots, \bar x_{n-1} \} \in R$. Then, if $M$ is an arbitrary countable elementary model with $R,\alpha \in M$ and if $k=0$, $(1)_{R,M,k}$ holds vacuously true. Applying the claim we find $M^+$ so that either $(1)_{R,M,1}$ or $(2)_{R,M^+,0}$. The two options easily translate to the conclusion of the main lemma.  

Let us first consider the successor step. Assume that $\alpha = \beta +1$, $R$ is an analytic hypergraph on $(2^\omega)^\alpha \times \omega$ and $M$ a countable model with $R \in M, \alpha +1 \subseteq M$ so that $(1)_{R,M,k}$ holds true for some given $k \in \omega$. Let $\mathbb{Q}$ be the forcing adding mutual Cohen reals $\langle z_{0,i,j}, z_{1,i,j} : i, j \in \omega \rangle$ in $2^\omega$. Then we define the hypergraph $\tilde R$ on $(2^\omega)^\beta \times \omega$ where $\{(\bar y_0, m_0), \dots, (\bar y_{n-1}, m_{n-1}) \} \in \tilde R \cap [(2^\omega)^\beta \times \omega]^n$ iff there is $p \in \mathbb{Q}$ and there are $K_i \in \omega$, $k_{i,0}, \dots, k_{i,K_i -1} < k$ for every $i < n$, so that $$p \Vdash_{\mathbb{Q}} \bigcup_{i<n} \{ (\bar y_i{}^\frown \dot z_{0,i,j}, k_{i,j}) : j < K_i \} \cup \{ (\bar y_i{}^\frown \dot z_{1,i,j}, k) : j < m_i \} \in R.$$ 

Then $\tilde R$ is analytic (see e.g. \cite[29.22]{Kechris1995}).

\begin{claim}
\label{claim:tildeR1}
$(1)_{\tilde R, M, 1}$ is satisfied. 
\end{claim}

\begin{proof}
Suppose $\bar y_0, \dots, \bar y_{n-1}$ are pairwise distinct and strongly mCg over $M$, but $\{(\bar y_0{}, 0),$ $\dots,(\bar y_{n-1}, 0) \} \in \tilde R$ as witnessed by $p \in \mathbb{Q}$, $\langle K_i : i < n \rangle$ and $\langle k_{i,j} : i < n, j < K_i \rangle$, each $k_{i,j} < k$. More precisely, \begin{equation}\tag{$*_0$}
    p \Vdash_{\mathbb{Q}} \bigcup_{i<n} \{ (\bar y_i{}^\frown \dot z_{0,i,j}, k_{i,j}) : j < K_i \} \in R.
\end{equation}

By absoluteness, $(*_0)$ is satisfied in $M[\bar y_0, \dots, \bar y_{n-1}]$. Thus, let $\langle z_{0,i,j}, z_{1,i,j} : i, j \in \omega \rangle$ be generic over $M[\bar y_0, \dots, \bar y_{n-1}]$ with $p$ in the associated generic filter. Then the $\bar y_i{}^\frown z_{(0,i,j)}$ for $i < n, j < K_i$ are pairwise distinct and strongly mCg over $M$, but $$ \bigcup_{i<n} \{ (\bar y_i{}^\frown z_{0,i,j}, k_{i,j}) : j < K_i \} \in R.$$ This poses a contradiction to $(1)_{R,M,k}$.  
\end{proof}

\begin{claim}
If $(1)_{\tilde R, M, m}$ is satisfied for every $m \in \omega$, then also $(1)_{R,M,k+1}$. 
\end{claim}

\begin{proof}
Let $\bar x_0, \dots, \bar x_{n-1} \in (2^\omega)^\alpha$ be pairwise distinct, strongly mCg over $M$ and let $k_0, \dots, k_{n-1} \leq k$. Then we may write $\{ (\bar x_0, k_0), \dots, (\bar x_{n-1}, k_{n-1}) \}$ as \begin{equation}\tag{$*_1$}\bigcup_{i < n'} \{ (\bar y_i{}^\frown z_{0,i,j}, k_{i,j}) : j < K_i \} \cup \{ (\bar y_i{}^\frown z_{1,i,j}, k) : j < m_i \},\end{equation} for some pairwise distinct $\bar y_0, \dots,\bar y_{n'-1}$, $\langle K_i : i < n' \rangle$, $\langle k_{i,j} : i < n', j<K_i \rangle$, $\langle m_i : i < n' \rangle$ and $\langle z_{0,i,j} : i, j \in \omega \rangle$, $\langle z_{1,i,j} : i, j \in \omega \rangle$ $2^\omega$-mutually Cohen generic over $M[\bar y_0, \dots, \bar y_{n'-1}]$. Letting $m = \max_{i < n'} m_i +1$, we follow the $R$-independence of the set in $(*_1)$ from $(1)_{\tilde R,M,m}$. 
\end{proof}

\begin{claim}
If there is $m \in \omega$ so that $(1)_{\tilde R, M, m}$ fails, then there is a countable model $M^+ \supseteq M$ so that $(2)_{R,M^+,k}$. 
\end{claim}

\begin{proof}
Let $m\geq 1$ be least so that $(2)_{\tilde R, M_0, m}$ for some countable model $M_0 \supseteq M$. We know that such $m$ exists, since from $(1)_{\tilde R, M, 1}$ we follow that either $(1)_{\tilde R, M, 2}$ or $(2)_{\tilde R, M_0,1}$ for some $M_0$, then, if $(1)_{\tilde R, M, 2}$, either $(1)_{\tilde R, M, 3}$ or $(2)_{\tilde R, M_0, 2}$ for some $M_0$, and so on. Let $\phi_0, \dots, \phi_{N-1}$, $m_0, \dots, m_{N-1} \leq m$ and $\bar s \in \bigotimes_{i < \beta}(2^{<\omega})$ witness $(2)_{\tilde R, M_0, m}$. Let $M_1$ be a countable elementary model such that $\phi_0, \dots, \phi_{N-1}, M_0 \in M_1$. Then we have that for any $\bar y$ that is Cohen generic in $(2^\omega)^\beta \cap [\bar s]$ over $M_1$, in particular over $M_0$, that $$\{(\bar y, m) \} \cup \{ (\phi_j(\bar y), m_j) : j <N \} \in \tilde R,$$ i.e. there is $p \in \mathbb{Q}$, there are $K_i \in \omega$, $k_{i,0}, \dots, k_{i,K_i -1} < k$ for every $i \leq N$, so that \begin{multline} \tag{$*_2$}
    p \Vdash_{\mathbb{Q}} \bigcup_{i<N} \{ (\phi_i(\bar y){}^\frown \dot z_{0,i,j}, k_{i,j}) : j < K_i \} \cup \{ (\phi_i(\bar y){}^\frown \dot z_{1,i,j}, k) : j < m_i \} \\ \cup \{(\bar y{}^\frown \dot z_{0,N,j}, k_{N,j}) : j < K_N  \} \cup \{(\bar y{}^\frown \dot z_{1,N,j}, k) : j < m \} \in R.
\end{multline}

By extending $\bar s$, we can assume wlog that $p$, $\langle K_i : i \leq N \rangle$, $\langle k_{i,j} : i \leq N, j < K_i \rangle$ are the same for each $\bar y \in [\bar s]$ generic over $M_1$, since $(*_2)$ can be forced over $M_1$. Also, from the fact that $\phi_j$ is continuous for every $j < N$, that $\phi_j(\bar y) \neq \bar y$ for every $j < N$, and that $\phi_{j_0}(\bar y) \neq \phi_{j_1}(\bar y)$ for every $j_0 < j_1 < N$, we can assume wlog that for any $\bar y_0, \bar y_1 \in [\bar s]$ and $j_0 < j_1 < N$, \begin{equation}\tag{$*_3$}
\phi_{j_0}(\bar y_0) \neq \bar y_1 \text{ and } \phi_{j_0}(\bar y_0) \neq \phi_{j_1}(\bar y_1).
\end{equation} Let us force in a finite support product over $M_1$ continuous functions $\chi_{0,i,j} \colon (2^\omega)^\beta \to [p(0,i,j)]$ and $\chi_{1,i,j} \colon (2^\omega)^\beta \to [p(1,i,j)]$ for $i,j \in \omega$ and write $M^+ = M_1[\langle \chi_{0,i,j}, \chi_{1,i,j} : i,j \in \omega \rangle]$. For every $i < N$ and $j < K_i$ and $\bar x \in (2^\omega)^\alpha$, define $$\phi_{0,i,j}(\bar x ) := \phi_i(\bar x \restriction \beta){}^\frown \chi_{0,i,j}(\phi_i(\bar x \restriction \beta)) \text{ and } k_{0,i,j} = k_{i,j}.$$ For every $i< N$ and $j < m_i$ and $\bar x \in (2^\omega)^\alpha$, define $$\phi_{1,i,j}(\bar x ) := \phi_i(\bar x \restriction \beta){}^\frown \chi_{1,i,j}(\phi_i(\bar x \restriction \beta)) \text{ and } k_{1,i,j} = k.$$ For every $j < K_{N}$ and $\bar x \in (2^\omega)^\alpha$, define $$\phi_{0,N,j}(\bar x ) := \bar x \restriction \beta{}^\frown \chi_{0,N,j}(\bar x \restriction \beta) \text{ and } k_{0,N,j} = k_{N,j}.$$ At last, define for every $j < m-1$ and $\bar x \in (2^\omega)^\alpha$, $$\phi_{1,N,j}(\bar x ) := \bar x \restriction \beta{}^\frown \chi_{1,N,j}(\bar x \restriction \beta) \text{ and } k_{1,N,j} = k.$$

Let $\bar t \in \bigotimes_{i < \alpha} 2^{<\alpha}$ be $\bar s$ with $p(1,N,m-1)$ added in coordinate $\beta$. Now we have that for any $\bar x \in [\bar t]$ that is Cohen generic in $(2^\omega)^\alpha$ over $M^+$, \begin{multline*}
    \{ (\bar x, k) \} \cup \{ (\phi_{0,i,j}(\bar x), k_{0,i,j}) : i \leq N, j < K_N \} \cup \{ (\phi_{1,i,j}(\bar x), k_{1,i,j}) : i < N, j < m_i\} \\ \cup \{(\phi_{1,N,j}(\bar x), k_{1,N,j}) : j < m -1 \} \in R.
\end{multline*}

This follows from $(*)_2$ and applying Lemma~\ref{lem:forcingcont} to see that the $\chi_{0,i,j}(\phi_i(\bar x \restriction \beta))$, $\chi_{1,i,j}(\phi_i(\bar x \restriction \beta))$, $\chi_{0,N,j}(\bar x \restriction \beta)$, $\chi_{1,N,j}(\bar x \restriction \beta)$ and $x(\beta)$ are mutually Cohen generic over $M_1[\bar x \restriction \beta]$. Moreover they correspond to the reals $z_{0,i,j}$, $z_{1,i,j}$ added by a $\mathbb{Q}$-generic over $M_1[\bar x \restriction \beta]$, containing $p$ in its generic filter. Also, remember that $(*)_2$ is absolute between models containing the relevant parameters, which $M_1[\bar y]$ is, with $\bar y = {\bar x \restriction \beta}$. 

On the other hand, whenever $\bar x_0, \dots, \bar x_{n-1} \in (2^\omega)^\alpha \cap [\bar t]$ are pairwise distinct and strongly mCg over $M^+$, letting $\bar y_0, \dots, \bar y_{n'-1}$ enumerate $\{\bar x_i \restriction \beta : i < n \}$, we have that \begin{equation}\tag{$*_4$}
    \{ (\bar y_i{},m-1) : i < n' \} \cup \{ (\phi_j(\bar y_i), m_j) : i < n', j < N \} \text{ is } \tilde R \text{-independent.}
\end{equation} According to the definition of $\tilde R$, $(*_4)$ is saying e.g. that whenever $A \cup B \subseteq (2^\omega)^\alpha$ is an arbitrary set of strongly mCg reals over $M_1$, where $A \restriction \beta, B \restriction \beta \subseteq \{ \bar y_i, \phi_j(\bar y_i)  : i < n', j < N \}$ and in $B$, $\bar y_i$ is extended at most $m-1$ many times and $\phi_j(\bar y_i)$ at most $m_j$ many times for every $i < n'$, $j < N$, and, assuming for now that $k > 0$, if $f \colon A \to k$, then $$\{ (\bar x{}, f(\bar x)) : \bar x \in A \} \cup (B \times \{k\}) \text{ is } R\text{-independent.}$$

As an example for such sets $A$ and $B$ we have, $$A = \{\phi_{0,i,j}(\bar x_l) : l < n, i \leq N, j < K_i \} \cup \{ \bar x_l : l < n' \}, \text{ and}$$ $$B = \{ \phi_{1,i,j}(\bar x_l) : l < n, i < N, j < m_i \} \cup \{ \phi_{1,N,j}(\bar x_l) : l < n', j < m-1 \}.$$ Again, to see this we apply Lemma~\ref{lem:forcingcont} to show that the relevant reals are mutually generic over the model $M_1[\bar y_0, \dots, \bar y_{n'-1}]$. Also, remember from the definition of $\phi_{1,i,j}$ for $i < N$ and $j < m_i$ that, if $\phi_i(\bar x_{l_0} \restriction \beta) = \phi_i(\bar x_{l_1} \restriction \beta)$, then also $\phi_{1,i,j}(\bar x_{l_0}) = \phi_{1,i,j}(\bar x_{l_1})$, for all $l_0, l_1 < n$. Equally, if $\bar x_{l_0} \restriction \beta = \bar x_{l_1} \restriction \beta$, then $\phi_{1,N,j}(\bar x_{l_0}) = \phi_{1,N,j}(\bar x_{l_1})$ for every $j < m-1$. Use $(*_3)$ to note that $\{ \bar y_i : i < n'\}$, $\{\phi_0(\bar y_i) : i < n' \}, \dots$, ${\{ \phi_{N-1}(\bar y_i) : i < n'  \}}$ are pairwise disjoint. From this we can follow that indeed, each $\bar y_i$ is extended at most $m-1$ many times in $B$ and $\phi_j(\bar y_i)$ at most $m_i$ many times. In total, we get that \begin{multline*}
    \{(\phi_{0,i,j}(\bar x_l), k_{0,i,j}) : l < n, i \leq N, j < K_i \} \cup \{ (\bar x_l,k-1) : l < n' \} \cup \\ \{ (\phi_{1,i,j}(\bar x_l), k) : l < n, i < N, j < m_i \} \cup \{ (\phi_{1,N,j}(\bar x_l), k) : l < n', j < m-1 \} \\ \text{ is } R\text{-independent}. 
\end{multline*}

It is now easy to check that we have the witnesses required in the statement of $(2)_{R,M^+,k}$. For example, $\phi_{0,i,j}(\bar x) \neq \bar x$ when $i < N$, follows from $\phi_i(\bar x) \neq \bar x$. For the values $\phi_{0,N,j}(\bar x)$ we simply have that $\chi_{0,N,j}(\bar x \restriction \beta) \neq x(\beta)$, as the two values are mutually generic. Everything else is similar and consists only of a few case distinctions. Also, the continuity of the functions is clear. 

If $k = 0$, then we can simply forget the set $A$ above, since $K_i$ must be $0$ for every $i \leq N$. In this case we just get that \begin{multline*}
   \{ (\phi_{1,i,j}(\bar x_l), k) : l < n, i < N, j < m_i \} \cup \{ (\phi_{1,N,j}(\bar x_l), k) : l < n', j < m-1 \} \\ \text{ is } R\text{-independent}, 
\end{multline*}
which then yields $(2)_{R,M^+,k}$.
\end{proof}

This finishes the successor step. Now assume that $\alpha$ is a limit ordinal. We fix some arbitrary tree $T \subseteq \omega^{<\omega}$ such that for every $t \in T$, $\vert \{ n \in \omega : t^\frown n \in T\} \vert = \omega$ and for any branches $x \neq y \in [T]$, if $d = \min \{ i \in \omega : x(i)\neq y(i) \}$ then $x(j) \neq x(j)$ for every $j \geq d$. We will use $T$ only for notational purposes. For every sequence $\xi_0 < \dots < \xi_{k'} = \alpha$, we let $\mathbb{Q}_{\xi_0, \dots, \xi_{k'}} = \left(\prod_{l < k'} (\bigotimes_{i \in [\xi_l, \xi_{l+1})} 2^{<\omega})^{<\omega}\right) \times ( \bigotimes_{i \in [\xi_0, \alpha)} 2^{<\omega})^{<\omega}$. $\mathbb{Q}_{\xi_0, \dots, \xi_{k'}}$ adds, in the natural way, reals $\langle \bar z^0_{l, i} : l < k', i \in \omega \rangle$ and $\langle \bar z^0_i : i \in \omega \rangle$, where $\bar z^0_{l,i} \in (2^\omega)^{[\xi_l, \xi_{l+1})}$ and $\bar z^1_i \in (2^{\omega})^{[\xi_0, \alpha)}$ for every $l < k'$, $i \in \omega$. Whenever $t \in T \cap \omega^{k'}$, we write $\bar z^0_{t} = \bar z^0_{0,t(0)}{}^\frown \dots {}^\frown \bar z^0_{k'-1, t(k'-1)}$.~ Note that for generic $\langle \bar z^0_{l,i} : i \in \omega, l < k' \rangle$, the reals $\langle \bar z^0_t : t \in T \cap \omega^{k'}\rangle$ are strongly $\langle 2^\omega : i \in [\xi_0, \alpha)\rangle$-mCg.

Now, let us define for each $\xi < \alpha$ an analytic hypergraph $R_\xi$ on $(2^\omega)^\xi \times 2$ so that $\{(\bar y^0_i, 0) : i < n_0 \} \cup \{ (\bar y^1_i, 1) : i < n_1 \} \in R_\xi \cap [(2^\omega)^\xi \times 2]^{n_0+n_1}$, where $\vert \{(\bar y^0_i, 0) : i < n_0 \}\vert = n_0$ and $\vert \{ (\bar y^1_i, 1) : i < n_1 \} \vert = n_1$, iff there are $\xi_0 = \xi < \dots < \xi_{k'} = \alpha$, $(p,q) \in \mathbb{Q}_{\xi_0, \dots, \xi_{k'}}$, $K_i \in \omega$, $k_{i,0}, \dots, k_{i,K_i -1}<k$ and distinct $t_{i,0}, \dots, t_{i,K_i -1} \in T \cap \omega^{k'}$ for every $i < n_0$, so that $t_{i_0,j_0}(0) \neq t_{i_1,j_1}(0)$ for all $i_0 < i_1 < n_0$ and $j_0 < K_{i_0}, j_1 < K_{i_1}$, and $$(p,q) \Vdash_{\mathbb{Q}_{\bar \xi}} \bigcup_{i < n_0} \{ (\bar y^0_i{}^\frown \bar z^0_{t_{i,j}}, k_{i,j}): j < K_i \} \cup \{ (\bar y^1_i{}^\frown \bar z^1_{i}, k): i < n_1 \} \in R.$$

Note that each $R_\xi$ can be defined within $M$. It should be clear, similar to the proof of Claim~\ref{claim:tildeR1}, that from $(1)_{R,M,k}$, we can show the following. 
\begin{claim}
For every $\xi < \alpha$, $(1)_{R_\xi,M,1}$. 
\end{claim}

\begin{claim}
Assume that for every $\xi < \alpha$, $(1)_{R_\xi,M,2}$. Then also $(1)_{R,M,k+1}$.
\end{claim}

\begin{proof}
Let $\bar x^0_0, \dots, \bar x^0_{n_0-1}, \bar x^1_0, \dots, \bar x^1_{n_1-1}$ be pairwise distinct and strongly mCg over $M$ and $k_0, \dots, k_{n_0 -1} < k$. Then there is $\xi < \alpha$ large enough so that $\bar x^0_0 \restriction \xi, \dots, \bar x^0_{n_0-1}\restriction \xi , \bar x^1_0\restriction \xi, \dots, \bar x^1_{n_1-1}\restriction \xi$ are pairwise distinct and in particular, $\bar x^0_0 \restriction [\xi,\alpha), \dots, \bar x^0_{n_0-1}\restriction [\xi, \alpha) , \bar x^1_0\restriction [\xi, \alpha), \dots, \bar x^1_{n_1-1}\restriction [\xi, \alpha)$ are pairwise different in every coordinate. Let $\xi_0 = \xi$, $\xi_1 = \alpha$, $K_i = 1$ for every $i < n_0$ and $t_{0,0}, \dots, t_{n_0 -1,0} \in T \cap \omega^1$ pairwise distinct. Also, write $k_{0,0} = k_0$, ..., $k_{n_0-1,0} = k_{n_0 -1}$. Then, from $(1)_{R_\xi,M,2}$, we have that $$\mathbbm{1} \Vdash_{\xi_0,\xi_1} \{ ((\bar x^0_i \restriction \xi ){}^\frown \bar z^0_{t_{i,0}}, k_{i,0}): i < n_0 \} \cup \{((\bar x^1_i\restriction \xi){}^\frown \bar z^1_{i}, k): i < n_1 \} \text{ is } R\text{-independent}.$$ By absoluteness, this holds true in $M[\langle \bar x^0_i \restriction \xi,\bar x^1_j\restriction \xi: i < n_0, j < n_1\rangle]$ and we find that $$\{(\bar x^0_i{}, k_i) : i < n_0 \} \cup \{ (\bar x^1_i, k) \} \text{ is } R\text{-independent},$$ as required.
\end{proof}

\begin{claim}
If there is $\xi < \alpha$ so that $(1)_{R_\xi,M,2}$ fails, then there is a countable model $M^+ \supseteq M$ so that $(2)_{R,M^+,k}$.  
\end{claim}

\begin{proof}
If $(1)_{R_\xi, M, 2}$ fails, then there is a countable model $M_0 \supseteq M$ so that $(2)_{R_\xi, M, 1}$ holds true as witnessed by $\bar s \in \bigotimes_{i < \xi} 2^{<\omega}$, $\phi^0_{0}, \dots, \phi^0_{N_0 -1}, \phi^1_{0}, \dots, \phi^1_{N_1 -1} \colon (2^\omega)^\xi \to (2^\omega)^\xi$ such that for any pairwise distinct $\bar y_0, \dots, \bar y_{n-1} \in (2^\omega)^\xi \cap [\bar s]$ that are strongly mCg over $M_0$, \begin{equation}\tag{$*_5$}
    \{ (\bar y_i, 0) : i < n\} \cup \{(\phi^0_j(\bar y_i), 0) : i < n, j < N_0\} \cup \{ (\phi^1_j(\bar y_i), 1) : i < n, j < N_1\} 
\end{equation}
is $R_\xi$-independent, but 
\begin{equation}\tag{$*_6$}
    \{ (\bar y_0, 1) \} \cup \{ (\phi^0_j(\bar y_0), 0) : j < N_0\} \cup  \{ (\phi^1_j(\bar y_0), 1) : j < N_1\} \in R_\xi.  
\end{equation}
As before, we may pick $M_1 \ni M_0$ elementary containing all relevant information, assume that $(*_6)$ is witnessed by fixed $\xi_0 = \xi < \dots< \xi_{k'} = \alpha$, $(p,q) \in \mathbb{Q}_{\xi_0, \dots, \xi_{k'}}$, $K_0, \dots, K_{N_0-1}$, $k_{i,0}, \dots, k_{i, K_i -1}$ and $t_{i,0}, \dots, t_{i, K_i-1} \in T \cap \omega^{k'}$ for every $i < N_0$, so that for every generic $\bar y_0 \in (2^\omega)^\xi \cap [\bar s]$ over $M_1$, \begin{multline}\tag{$*_7$}
    (p,q) \Vdash_{\mathbb{Q}_{\bar \xi}} \{(\bar y_0{}^\frown \bar z^1_{N_1}, k) \}\cup \bigcup_{i < N_0}\{ (\phi^0_i(\bar y_0){}^\frown \bar z^0_{t_{i,j}}, k_{i,j}): j < K_i\} \cup \\ \{ (\phi^1_j(\bar y_0){}^\frown \bar z^1_{j}, k): j < N_1 \} \in R.
\end{multline}

As before, we may also assume that $\bar y_0 \neq \phi^{j_0}_{i_0}(\bar y_0) \neq \phi^{j_1}_{i_1}(\bar y_1)$ for every $\bar y_0, \bar y_1 \in [\bar s]$ and $(j_0,i_1) \neq (j_1,i_1)$. We let $\bar s' = \bar s{}^\frown q(N_1)$. Now we force continuous functions $\chi^0_{l,i} \colon (2^\omega)^\xi \to (2^{\omega})^{[\xi_l, \xi_{l+1})} \cap [p(l,i)]$ and $\chi^1_{i} \colon (2^\omega)^\xi \to (2^{\omega})^{[\xi, \alpha)} \cap [q(i)]$ over $M_1$ for every $i \in \omega$, $l < k'$ and we let $M^+ = M_1[\langle \chi^0_{l,i}, \chi^1_{i} : i \in \omega, l < k'\rangle]$. Finally we let $$ \phi_{0,i,j}(\bar x) := \phi^0_i(\bar x \restriction \xi){}^\frown \chi_{0,t_{i,j}(0)}{}(\phi^0_i(\bar x \restriction \xi)){}^\frown \dots{}^\frown \chi_{k'-1,t_{i,j}(k'-1)}(\phi^0_i(\bar x \restriction \xi))$$ for every $i < N_0$ and $j < K_i$, $\bar x \in (2^\omega)^\alpha$, and $$\phi_{1,i}(\bar x) :=  \phi^1_i(\bar x \restriction \xi){}^\frown \chi_{1,i}(\phi
^1_i(\bar x \restriction \xi))$$ for every $i < N_1$, $\bar x \in (2^\omega)^\alpha$.

We get from $(*_7)$, and, as usual, applying Lemma~\ref{lem:forcingcont}, that for any $\bar x \in (2^\omega)^\alpha \cap [\bar s']$ which is generic over $M^+$, 
   $$ \{(\bar x, k)\} \cup \bigcup_{i < N_0} \{(\phi_{0,i,j}(\bar x), k_{i,j}) : j < K_i \} \cup \{(\phi_{1,i}(\bar x), k) : i < N_1 \} \in R.$$
   
   On the other hand, whenever $\bar x_0,\dots, \bar x_{n'-1}\in (2^{\omega})^\alpha \cap [\bar s']$ are strongly mCg over $M^+$, and letting $\bar y_0, \dots, \bar y_{n-1}$ enumerate $\{ \bar x_i \restriction \xi : i < n'\}$, knowing that the set in $(*_5)$ is $R_\xi$-independent, we get that \begin{multline*}
       \{ (\bar x_l,k-1) : l < n' \} \cup \bigcup_{i < N_0} \{(\phi_{0,i,j}(\bar x_l),k_{i,j}) : j < K_i, l < n' \} \cup \\\{(\phi_{1,i}(\bar x_l), k) : i < N_1, l < n' \} \text{ is } R\text{-independent},
   \end{multline*}
   
 in case $k>0$. To see this, we let $\eta_0< \dots< \eta_{k''}$ be a partition refining $\xi_0 < \dots, \xi_{k'}$ witnessing the mCg of $\bar x_0 \restriction [\xi, \alpha), \dots, \bar x_{n'-1}\restriction [\xi, \alpha)$ and we find appropriate $u_{0,0}, \dots, u_{0,L_0 -1}, \dots, u_{n-1,0 }, \dots, u_{n-1,L_{n-1}-1} \in T \cap \omega^{k''}$ and $v_{i,j} \in T \cap \omega^{k''}$ for $i < N_0, j < K_i$ to interpret the above set in the form \begin{multline*}
       \{ (\bar y_l{}^\frown \bar z^0_{u_{l,i}}, k-1) : l < n, i < L_i \} \cup \bigcup_{i < N_0} \{(\phi^0_i(\bar y_l){}^\frown \bar z^0_{v_{i,j}}, k_{i,j}) : i < N_0, j < K_i, l < n \} \\\cup \{(\phi^1_{i}(\bar y_l){}^\frown \bar z^1_i , k) : i < N_1, l < n \},
   \end{multline*} for $\mathbb{Q}_{\eta_0, \dots, \eta_{k''-1}}$-generic $\langle \bar z
^0_{l,i}, \bar z^1_{i} : l < k'', i \in \omega \rangle$ over $M_1[\bar y_0, \dots, \bar y_{n-1}]$. We leave the details to the reader. In case $k= 0$, all $K_i$ are $0$ and we get that $$\{(\phi_{1,i}(\bar x_l), k) : i < N_1, l < n' \} \text{ is } R\text{-independent}.$$
Everything that remains, namely showing e.g. that $\bar x \neq \phi_{1,i}(\bar x)$ is clear.
\end{proof}
As a final note, let us observe that the case $\alpha = 0$ is trivial, since $(2^\omega)^\alpha$ has only one element. 
\end{proof}

\begin{remark}\label{rem:E1}
	If we replace ``strong mCg" with ```mCg" in the above Lemma, then it already becomes false for $\alpha = \omega$. Namely consider the equivalence relation $E$ on $(2^\omega)^\omega$, where $\bar x E \bar y$ if they eventually agree, i.e. if $\exists n \in \omega \forall m \geq n (x(n) = y(n))$.\footnote{This equivalence relation is usually called $E_1$.} Then we can never be in case (1) since we can always find two distinct $\bar x$ and $\bar y$ that are mCg and $\bar x E \bar y$. On the other hand, in case (2) we get a continuous selector $\phi_0$ for $E$ (note that $N = 0$ is not possible). More precisely we have that for any $\bar x$, $\bar y$ that are mCg, $\bar x E \phi_0(\bar x)$ and $\phi_0(\bar x) = \phi_0(\bar y)$ iff $\bar x E \bar y$. But for arbitrary mCg $\bar x$ and $\bar y$ so that $\bar x \neg E \bar y$, we easily find a sequence $\langle \bar x_n : n \in \omega \rangle$ so that $\bar x$ and $\bar x_n$ are mCg and $\bar x E \bar x_n$, but $\bar x_n \restriction n = \bar y \restriction n$ for all $n$. In particular $\lim_{n \in \omega} \bar x_n = \bar y$. Then $\phi_0(\bar y) = \lim_{n\in \omega} \phi_0(\bar x_n) = \lim_{n\in \omega} \phi_0(\bar x) = \phi_0(\bar x)$.
\end{remark}

The proofs of Main Lemma~\ref{thm:mainlemma} and \ref{lem:mainlemmainf} can be generalized to $E$ that is $\omega$-universally Baire, in particular they also hold for coanalytic hypergraphs. The only assumptions on analytic sets that we used in the proofs are summarized below.  

\begin{prop}
Let $\Gamma$ be a pointclass closed under countable unions, countable intersections and continuous preimages and assume that for every $A \in \Gamma \cap \mathcal{P}(\omega^\omega)$, there are formulas $\varphi$, $\psi$ (with parameters) in the language of set theory, such that for every countable elementary model $M$ (with the relevant parameters) and $G$ a generic over $M$ for a finite support product of Cohen forcing, 
\begin{enumerate}
    \item for $x \in M[G] \cap \omega^\omega$, $M[G] \models \varphi(x) \text{ iff } x \in A \text{ and } M[G] \models \psi(x) \text{ iff } x \notin A,$
    \item for $\dot x \in M[G]$ a $\mathbb{C}$-name for a real, $p \in \mathbb{C}$, $M[G] \models ``p \Vdash \varphi(\dot x)"$ iff $p \Vdash \varphi(\dot x)$, 
    \item for $\dot y$ a $\mathbb{C}$-name for a real, $p \in \mathbb{C}$ and a continuous function $f \colon \omega^\omega \times \omega^\omega \to \omega^\omega$, $\{ x \in \omega^\omega : p \Vdash \varphi(f(x,\dot y)) \} \in \Gamma.$
\end{enumerate}
Then Main Lemma~\ref{thm:mainlemma} and \ref{lem:mainlemmainf} hold, where ``analytic" is replaced by $\Gamma$.
\end{prop}

\begin{definition}\label{def:Delta}
For $\bar x_0, \dots, \bar x_{n-1} \in \prod_{i<\alpha} X_i$, we define $$\Delta(\bar x_0, \dots, \bar y_{n-1}) := \{ \Delta_{\bar x_i, \bar x_j} : i \neq j < n \} \cup \{ 0, \alpha\},$$ where $\Delta_{\bar x_i, \bar x_j} := \min \{ \xi < \alpha: x_i(\xi) \neq x_j(\xi) \}$ if this exists and $\Delta_{\bar x_i, \bar x_j} = \alpha$ if $\bar x_i = \bar x_j$. 
\end{definition}

\begin{remark}
Whenever $\bar x_0, \dots, \bar x_{n-1}$ are strongly mCg, then they are mCg as witnessed by the partition $\xi_0 < \dots < \xi_k$, where $\{\xi_0, \dots, \xi_k \} = \Delta(\bar x_0, \dots, \bar x_{n-1})$.
\end{remark}

\section{Sacks and splitting forcing}

\subsection{Splitting Forcing}

\begin{definition}\label{def:splittingforcing}
	We say that $S \subseteq 2^{<\omega}$ is \textit{fat} if there is $m \in \omega$ so that for all $n \geq m$, there are $s,t \in S$ so that $s(n) =0$ and $t(n) =1$. A tree $T$ on $2$ is called \textit{splitting tree} if for every $s \in T$, $T_s$ is fat. We call \textit{splitting forcing} the tree forcing $\mathbb{SP}$ consisting of splitting trees.  
\end{definition}

Note that for $T \in \mathbb{SP}$ and $s \in T$, $T_s$ is again a splitting tree. Recall that $x \in 2^\omega$ is called splitting over $V$, if for every $y \in 2^{\omega} \cap V$, $\{ n \in \omega: y(n) = x(n) = 1 \}$ and $\{n \in \omega: x(n) = 1 \wedge y(n) =0 \}$ are infinite. The following is easy to see.

\begin{fact}
	Let $G$ be $\mathbb{SP}$-generic over $V$. Then $x_G$, the generic real added by $\mathbb{SP}$, is splitting over $V$.
\end{fact}

Whenever $S$ is fat let us write $m(S)$ for the minimal $m \in \omega$ witnessing this.

\begin{definition}
	Let $S,T$ be splitting trees and $n \in \omega$. Then we write $S \leq_n T$ iff $S \leq T$, $\splt_{\leq n}(S) =  \splt_{\leq n}(T)$ and $\forall s \in \splt_{\leq n}(S) (m(S_s) = m(T_s))$. 
\end{definition}

\begin{prop}
	The sequence $\langle \leq_n : n \in \omega \rangle$ witnesses that $\mathbb{SP}$ has Axiom A with continuous reading of names. 
\end{prop}

\begin{proof}
	It is clear that $\leq_{n}$ is a partial order refining $\leq$ and that $\leq_{n+1} \subseteq \leq_n$ for every $n \in \omega$. Let $\langle T_n : n \in \omega \rangle$ be a fusion sequence in $\mathbb{SP}$, i.e. for every $n$, $T_{n+1} \leq_n T_n$. Then we claim that $T := \bigcap_{n \in \omega} T_n$ is a splitting tree. More precisely, for $s \in T$, we claim that $m := m((T_{\vert s\vert})_s)$ witnesses that $T_s$ is fat. To see this, let $n \geq m$ be arbitrary and note that $n \geq m \geq \vert s \vert$ must be the case. Then, since $\splt_{\leq n+1}(T_{n+1}) \subseteq T$ we have that $s \in \splt_{\leq n+1}( T_{n+1})$ and $m((T_{n+1})_s) = m$. So find $t_0, t_1 \in T_{n+1}$ so that $t_0(n) = 0$, $t_1(n) = 1$ and $\vert t_0 \vert = \vert t_1 \vert = n+1$. But then $t_0,t_1 \in T$, because $t_0,t_1 \in \splt_{\leq n+1}(T_{n+1}) \subseteq T$. 
	
	Now let $D \subseteq \mathbb{SP}$ be open dense, $T \in \mathbb{SP}$ and $n \in \omega$. We will show that there is $S \leq_n T$ so that for every $x \in [S]$, there is $t \subseteq x$, with $S_t \in D$. This implies condition (3) in Definition~\ref{def:axiomA}.
	
	\begin{claim}
		Let $S$ be a splitting tree. Then there is $A \subseteq S$ an antichain (seen as a subset of $2^{<\omega}$) so that for every $k \in \omega, j \in 2$, if $\exists s \in S (s(k) = j)$, then $\exists t \in A (t(k) = j)$. 
	\end{claim}
	
	\begin{proof}
		Start with $\{ s_i : i \in \omega \} \subseteq S$ an arbitrary infinite antichain and let $m_i := m(S_{s_i})$ for every $i \in \omega$. Then find for each $i \in \omega$, a finite set $H_i \subseteq S_{s_i}$ so that for all $k \in [m_i, m_{i+1})$, there are $t_0,t_1 \in H_i$, so that $t_0(k) = 0$ and $t_1(k) = 1$.  Moreover let $H \subseteq S$ be finite so that for all $k \in [0,m_0)$ and $j \in 2$, if $\exists s \in S (s(k) =j)$, then $\exists t \in H (t(k) = j)$. Then define $F_i = H_i \cup (H \cap S_{s_i})$ for each $i \in \omega$ and let $F_{-1} := H \setminus \bigcup_{i \in \omega} F_i$. 
		Since $F_i$ is finite for every $i \in \omega$, it is easy to extend each of its elements to get a set $F_i'$ that is an additionally an antichain in $S_{s_i}$. Also extend the elements of $F_{-1}$ to get an antichain $F_{-1}'$ in $S$. It is easy to see that $A := \bigcup_{i \in [-1,\omega)} F_i'$ works. 
	\end{proof}
	
	Now enumerate $\splt_n(T)$ as $\langle \sigma_i : i < N \rangle$, $N := 2^n$. For each $i < N$, let $A_i \subseteq T_{\sigma_i}$ be an antichain as in the claim applied to $S = T_{\sigma_i}$. For every $i < N$ and $t \in A_i$, let $S^{t} \in D$ be so that $S^t \leq T_t$. For every $i <N$ pick $t_i \in A_i$ arbitrarily and $F_i \subseteq A_i$ a finite set so that for every $k \in [0, m(S^{t_i}))$ and $j \in 2$, if $\exists s \in A_i (s(k)= j)$, then $ \exists t \in F_i (t(k)=j)$. Then we see that $S := \bigcup_{i < N} (\bigcup_{t \in F_i} S^t \cup S^{t_i} )$ works. We constructed $S$ so that $S \leq_n T$. Moreover, whenever $x \in [S]$, then there is $i < N$ be so that $\sigma_i \subseteq x$. Then $x \in [\bigcup_{t \in F_i} S_t \cup S_{t_i}]$ and since $F_i$ is finite, there is $t \in F_i \cup \{t_i \}$ so that $t \subseteq x$. But then $S_{t} \leq S^t \in D$. 
	
	Finally, in order to show the continuous reading of names, let $\dot y$ be a name for an element of $\omega^\omega$, $n \in \omega$ and $T \in \mathbb{SP}$. It suffices to consider such names, since for every Polish space $X$, there is a continuous surjection $F \colon \omega^\omega \to X$. Then we have that for each $i \in \omega$, $D_i := \{ S \in \mathbb{SP} : \exists s \in \omega^{i}( S \Vdash \dot y \restriction i = s ) \}$ is dense open. Let $\langle T_i : i \in \omega \rangle$ be so that $T_0 \leq_n T$, $T_{i+1} \leq_{n+i} T_i$ and for every $x \in [T_i]$, there is $t \subseteq x$ so that $(T_i)_t \in D_i$. Then $S = \bigcap_{i \in \omega} T_i \leq_n T$. For every $x \in [S]$, define $f(x) = \bigcup \{ s \in \omega^{<\omega} : \exists t \subseteq x (S_t \Vdash s \subseteq \dot y) \}$. Then $f \colon [S] \to \omega^\omega$ is continuous and $S \Vdash \dot y = f(x_G)$.
\end{proof}

\begin{cor}
	$\mathbb{SP}$ is proper and $\omega^\omega$-bounding.
\end{cor}

\subsection{Weighted tree forcing}

    In this subsection we define a class of forcings that we call \emph{weighted tree forcing}. The definition is slightly ad-hoc, but simple enough to formulate and to check for various forcing notions. In earlier versions of this paper we proved many of the results only for splitting forcing, but we noted that similar combinatorial arguments apply more generally. The notion of a \emph{weight} resulted directly from analysing the proof of Proposition~\ref{prop:weightedmain} for splitting forcing. It also turned out to be particularly helpful in the next subsection. 

\begin{definition}
	Let $T$ be a perfect tree. A \emph{weight} on $T$ is a map $\rho \colon T \times T \to [T]^{<\omega}$ so that $\rho(s,t) \subseteq T_s \setminus T_t$ for all $s,t \in T$. Whenever $\rho_0, \rho_1$ are weights on $T$ we write $\rho_0 \subseteq \rho_1$ to say that for all $s,t \in T$, $\rho_0(s,t) \subseteq \rho_1(s,t)$. 
\end{definition}

Note that if $t \subseteq s$ then $\rho(s,t) = \emptyset$ must be the case. 

\begin{definition}
	Let $T$ be a perfect tree, $\rho$ a weight on $T$ and $S$ a tree. Then we write $S \leq_\rho T$ if $S \subseteq T$ and there is a dense set of $s_0 \in S$ with an injective sequence $(s_n)_{n\in \omega}$  in $S_{s_0}$ such that $\forall n \in \omega (\rho(s_n, s_{n+1}) \subseteq S)$. 
\end{definition}

\begin{remark}
	Whenever $\rho_0 \subseteq \rho_1$, we have that $S \leq_{\rho_1} T$ implies $S \leq_{\rho_0} T$. 
\end{remark}

\begin{definition}
\label{def:weightedtreeforcing}
	Let $\mathbb{P}$ be a tree forcing. Then we say that $\mathbb{P}$ is \emph{weighted} if for any $T \in \mathbb{P}$ there is a weight $\rho$ on $T$ so that for any tree $S$, if $S \leq_\rho T$ then $S \in \mathbb{P}$. 
\end{definition}

\begin{lemma}
	\label{lem:spweighted}
	$\mathbb{SP}$ is weighted. 
\end{lemma}

\begin{proof}
	Let $T \in \mathbb{SP}$. For any $s,t \in T$ let $\rho(s,t) \subseteq T_s \setminus T_t$ be finite so that for any $k \in \omega$ and $i \in 2$, if there is $r \in T_s$ so that $r(k) = i$ and there is no such $r \in T_t$, then there is such $r$ in $\rho(s,t)$. This is possible since 
	$T_t$ is fat. 
	Let us show that $\rho$ works. Assume that $S \leq_\rho T$ and let $s \in S$ be arbitrary. Then there is $s_0 \supseteq s$ in $S$ with a sequence $(s_n)_{n\in \omega}$ as in the definition of $\leq_\rho$. Let $k \geq m(T_{s_0})$ and $i \in 2$ and suppose there is no $r \in S_{s_0}$ with $r(k) = i$. In particular this means that no such $r$ is in $\rho(s_n,s_{n+1})$ for any $n \in \omega$, since $\rho(s_n, s_{n+1}) \subseteq S_{s_0}$. But then, using the definition of $\rho$ and $m(T_{s_0})$, we see inductively that for each $n \in \omega$ such $r$ must be found in $T_{s_n}$. Letting $n$ large enough so that $k < \vert s_n \vert$, $s_n(k) = i$ must be the case. But $s_n \in S_{s_0}$, which is a contradiction.   
\end{proof}

\begin{definition}
\textit{Sacks forcing} is the tree forcing $\mathbb{S}$ consisting of all perfect subtrees of $2^{<\omega}$. It is well-known that it is Axiom A with continuous reading of names. 
\end{definition}

\begin{lemma}
	\label{lem:sweighted}
	$\mathbb{S}$ is weighted. 
\end{lemma}

\begin{proof}
	Let $T \in \mathbb{S}$. For $s,t \in T$, we let $\rho(s,t)$ contain all $r^\frown i \in T_s \setminus T_t$ such that $r^\frown (1-i) \in T$ and where $\vert r \vert$ is minimal with this property.
\end{proof}

Recall that for finite trees $T_0$, $T_1$ we say that $T_1$ is an end-extension of $T_0$, written as $T_0 \sqsubset T_1$, if $T_0 \subsetneq T_1$ and for every $t \in T_1 \setminus T_0$ there is a terminal node $\sigma \in \term (T_0)$ so that $\sigma \subseteq t$. A node $\sigma \in T_0$ is called terminal if it has no proper extension in $T_0$. 

\begin{definition}
	Let $T$ be a perfect tree, $\rho$ a weight on $T$ and $T_0, T_1$ finite subtrees of $T$. Then we write $T_0 \lhd_\rho T_1$ iff $T_0 \sqsubset T_1$ and 
	\begin{multline*} 
	\forall \sigma \in \term(T_0) \exists N \geq 2 \exists \langle s_i\rangle_{i < N} \in ((T_{1})_\sigma)^N \text{ injective } \\ \bigl( s_0 = \sigma \wedge s_{N-1} \in \term(T_{1}) \wedge \forall i < N (\rho(s_i, s_{i+1}) \subseteq T_{1}) \bigr).\tag{$*_0$}
	\end{multline*}
	
\end{definition}

\begin{lemma}
	\label{lem:basicweight}
	Let $T$ be a perfect tree, $\rho$ a weight on $T$ and $\langle T_n : n \in \omega \rangle$ be a sequence of finite subtrees of $T$ so that $T_n \lhd_\rho T_{n+1}$ for every $n \in \omega$. Then $\bigcup_{n \in \omega} T_n \leq_\rho T$. 
\end{lemma}

\begin{proof}
	Let $S := \bigcup_{n \in \omega} T_n$. To see that $S \leq_\rho T$ note that $\bigcup_{n\in \omega} \term(T_n)$ is dense in $S$, in a very strong sense. Let $\sigma \in \term(T_n)$ for some $n \in \omega$, then let $s_0, \dots, s_{N_0-1}$ be as in $(*_0)$ for $T_n,T_{n+1}$. Since $s_{N_0-1} \in \term(T_{n+1})$ we again find $s_{N_0}, \dots, s_{N_1-1}$ so that $s_{N_0-1}, \dots, s_{N_1-1}$ is as in $(*_0)$ for $T_{n+1}, T_{n+2}$. Continuing like this, we find a sequence $\langle s_i : i \in \omega \rangle$ in $S$ starting with $s_0 = \sigma$ so that $\rho(s_i, s_{i+1}) \subseteq S$ for all $i \in \omega$, as required. 
\end{proof}

\begin{lemma}
	\label{lem:wghtdenseset}
	Let $T$ be a perfect tree, $\rho$ a weight on $T$ and $T_0$ a finite subtree of $T$. Moreover, let $k \in \omega$ and $D \subseteq (T)^k$ be dense open. Then there is $T_1 \rhd_\rho T_0$ so that
	\begin{multline*} 
	\forall \{\sigma_0, \dots, \sigma_{k -1}\} \in [\term(T_{0})]^{k} \forall \sigma_0', \dots, \sigma_{k -1}' \in \term(T_{1}) \\ \left(\forall l < k (\sigma_l \subseteq \sigma_l') \rightarrow (\sigma_0', \dots, \sigma_{k -1}') \in D\right).\tag{$*_1$}
	\end{multline*}
\end{lemma}

\begin{proof}
	First let us enumerate $\term(T_0)$ by $\sigma_0, \dots, \sigma_{K-1}$. We put $s_0^l = \sigma_{l}$ for each $l < K$. Next find for each $l < K$, $s_1^l \in T, s_0^l \subsetneq s_1^l$ above a splitting node in $T_{s_0^l}$. Moreover we find $s_2^l \in T_{s_0^l}$ so that $s_2^l \perp s_1^l$ and $s_2^l$ is longer than any node appearing in $\rho(s_0^l, s_1^l)$. This is possible since we chose $s_1^l$ to be above a splitting node in $T_{s_0^l}$. For each $l<K$ we let $\tilde{T}_{2}^l$ be the tree generated by (i.e. the downwards closure of) $\{s_1^l, s_2^l\} \cup \rho(s_0^l,s_1^l) \cup \rho(s_1^l,s_2^l)$. Note that $s_2^l \in \term(\tilde{T}_{2}^l)$ as $\rho(s_1^l,s_2^l) \perp s_2^l$.
	
	Let us enumerate by $(f_j)_{2\leq j < N}$ all functions $f \colon K \to \{1,2\}$ starting with $f_2$ the constant function mapping to $1$. We are going to construct recursively a sequence $\langle \tilde{T}_j^l : 2 \leq j \leq N \rangle$ where $\tilde{T}_j^l \sqsubseteq \tilde{T}_{j+1}^l$, and $\langle s_j^l : 2 \leq j \leq N \rangle$ without repetitions, for each $l< K$ such that at any step $j <N$: 
	
	\begin{enumerate}
		\item for every $l < K$,  $s_j^l \in \term(\tilde{T}_{j}^l)$ and $\begin{cases}  s_2^l \subseteq s_j^l & \text{if } f_j(l) = 1  \\
		s_1^l \subseteq s_j^l &\text{if } f_j(l) = 2.
		\end{cases}$
		\item  for any $\{l_i : i <k \} \in [K]^k$ and $(t_i)_{i<k}$ where $t_i \in \term(\tilde{T}_{j+1}^{l_i})$ and $\begin{cases}  s_1^{l_i} \subseteq t_{i}  & \text{if } f_j(l_i) = 1  \\
		s_1^{l_i} \perp t_{i} &\text{if } f_j(l_i) = 2
		\end{cases}$ for every $i< k$, $(t_0, \dots, t_{k-1}) \in D$
		
		\item for every $l < K$, $\rho(s_j^l, s_{j+1}^l ) \subseteq \tilde{T}_{j+1}^{l}$. 
		
	\end{enumerate}
	
	Note that $(1)$ holds true at the initial step $j=2$ since $f_2(l) = 1$, $s_2^l \subseteq s_2^l$ and $s_2^l \in \term(\tilde{T}_2^l)$ for each $l < K$. Now suppose that for some $j < N$ we have constructed $\tilde{T}_j^l$ and $s_j^l$ for each $l$ with $(1)$ holding true. Then we proceed as follows. Let $\{ t^l_i : i < N_l\}$ enumerate $\{ t :  t\in \term(\tilde{T}_j^l)  \wedge s_1^l \subseteq t \text{ if } f_j(l) = 1  \wedge s_1^l \perp t \text{ if } f_j(l) = 2\} $ for each $l <K$. Now it is simple to find $r^l_i \in T$, $t^l_i \subseteq r^l_i$ for each $i <N_l, l< K$ so that $[\{r^l_i : i < N_l, l< K \}]^k \subseteq D$.  
	
	Let $R_l$ be the tree generated by $\tilde{T}_{j}^l$ and $\{r^l_i : i < N_l\}$ for each $l <K$. It is easy to see that $\tilde{T}_{j}^l \sqsubseteq R_l$ since we only extended elements from $\term(\tilde{T}_j^l)$ (namely the $t_i^l$'s). Note that it is still the case that $s_j^l \in \term (R_l)$ since $s_j^l \perp t_i^l$ for all $i < N_l$. 
	Next we choose $s_{j+1}^l$ extending an element of $\term (R_l)$, distinct from all previous choices and so that $s_2^l \subseteq s_j^l \text{ if } f_{j+1}(l) = 1$ and $s_1^l \subseteq s_j^l \text{ if } f_{j+1}(l) = 2.$
	
	Taking $\tilde{T}_{j+1}^l$ to be the tree generated by $R_l \cup \{ s_{j+1}^l \} \cup \rho(s_j^l, s_{j+1}^{l})$ gives the next step of the construction. Again $R_l \sqsubseteq \tilde{T}_{j+1}^l$, as we only extended terminal nodes of $R_l$. Then $(3)$ obviously holds true and $s^l_{j +1} \in \term(\tilde{T}_{j+1}^l)$ since $\rho(s_j^l, s_{j+1}^{l}) \perp s^l_{j +1}$. It follows from the construction that $(2)$ holds true for each $\tilde{T}_{j+1}^l$ replaced by $R_l$. Since $R_l \sqsubseteq \tilde{T}_{j+1}^l$ we easily see that $(2)$ is satisfied. 
	
	Finally we put $T_{1} = \bigcup_{l<K} \tilde T_{N}^l$. It is clear that $(*_0)$ is true, in particular that $T_0 \lhd_\rho T_1$. For $(*_1)$ let $\{l_i : i<k \} \in [K]^k$ be arbitrary and assume that $t_i \in \term (\tilde T_{N}^{l_i})$ for each $i < k$. Let $f \colon K \to \{1,2\}$ be so that for each $i < k$ if $s_1^{l_i} \subseteq t_i$ then $f(l_i) = 1$, and if $s_1^{l_i} \perp t_i$ then $f(l_i) = 2$. Then there is $j \in [2,N)$ so that $f_j = f$. Clause $(2)$ ensured that for initial segments $t_i' \subseteq t_i$ where $t_i' \in \term(\tilde{T}_{j+1}^{l_i})$, $(t_0',\dots, t_{k-1}' ) \in D$. In particular $(t_0,\dots, t_{k-1}) \in D$ which proves $(*_1)$. 
\end{proof}

\begin{prop}
	\label{prop:weightedmain}
	Let $M$ be a countable model of set theory, $R_l \in M$ a perfect tree and $\rho_l$ a weight on $R_l$ for every $l < k \in \omega$. Then there is $S_l \leq_{\rho_l} R_l$ for every $l < k$ so that any $\bar x_0, \dots, \bar x_{n-1} \in \prod_{l < k}[S_l]$ are $\langle [R_l] : l < k\rangle$-mutually Cohen generic over $M$.
\end{prop}

\begin{proof}
	Let $T := \{ \emptyset\} \cup \{ \langle l \rangle^\frown s : s \in R_l, l<k \}$ be the disjoint sum of the trees $R_l$ for $l < k$. Also let $\rho$ be a weight on $T$ extending arbitrarily the weights $\rho_l$ defined on the copy of $R_l$ in $T$. As $M$ is countable, let $(D_n,k_n)_{n\in \omega}$ enumerate all pairs $(D,m) \in M$, such that $D$ is a dense open subset of $T^m$ and $m \in \omega \setminus \{0\}$, infinitely often. 
	Let us find a sequence $(T_n)_{n\in \omega}$ of finite subtrees of $T$, such that for each $n\in \omega$, $T_n \lhd_\rho T_{n+1}$ and
	
	\begin{multline*} 
	\forall \{\sigma_0, \dots, \sigma_{k_n -1}\} \in [\term(T_{n})]^{k_n} \forall \sigma_0', \dots, \sigma_{k_n -1}' \in \term(T_{n+1}) \\ [\forall l < k (\sigma_l \subseteq \sigma_l') \rightarrow (\sigma_0', \dots, \sigma_{k_n -1}') \in D_n].\tag{$*_1$}
	\end{multline*}
	
	We start with $T_0 = k^{<2} = \{ \emptyset \} \cup \{ \langle l \rangle : l<k \}$ and then apply Lemma~\ref{lem:wghtdenseset} recursively. Let $S := \bigcup_{n \in \omega} T_n$. Then we have that $S \leq_\rho T$. 
	
	\begin{claim}
		For any $m \in \omega$ and distinct $x_0, \dots, x_{m-1} \in [S]$, $(x_0, \dots, x_{m-1})$ is $T^m$-generic over $M$. 
	\end{claim}
	
	\begin{proof}
		Let $D \subseteq T^m$ be open dense with $D \in M$. Then there is a large enough $n \in \omega$ with $(D_n, k_n) = (D,m)$ and $\sigma_0, \dots, \sigma_{m-1} \in \term(T_n)$ distinct such that $\sigma_0 \subseteq x_0, \dots, \sigma_{m-1} \subseteq x_{m-1}$. Then there are unique $\sigma'_0, \dots, \sigma'_{m-1} \in \term(T_{n+1})$ such that $\sigma'_0 \subseteq x_0, \dots, \sigma'_{m-1} \subseteq x_{m-1}$. By $(*_1)$, $(\sigma'_0, \dots, \sigma'_{m-1}) \in D$.
	\end{proof}
	
	Finally let $S_l = \{ s : \langle l \rangle^\frown s \in S \}$ and note that $S_l \leq_{\rho_l} R_l$ for every $l < k$. The above claim clearly implies the statement of the proposition.
\end{proof}

\begin{remark}\label{rem:mgpspl} Proposition~\ref{prop:weightedmain} implies directly the main result of \cite{Spinas2007}.
	A modification of the above construction for splitting forcing can be used to show that for $T \in M$, we can in fact find a master condition $S \leq T$ so that for any distinct $x_0,\dots, x_{n-1} \in [S]$, $(x_0, \dots, x_{n-1})$ is $\mathbb{SP}^n$-generic over $M$. In that case $(S,\dots, S) \in \mathbb{SP}^n$ is a $\mathbb{SP}^n$-master condition over $M$. We won't provide a proof of this since our only application is Corollary~\ref{cor:minimality} below, which seems to be implicit in \cite{Spinas2007}. The analogous statement for Sacks forcing is a standard fusion argument. 
\end{remark}

\begin{cor}
	\label{cor:minreal}
	Let $\mathbb{P}$ be a weighted tree forcing with continuous reading of names. Then $\mathbb{P}$ adds a minimal real. In fact for any $\mathbb{P}$-generic $G$, if $y \in 2^\omega \cap V[G] \setminus V$, then there is a continuous map $f \colon 2^\omega \to A^\omega$ in $V$ so that $x_G = f(y)$.  
\end{cor}

\begin{proof}
	Using the continuous reading of names let $T \in G$ be so that there is a continuous map $g \colon [T] \to 2^\omega$ with $T \Vdash \dot y = g(x_G)$. It is easy to see from the definition, that in any weighted tree forcing, the set of finitely branching trees is dense.\footnote{In particular, weighted tree forcing with the crn is $\omega^\omega$-bounding.} Thus, let us assume that $[T]$ is compact. Moreover let $M$ be countable elementary with $g, T \in M$. Now let $S \leq T$ be so that any $x_0,x_1 \in [S]$ are $[T]$-mCg over $M$. 
	
	Suppose that there are $x_0 \neq x_1 \in [S]$, with $g(x_0) = g(x_1)$. Then there must be $s \subseteq x_0$ and $t \subseteq x_1$, so that $M \models (s,t) \Vdash_{T^2} g(\dot c_0) = g(\dot c_1)$, where $\dot c_0$, $\dot c_1$ are names for the generic branches added by $T^2$. But then note that for any $x \in S_{t}$, since $x$ and $x_0$ are mCg and $s \subseteq x_0$, $t \subseteq x$, we have that $g(x) = g(x_0)$. In particular $g$ is constant on $S_t$ and $S_t \Vdash g(x_G) = g(\check x_0) \in V$. 
	
	On the other hand, if $g$ is injective on $[S]$, then $g^{-1}$ is continuous as $[S]$ is compact and it is easy to extend $g^{-1}$ to a continous function $f \colon A^\omega \to 2^\omega$.
\end{proof}

\begin{cor}\label{cor:minimality}
	$V^{\mathbb{SP}}$ is a minimal extension of $V$, i.e. whenever $W$ is a model of ZFC so that $V \subseteq W \subseteq V^{\mathbb{SP}}$, then $W = V$ or $W = V^{\mathbb{SP}}$.
\end{cor}

\begin{proof}
	Let $G$ be an $\mathbb{SP}$-generic filter over $V$. It suffices to show that if $\langle \alpha_\xi : \xi < \delta \rangle \in W \setminus V$ is an increasing sequence of ordinals, then $x_G \in W$ (see also \cite[Theorem 13.28]{Jech2013}). So let $ \langle \dot \alpha_\xi : \xi < \delta \rangle$ be a name for such a sequence of ordinals and $T \in \mathbb{SP}$ be such that $T \Vdash \langle \dot \alpha_\xi : \xi < \delta \rangle \notin V$. Note that this is in fact equivalent to saying that $(T,T) \Vdash_{\mathbb{SP}^2} \langle \dot \alpha_\xi[\dot x_0] : \xi < \delta \rangle \neq \langle \dot \alpha_\xi[\dot x_1] : \xi < \delta \rangle$, where $\dot x_0, \dot x_1$ are names for the generic reals added by $\mathbb{SP}^2$. Let $M$ be a countable elementary model so that $T, \langle \dot \alpha_\xi : \xi < \delta \rangle \in M$ and let $T' \leq T$ be a master condition over $M$ as in Remark~\ref{rem:mgpspl}. Then also $T' \Vdash \langle \dot \alpha_\xi : \xi \in \delta \cap M \rangle \notin V$. Namely, suppose towards a contradiction that there are $x_0, x_1 \in [T']$ generic over $V$ so that $\langle \dot \alpha_\xi[x_0] : \xi \in \delta \cap M \rangle = \langle \dot \alpha_\xi[x_1] : \xi \in \delta \cap M \rangle$, then $(x_0,x_1)$ is $\mathbb{SP}^2$-generic over $M$ and $M[x_0][x_1] \models \langle \dot \alpha_\xi[x_0] : \xi < \delta \rangle = \langle \dot \alpha_\xi[x_1] : \xi < \delta \rangle$ which yields a contradiction to the sufficient elementarity of $M$. Since $T' \Vdash \langle \dot \alpha_\xi : \xi \in \delta \cap M \rangle \subseteq M$ we can view $\langle \dot \alpha_\xi : \xi \in \delta \cap M \rangle$ as a name for a real, for $M$ is countable. Back in $W$, we can define $\langle \alpha_\xi : \xi \in \delta \cap M \rangle$ since $M \in V \subseteq W$. But then, applying Corollary~\ref{cor:minreal}, we find that $x_G \in W$. 
\end{proof}

\subsection{The countable support iteration}

Recall that for any perfect subtree $T$ of $2^{<\omega}$, $\splt(T)$ is order-isomorphic to $2^{<\omega}$ in a canonical way, via a map $\eta_T \colon \splt(T) \to 2^{<\omega}$. This map induces a homeomorphism $\tilde \eta_T \colon [T] \to 2^\omega$ and note that the value of $\tilde \eta_T(x)$ depends continuously on $T$ and $x$. Whenever $\rho$ is a weight on $T$, $\eta_T$ also induces a weight $\tilde \rho$ on $2^{<\omega}$, so that whenever $S \leq_{\tilde \rho} 2^{<\omega}$, then $\eta_T^{-1}(S)$ generates a tree $S'$ with $S' \leq_{\rho} T$. 

Let $\langle \mathbb{P}_\beta, \dot{\mathbb{Q}}_\beta : \beta < \lambda \rangle$ be a countable support iteration where for each $\beta < \lambda$, $\Vdash_{\mathbb{P}_\beta} \dot{\mathbb{Q}}_\beta \in  \{ \mathbb{SP},\mathbb{S}\}$. We fix in this section a $\mathbb{P}_\lambda$ name $\dot y$ for an element of a Polish space $X$, a good master condition $\bar p \in \mathbb{P_\lambda}$  over a countable model $M_0$, where $\dot y, X \in M_0$, and let $C\subseteq \lambda$ be a countable set as in Lemma~\ref{lem:intermediategoodanalytic}. For every $\beta \in C$ and $\bar y \in [\bar p]\restriction (C \cap\beta)$, let us write $$T_{\bar y} = \{s \in 2^{<\omega} : \exists \bar x \in [\bar p] \left[ \bar x \restriction (C \cap\beta) = \bar y \wedge s \subseteq x(\beta)\right]  \}.$$ According to Lemma~\ref{lem:intermediategoodanalytic}, the map $\bar y \mapsto T_{\bar y}$ is a continuous function from $[\bar p]\restriction (C \cap\beta)$ to $\mathcal{T}$. Let $\alpha := \otp(C) < \omega_1$ as witnessed by an order-isomorphism $\iota \colon \alpha \to C$. Then we define the homeomorphism $\Phi\colon [\bar p] \restriction C \to (2^\omega)^\alpha$ so that for every $\bar y \in [\bar p] \restriction C$ and every $\delta < \alpha$, $$\Phi(\bar y) \restriction (\delta + 1) = \Phi(\bar y) \restriction \delta^\frown \tilde \eta_{T_{\bar y \restriction \iota(\delta)}}(y(\iota(\delta))).$$

Note that for $\mathbb{P} \in \{ \mathbb{SP}, \mathbb{S}\}$, the map sending $T \in \mathbb{P}$ to the weight $\rho_T$ defined in Lemma~\ref{lem:spweighted} or Lemma~\ref{lem:sweighted} is a Borel function from $\mathbb{P}$ to the Polish space of partial functions from $(2^{<\omega})^2$ to $[2^{<\omega}]^{<\omega}$. Thus for $\beta \in C$ and $\bar x \in [\bar p] \restriction (C \cap\beta)$, letting $\rho_{\bar x} := \rho_{T_{\bar x}}$, we get that $\bar x \mapsto \rho_{\bar x}$ is a Borel function on $[\bar p] \restriction (C \cap \beta)$. For each $\delta < \alpha$ and $\bar y \in (2^\omega)^\delta$, we may then define $\tilde \rho_{\bar y}$ a weight on $2^{<\omega}$, induced by $\rho_{\bar x}$ and $\eta_{T_{\bar x}}$, where $\bar x = \Phi^{-1}(\bar y^\frown \bar z) \restriction \beta$ for arbitrary, equivalently for every, $\bar z \in (2^\omega)^{\alpha \setminus \delta}$. The map sending $\bar y \in (2^\omega)^\delta$ to $\tilde \rho_{\bar y}$ is then Borel as well.

\begin{lemma}
	\label{lem:mCgforiteration}
	Let $M_1$ be a countable elementary model with $M_0, \bar p, \mathbb{P}_\lambda \in M_1$ and let $\bar s \in \bigotimes_{i < \alpha} 2^{<\omega}$. Then there is $\bar q \leq \bar p$, a good master condition over $M_0$, so that \begin{multline*}
	\forall \bar x_0, \dots, \bar x_{n-1} \in [\bar q] \big(\Phi(\bar x_0 \restriction C), \dots, \Phi(\bar x_{n-1} \restriction C) \in (2^\omega)^\alpha \cap [\bar s] \\ \text{ are strongly }\langle 2^\omega : i < \alpha \rangle \text{-mCg } \text{ over } M_1\big).
	\end{multline*}  
	Moreover $[\bar q] \restriction C$ is a closed subset of $[\bar p] \restriction C$ and $[\bar q] = ([\bar q] \restriction C) \times (2^\omega)^{\lambda \setminus C}$ (cf. Lemma~\ref{lem:goodmaster2}).
\end{lemma}

\begin{proof}
	We can assume without loss of generality that $\bar s = \emptyset$, i.e. $[\bar s] = (2^\omega)^\alpha$. It will be obvious that this assumption is inessential. 
	Next, let us introduce some notation. For any $\delta \leq \alpha$ and $\bar y_0, \dots, \bar y_{n-1} \in (2^\omega)^\delta$, recall from Definition~\ref{def:Delta} that $$ \Delta(\bar y_0, \dots, \bar y_{n-1}) := \{ \Delta_{\bar y_i, \bar y_j} : i \neq j < n \} \cup \{ 0, \delta\}.$$
	Let us write $$\tp(\bar y_0, \dots, \bar y_{n-1}) := (\langle \xi_l : l \leq k \rangle, \langle K_l : l < k \rangle, \langle U_i : i < n \rangle),$$ where $\{  \xi_0 < \dots < \xi_k\} = \Delta(\bar y_0, \dots, \bar y_{n-1})$, $K_l = \vert \{ \bar y_i \restriction [\xi_l, \xi_{l+1}) : i < n  \} \vert$ for every $l < k$ and $\langle U_i : i < n \rangle$ are the clopen subsets of $(2^\omega)^\delta$ of the form $U_i = [\bar s_i]$ for $\bar s_i \in \bigotimes_{\xi < \delta} 2^{<\omega}$ with $\dom(\bar s_i) = \{ \Delta_{\bar y_i, \bar y_j} : j < n, \bar y_j \neq \bar y_i \}$ and $\bar s_i$ minimal in the order of $\bigotimes_{\xi < \delta} 2^{<\omega}$ so that $$\bar y_i \in [\bar s_i] \text{ and } \forall j < n (\bar y_j \neq \bar y_i \rightarrow \bar y_j \notin [\bar s_i]),$$ for every $i < n$. 
	
	Note that for any $\delta_0 \leq \delta$, if $$\tp(\bar y_0 \restriction \delta_0, \dots, \bar y_{n-1} \restriction \delta_0) := (\langle \eta_l : l \leq k' \rangle, \langle M_l : l < k' \rangle, \langle V_i : i < n \rangle),$$ then $V_i = U_i \restriction \delta_0$ for every $i < n$. Moreover, for any $\bar y'_0, \dots, \bar y'_{n-1} \in (2^\omega)^\delta$ with $$\tp(\bar y'_0, \dots, \bar y'_{n-1}) = (\langle \xi_l : l \leq k \rangle, \langle K_l : l < k \rangle, \langle U_i : i < n \rangle),$$ we have that $$\tp(\bar y'_0 \restriction \delta_0, \dots, \bar y'_{n-1} \restriction \delta_0) = (\langle \eta_l : l \leq k' \rangle, \langle M_l : l < k' \rangle, \langle V_i : i < n \rangle).$$
	
	Any $\bar y_0, \dots, \bar y_{n-1}$, with $\tp(\bar y_0, \dots, \bar y_{n-1}) := (\langle \xi_l : l \leq k \rangle, \langle K_l : l < k \rangle, \langle U_i : i < n \rangle)$, that are $\langle 2^\omega : i < \delta \rangle$-mutually Cohen generic over $M_1$ as witnessed by $\xi_0 < \dots < \xi_k$, induce a $\prod_{l < k} (\bigotimes_{\xi \in [\xi_l,\xi_{l+1})} 2^{<\omega})^{K_l}$-generic and vice-versa. Thus whenever $\tau$ is a $\prod_{l < k} (\bigotimes_{\xi \in [\xi_l,\xi_{l+1})} 2^{<\omega})^{K_l}$-name, we may write $\tau[\bar y_0, \dots, \bar y_{n-1}]$ for the evaluation of $\tau$ via the induced generic. It will not matter in what particular way we define the $\prod_{l < k} (\bigotimes_{\xi \in [\xi_l,\xi_{l+1})} 2^{<\omega})^{K_l}$-generic from given $\bar y_0, \dots, \bar y_{n-1}$. We may stipulate for instance, that the generic induced by $\bar y_0, \dots, \bar y_{n-1}$ is $\langle \bar z_{l,j} : l < k, j < K_l \rangle$, where for each fixed $l < k$, $\langle \bar z_{l,j} : l < k, j < K_l \rangle$ enumerates $\{ \bar y_i \restriction [\xi_l, \xi_{l+1}) : i < n  \}$ in lexicographic order. 
	
	Let us get to the bulk of the proof. We will define a finite support iteration $\langle \mathbb{R}_{\delta}, \dot{\mathbb{S}}_\delta : \delta \leq \alpha \rangle$ in $M_1$, together with, for each $\delta \leq \alpha$, an $\mathbb{R}_{\delta}$-name $\dot X_\delta$ for a closed subspace of $(2^{\omega})^\delta$, where $\Vdash_{\mathbb{R}_{\delta_1}} \dot X_{\delta_0} = \dot X_{\delta_1} \restriction \delta_0$ for every $\delta_0 < \delta_1 \leq \alpha$. This uniquely determines the limit steps of the construction. Additionally we will make the following inductive assumptions $(1)_\delta$ and $(2)_\delta$ for all $\delta \leq \alpha$ and any $\mathbb{R}_\delta$-generic $G$. Let $\bar y_0, \dots, \bar y_{n-1} \in \dot X_{\delta}[G]$ be arbitrary and $\tp(\bar y_0, \dots, \bar y_{n-1}) = (\langle \xi_l : l \leq k \rangle, \langle K_l : l < k \rangle, \langle U_i : i < n \rangle)$. Then
	
	\begin{enumerate}[$(1)_\delta$]
		\item $\bar y_0, \dots, \bar y_{n-1}$ are strongly $\langle 2^\omega : i < \delta \rangle$-mCg over $M_1$, 
	\end{enumerate}

	\begin{enumerate}[$(1)_\delta$]
		\setcounter{enumi}{1}
		\item and for any $\prod_{l < k} (\bigotimes_{\xi \in [\xi_l,\xi_{l+1})} 2^{<\omega})^{K_l}$-name $\dot D \in M_1$ for an open dense subset of a countable poset $\mathbb{Q} \in M_1$, \begin{multline*}
		\bigcap \Bigl\{ \dot D[\bar y'_0, \dots, \bar y'_{n-1}] : \bar y'_0, \dots, \bar y'_{n-1} \in X_\delta, \\ \tp(\bar y'_0, \dots, \bar y'_{n-1}) = (\langle \xi_l : l \leq k \rangle, \langle K_l : l < k \rangle, \langle U_i : i < n \rangle) \Bigr\}
		\end{multline*} is open dense in $\mathbb{Q}$.
	\end{enumerate}

	Having defined $\mathbb{R}_\delta$ and $\dot{X}_\delta$, for $\delta < \alpha$, we proceed as follows. Fix for now $G$ an $\mathbb{R}_\delta$-generic over $M_1$ and $X_\delta := \dot{X}_\delta[G]$. Then we define a forcing $\mathbb{S}_\delta \in M_1[G]$ which generically adds a continuous map $F \colon X_\delta \to \mathcal{T}$, so that for each $\bar y \in X_\delta$, $ S_{\bar y} := F(\bar y) \leq_{\tilde \rho_{\bar y}} 2^{<\omega}$. In $M_1[G][F]$, we then define $X_{\delta+1} \subseteq (2^\omega)^{\delta +1}$ to be $\{ \bar y^\frown z : \bar y \in X_\delta, z \in [S_{\bar y}] \}$. The definition of $\mathbb{S}_\delta$ is as follows.  
	
	Work in $M_1[G]$. Since the map $\bar y \in (2^\omega)^\delta \mapsto \tilde \rho_{\bar y}$ is Borel and an element of $M_1$ and by $(1)_\delta$ any $\bar y \in X_\delta$ is Cohen generic over $M_1$, it is continuous on $X_\delta$. Since $X_\delta$ is compact we find a single weight $\tilde \rho$ on $2^{<\omega}$, so that $\tilde \rho_{\bar y} \subseteq \tilde \rho$ for every $\bar y \in X_\delta$. Let $\{ O_s : s \in 2^{<\omega} \}$ be a basis of $X_\delta$ so that $O_s \subseteq O_t$ for $t \subseteq s$ and $O_s \cap O_t = \emptyset$ for $s \perp t$. This is possible since $X_\delta$ is homeomorphic to $2^\omega$. Let $\mathcal{FT}$ be the set of finite subtrees of $2^{<\omega}$. Then $\mathbb{S}_\delta$ consists of functions $h \colon 2^{\leq n} \to \mathcal{FT}$, for some $n \in \omega$, so that for every $s \subseteq t \in 2^{\leq n}$, $ (h(s) \unlhd_{\tilde \rho} h(t))$. The extension relation is defined by function extension. Note that $\mathbb{S}_\delta$ is indeed a forcing poset with trivial condition $\emptyset$. 
	
	Given $H$, an $\mathbb{S}_\delta$-generic over $M_1[G]$, we let $F \colon X_\delta \to \mathcal{T}$ be defined as $$F(\bar y) :=  \bigcup_{\substack{s \in 2^{<\omega}, \bar y \in O_s \\ h \in H}} h(s).$$
	
	\begin{claim}
		For every $\bar y \in X_\delta$, $F(\bar y) = S_{\bar y} \leq_{\tilde \rho} 2^{<\omega}$, in particular $S_{\bar y} \leq_{\tilde \rho_{\bar y}} 2^{<\omega}$. For any $\bar y_0, \bar y_1 \in X_\delta$, $[S_{\bar y_0}] \cap [S_{\bar y_1}] \neq \emptyset$. Any $z_0, \dots, z_{n-1} \in \bigcup_{\bar y \in X_\delta}[S_{\bar y}]$ are $2^\omega$-mutually Cohen generic over $M_1[G]$. And for any countable poset $\mathbb{Q} \in M_1$, any $m \in \omega$ and any dense open $E \subseteq (2^{<\omega})^n \times \mathbb{Q}$ in $M_1[G]$, there is $r \in \mathbb{Q}$ and $m_0 \geq m$ so that for any  $z_0, \dots, z_{n-1} \in \bigcup_{\bar y \in X_\delta}[S_{\bar y}]$ where $z_0 \restriction m, \dots, z_{n-1} \restriction m$ are pairwise distinct, $( (z_0 \restriction m_0, \dots, z_{n-1} \restriction m_0), r) \in E$.  
	\end{claim}
	
	\begin{proof}
		We will make a genericity argument over $M_1[G]$. Let $h \in \mathbb{S}_\delta$ be arbitrary. Then it is easy to find $h' \leq h$, say with $\dom(h') = 2^{\leq a_0}$, so that for every $s \in 2^{a_0}$ and every $t \in \term(h(s))$, $\vert t \vert \geq m$. For the first claim, it suffices through Lemma~\ref{lem:basicweight} to find $h'' \leq h'$, say with $\dom(h'') = 2^{\leq a_1}$, $a_0 < a_1$, so that for every $s \in 2^{a_0}$ and $t \in 2^{a_1}$, with $s \subseteq t$, $h''(s) \lhd_{\tilde \rho} h''(t)$. Finding $h''$ so that additionally $\term(h''(t_0)) \cap \term(h''(t_1)) = \emptyset$ for every $t_0 \neq t_1 \in 2^{a_1}$ proves the second claim. For the last two claims, given a fixed dense open subset $E \subseteq (2^{<\omega})^n \times \mathbb{Q}$ in $M_1[G]$, it suffices to find $r \in \mathbb{Q}$ and to ensure that for any pairwise distinct $s_0, \dots, s_{n-1} \in \bigcup_{s \in 2^{a_0}} \term(h''(s))$ and $t_0 \supseteq s_0, \dots, t_{n-1} \supseteq s_{n-1}$ with $t_0, \dots, t_{n-1} \in \bigcup_{t \in 2^{a_1}} \term(h''(t))$, $((t_0, \dots, t_{n-1}), r) \in E$. Then we may put $m_0 = \max \{ \vert t \vert : t \in \bigcup_{s \in 2^{a_1}} \term(h''(s)) \}$. We may also assume wlog that $\mathbb{Q} = 2^{<\omega}$. 
		
		To find such $h''$ we apply Lemma~\ref{lem:wghtdenseset} as in the proof of Proposition~\ref{prop:weightedmain}. More precisely, for every $s \in 2^{a_0}$, we find $T^0_s, T^1_s \rhd_{\tilde \rho} h'(s)$, and we find $T \subseteq 2^{<\omega}$ finite, so that for any pairwise distinct $s_0, \dots, s_{n-1} \in \bigcup_{s \in 2^{a_0}} \term(h'(s))$, any $t_0 \supseteq s_0, \dots, t_{n-1} \supseteq s_{n-1}$ with $t_0, \dots, t_{n-1} \in \bigcup_{s \in 2^{a_0}, i \in 2} \term(T^i_s)$ and any $\sigma \in \term(T)$, $((t_0, \dots, t_{n-1}), \sigma ) \in E$ and $\term (T_s^i) \cap \term (T_{t}^j) = \emptyset$ for every $i, j \in 2$, $s, t \in 2^{a_0}$. Then simply define $h'' \leq h'$ with $\dom(h'') = 2^{a_0 +1}$, where $h''(s^\frown i) = T_s^i$ for $s \in 2^{a_0}$, $i \in 2$. 
	\end{proof}
	
	The function $F$ is obviously continuous and $X_{\delta +1 }$ is a closed subset of $(2^\omega)^{\delta +1}$, with $X_{\delta +1} \restriction \delta_0 = (X_{\delta +1} \restriction \delta) \restriction \delta_0 = X_\delta \restriction \delta_0 = X_{\delta_0}$ for every $\delta_0 < \delta + 1$.
	
	\begin{proof}[Proof of $(1)_{\delta +1}$, $(2)_{\delta +1}$]
		Let $G$ be $\mathbb{R}_{\delta +1}$ generic over $M_1$ and $\bar y_0, \dots, \bar y_{n-1} \in \dot X_{\delta +1}[G] = X_{\delta +1}$ be arbitrary. By the inductive assumption we have that $\bar y_0 \restriction \delta, \dots, \bar y_{n-1} \restriction \delta$ are strongly $\langle 2^\omega : i < \delta \rangle$-mCg over $M_1$. By the above claim, whenever $\bar y_i \restriction \delta \neq \bar y_j \restriction \delta$, then $\bar y_i(\delta) \neq \bar y_j(\delta)$. Thus, for $(1)_{\delta +1}$, we only need to show that $\bar y_0, \dots, \bar y_{n-1}$ are mCg. Let $\tp(\bar y_0, \dots, \bar y_{n-1}) = (\langle \xi_{l} : l < k \rangle,\langle K_{l} : l < k \rangle, \langle U_i : i < n \rangle )$, $\tp(\bar y_0 \restriction \delta , \dots, \bar y_{n-1} \restriction \delta) = (\langle \eta_l : l \leq k' \rangle, \langle M_l : l < k \rangle, \langle U_i\restriction \delta : i < n \rangle)$ and $n' = \vert \{ y_i(\delta) : i < n \} \vert = K_{k-1}$. Then we may view a dense open subset of $\prod_{l < k} (\bigotimes_{\xi \in [\xi_l,\xi_{l+1})} 2^{<\omega})^{K_l}$ as a $\prod_{l < k'} (\bigotimes_{\xi \in [\eta_l,\eta_{l+1})} 2^{<\omega})^{M_l}$-name for a dense open subset of $(2^{<\omega})^{n'}$. To this end, let $\dot D \in M_1$ be a $\prod_{l < k'} (\bigotimes_{\xi \in [\eta_l,\eta_{l+1})} 2^{<\omega})^{M_l}$ name for a dense open subset of $(2^{<\omega})^{n'}$. Then we have, by $(2)_\delta$, that \begin{multline*}
		\tilde D = \bigcap \Bigl\{ \dot D[\bar y'_0, \dots, \bar y'_{n-1}] : \bar y'_0, \dots, \bar y'_{n-1} \in X_\delta, \\ \tp(\bar y'_0, \dots, \bar y'_{n-1}) = (\langle \eta_l : l \leq k' \rangle, \langle M_l : l < k' \rangle, \langle U_i \restriction \delta : i < n \rangle) \Bigr\}
		\end{multline*} is a dense open subset of $(2^{<\omega})^{n'}$ and $\tilde D \in M_1[G\restriction \delta]$. By the above claim, $y_0(\delta), \dots, y_{n-1}(\delta)$ are mCg over $M_1[G\restriction \delta]$ in $2^\omega$. Altogether, this shows that $\bar y_0, \dots, \bar y_{n-1}$ are $\langle 2^\omega : i < \delta +1 \rangle$-mCg over $M_1$.
		
		For $(2)_{\delta+1}$, let $\dot D \in M_1$ now be a $\prod_{l < k} (\bigotimes_{\xi \in [\xi_l,\xi_{l+1})} 2^{<\omega})^{K_l}$-name for a dense open subset of $\mathbb{Q}$.  Consider a name $\dot E$ in $M_1$ for the dense open subset of $(2^{<\omega})^{n'} \times \mathbb{Q}$, where for any $\bar y'_0, \dots, \bar y'_{n-1} \in X_{\delta}$, with $\tp(\bar y'_0, \dots, \bar y'_{n-1}) = (\langle \eta_l : l \leq k' \rangle, \langle M_l : l < k \rangle, \langle U_i\restriction \delta : i < n \rangle)$, \begin{multline*}
		\dot E[\bar y'_0 , \dots, \bar y'_{n-1}] = \{ (\bar t,r) : M_1[\bar y'_0, \dots, \bar y'_{n-1}] \models \\\bar t \Vdash r \in \dot D[\bar y'_0, \dots, \bar y'_{n-1}][\dot z_0, \dots, \dot z_{n'-1}]\},
		\end{multline*} 
		
		where $(\dot z_0, \dots, \dot z_{n'-1})$ is a name for the $(2^{<\omega})^{n'}$-generic. By $(2)_\delta$, we have that \begin{multline*}
		\tilde E = \bigcap \Bigl\{ \dot E[\bar y'_0, \dots, \bar y'_{n-1}] : \bar y'_0, \dots, \bar y'_{n-1} \in X_{\delta}, \\ \tp(\bar y'_0, \dots, \bar y'_{n-1}) = (\langle \eta_l : l \leq k' \rangle, \langle M_l : l < k \rangle, \langle U_i\restriction \delta : i < n \rangle) \Bigr\}
		\end{multline*}
		is a dense open subset of $(2^{<\omega})^{n'} \times \mathbb{Q}$ and $\tilde E \in M_1[G\restriction \delta]$. Let $m \in \omega$ be large enough so that for any $i,j < n$, if $U_i \neq U_j$, then $\forall \bar y'_i \in U_i \cap X_{\delta+1}, \bar y'_j \in U_j \cap X_{\delta +1} (y'_i(\delta) \restriction m \neq y'_j(\delta)\restriction m)$. To see that such $m$ exists, note that if $U_i \neq U_j$, then $U_i \cap X_{\delta+1}$ and $U_j \cap X_{\delta +1}$ are disjoint compact subsets of $X_{\delta +1}$. By the claim, there is $r \in \mathbb{Q}$ and $m_0 \geq m$ so that for any $z_0, \dots, z_{n'-1} \in \bigcup_{\bar y \in X_\delta}[S_{\bar y}]$, if $z_0 \restriction m, \dots, z_{n'-1} \restriction m$ are pairwise different, then $((z_0 \restriction m_0, \dots, z_{n'-1} \restriction m_0), r ) \in \tilde E$. Altogether we find that \begin{multline*}
		r \in \bigcap \Bigl\{ \dot D[\bar y'_0, \dots, \bar y'_{n-1}] : \bar y'_0, \dots, \bar y'_{n-1} \in X_{\delta+1}, \\ \tp(\bar y'_0, \dots, \bar y'_{n-1}) = (\langle \xi_l : l \leq k \rangle, \langle K_l : l < k \rangle, \langle U_i : i < n \rangle) \Bigr\}.
		\end{multline*}
		Of course the same argument can be carried out below any condition in $\mathbb{Q}$, showing that this set is dense. That it is open is also clear since it is the intersection of open subsets of a partial order.  
	\end{proof}
	
	Now let $\delta \leq \alpha$ be a limit ordinal. 
	
	\begin{proof}[Proof of $(1)_\delta$ and $(2)_{\delta}$.]
		Let $G$ be $\mathbb{R}_{\delta}$-generic over $M_1$, $\bar y_0, \dots, \bar y_{n-1} \in \dot X_\delta[G] = X_\delta$, this time wlog pairwise distinct, and $\tp(\bar y_0, \dots, \bar y_{n-1}) = (\langle \xi_l : l \leq k \rangle, \langle K_l : l < k \rangle, \langle U_i : i < n \rangle)$. We will make a genericity argument over $M_1$ to show $(1)_\delta$ and $(2)_\delta$. To this end, let $D_0 \subseteq \prod_{l < k} (\bigotimes_{\xi \in [\xi_l,\xi_{l+1})} 2^{<\omega})^{K_l}$ be dense open, $D_0 \in M_1$, and let $\dot D_1 \in M_1$ be a $\prod_{l < k} (\bigotimes_{\xi \in [\xi_l,\xi_{l+1})} 2^{<\omega})^{K_l}$-name for a dense open subset of $\mathbb{Q}$. Then consider the dense open subset $D_2 \subseteq  \prod_{l < k} (\bigotimes_{\xi \in [\xi_l,\xi_{l+1})} 2^{<\omega})^{K_l} \times \mathbb{Q}$ in $M_1$, where $$D_2 = \{ (r_0, r_1) : r_0 \in D_0 \wedge r_0 \Vdash r_1 \in \dot D_1 \}.$$ Also let $\bar h_0 \in G$ be an arbitrary condition so that $$\bar h_0 \Vdash \forall i < n (U_i \cap \dot X_\delta \neq \emptyset).$$
		Then there is $\delta_0 < \delta$ so that $\supp(\bar h_0), \xi_{k-1}+1 \subseteq \delta_0$. We may equally well view $D_2$ as a $\prod_{l < k-1} (\bigotimes_{\xi \in [\xi_l,\xi_{l+1})} 2^{<\omega})^{K_l} \times (\bigotimes_{\xi \in [\xi_{k-1},\delta_0)} 2^{<\omega})^{K_{k-1}}$-name $\dot E \in M_1$ for a dense open subset $$E \subseteq (\bigotimes_{\xi \in [\delta_0,\xi_k)} 2^{<\omega})^{K_{k-1}} \times \mathbb{Q}= (\bigotimes_{\xi \in [\delta_0,\delta)} 2^{<\omega})^{n} \times \mathbb{Q}.$$
		
		We follow again from $(2)_{\delta_0}$, that the set $\tilde E \in M_1[G \cap \mathbb{R}_{\delta_0}]$, where  \begin{multline*}
		\tilde E = \bigcap \Bigl\{ \dot E[\bar y'_0, \dots, \bar y'_{n-1}] : \bar y'_0, \dots, \bar y'_{n-1} \in X_{\delta_0}, \\ \tp(\bar y'_0, \dots, \bar y'_{n-1}) = (\langle \xi_0 < \dots < \xi_{k-1} < \delta_0 \rangle, \langle K_l : l < k \rangle, \langle U_i \restriction \delta_0 : i < n \rangle) \Bigr\},
		\end{multline*} is dense open. Let $((\bar t_0, \dots, \bar t_{n-1}), r) \in \tilde E$ be arbitrary and $\bar h_1 \in G \cap \mathbb{R}_{\delta_0}$, $\bar h_1 \leq \bar h_0$, so that $\bar h_1 \Vdash ((\bar t_0, \dots, \bar t_{n-1}), r) \in \tilde E$.
		
		Let us show by induction on $\xi \in [\delta_0, \delta)$, $\xi > \sup \left( \bigcup_{i < n} \dom(\bar t_i) \right)$, that there is a condition $\bar h_2 \in \mathbb{R}_\xi,$ $\bar h_2 \leq \bar h_1$, so that \begin{multline*}
		\bar h_2 \Vdash \forall \bar y'_0, \dots, \bar y'_{n-1} \in \dot X_\delta \big(\tp(\bar y'_0, \dots, \bar y'_{n-1}) = (\langle \xi_l : l \leq k \rangle, \langle K_l : l < k \rangle, \langle U_i : i < n \rangle) \\ \rightarrow  \bar y'_0 \in [\bar t_0] \wedge \dots \wedge \bar y'_{n-1} \in [\bar t_{n-1}] \big)
		\end{multline*} 
		
		and in particular, if $\bar h_2 \in G$, then for all $\bar y'_0, \dots, \bar y'_{n-1} \in X_\delta$ with $\tp(\bar y'_0, \dots, \bar y'_{n-1}) = (\langle \xi_l : l \leq k \rangle, \langle K_l : l < k \rangle, \langle U_i : i < n \rangle)$, the generic corresponding to $\bar y'_0, \dots, \bar y'_{n-1}$ hits $D_0$, and $r \in \dot D_1[\bar y'_0, \dots, \bar y'_{n-1}]$. Since $\bar h_0 \in G$ was arbitrary, genericity finishes the argument. 
		
		The limit step of the induction follows directly from the earlier steps since if $\dom (\bar t_i) \subseteq \xi$, with $\xi$ limit, then there is $\eta < \xi$ so that $\dom (\bar t_i) \subseteq \eta$. So let us consider step $\xi +1$. Then there is, by the inductive assumption, $\bar h'_2 \in \mathbb{R}_\xi$, $\bar h'_2 \leq \bar h_1$, so that 
		
		\begin{multline*}
		\bar h'_2 \Vdash \forall \bar y'_0, \dots, \bar y'_{n-1} \in \dot X_\delta \big(\tp(\bar y'_0, \dots, \bar y'_{n-1}) = (\langle \xi_l : l \leq k \rangle, \langle K_l : l < k \rangle, \langle U_i : i < n \rangle) \\ \rightarrow ( \bar y'_0 \in [\bar t_0 \restriction \xi] \wedge \dots \wedge \bar y'_{n-1} \in [\bar t_{n-1} \restriction \xi] \big).
		\end{multline*}
		
		Now extend $\bar h'_2$ to $\bar h''_2$ in $\mathbb{R}_\xi$, so that there is $m \in \omega$ such that for every $s \in 2^{m}$ and every $i < n$, either $\bar h''_2 \Vdash \dot O_{s} \subseteq U_i \restriction \xi$ or $\bar h''_2 \Vdash \dot O_{s} \cap (U_i \restriction \xi) = \emptyset$, where $\langle \dot O_s : s \in 2^{<\omega} \rangle$ is a name for the base of $\dot X_\xi$ used to define $\dot{\mathbb{S}}_\xi$. The reason why this is possible, is that in any extension by $\mathbb{R}_\xi$ and for every $i < n$, by compactness of $X_\xi \cap (U_i \restriction \xi)$, there is a finite set $a \subseteq 2^{<\omega}$ so that $X_\xi \cap (U_i \restriction \xi) = \bigcup_{s \in a} O_s$.  Let us define $h \colon 2^{\leq m} \to \mathcal{FT}$, where $$h(s) =\begin{cases}
		\emptyset & \text{if } \forall i < n ( \bar h''_2 \Vdash \dot O_{s}  \cap U_i \restriction \xi = \emptyset) \\
		\{ t \in 2^{<\omega} : t \subseteq t_i(\xi)  \} & \text{if } \bar h''_2 \Vdash \dot O_{s} \subseteq U_i \restriction \xi \text{ and } i < n.
		\end{cases} $$
		
		Note that $h$ is well-defined as $(U_i \restriction \xi) \cap (U_j \restriction \xi) = \emptyset$ for every $i \neq j < n$. Since $\emptyset \unlhd_{\rho} T$ and $T \unlhd_{\rho} T$ for any weight $\rho$ and any finite tree $T$, we have that $\bar h''_2 \Vdash h \in \dot{\mathbb{S}}_\xi$ and $\bar h_2 = \bar h''_2{}^\frown h \in \mathbb{R}_{\xi+1}$ is as required. \end{proof}	
	
	This finishes the definition of $\mathbb{R}_\alpha$ and $\dot X_\alpha$. Finally let $G$ be $\mathbb{R}_\alpha$-generic over $M_1$ and $X_\alpha = \dot X_\alpha[G]$. Now let us define $\bar q \leq \bar p$ recursively so that for every $\delta \leq \alpha$,
	$$\forall \bar x \in [\bar q] (\Phi(\bar x \restriction C) \restriction \delta \in X_\alpha \restriction \delta).$$ 
	
	If $\beta \notin C$ we let $\dot q(\beta)$ be a name for the trivial condition $2^{<\omega}$, say e.g. $\dot q(\beta) = \dot p(\beta)$. If $\beta \in C$, say $\beta = \iota(\delta)$, we define $\dot q(\beta)$ to be a name for the tree generated by $$ \eta^{-1}_{T_{\bar x_{G} \restriction (C \cap \beta) }} (S_{\bar y}),$$ where $\bar x_{G}$ is the generic sequence added by $\mathbb{P}_\lambda$ and $\bar y = \Phi(\bar x_G \restriction C) \restriction \delta$. This ensures that $\bar q \restriction \beta \Vdash \dot q(\beta) \in \mathbb{Q}_\beta \wedge \dot q(\beta) \leq \dot p(\beta)$. Inductively we see that $\bar q \restriction \beta^\frown \bar p \restriction (\lambda \setminus \beta) \Vdash \Phi(\bar x_G \restriction  C) \restriction \delta \in X_\alpha \restriction \delta$. Having defined $\bar q$, it is also easy to check that it is a good master condition over $M_0$, with $[\bar q] = \Phi^{-1}(X_\alpha) \times (2^{\omega})^{\lambda\setminus C}$. Since for every $\bar x \in [\bar q]$, $\Phi(\bar x \restriction C) \in X_\alpha$ and by $(1)_\alpha$, $\bar q$ is as required. 
\end{proof}

\begin{prop}
\label{prop:propheart}
	Let $E \subseteq [X]^{<\omega} $ be an analytic hypergraph on $X$, say $E$ is the projection of a closed set $F \subseteq [X]^{<\omega} \times \omega^\omega$, and let $f \colon [\bar p] \restriction C \to X$ be continuous so that $\bar p \Vdash \dot y = f(\bar x_G \restriction C)$ (cf. Lemma~\ref{lem:goodmaster2}). Then there is a good master condition $\bar q \leq \bar p$, with $[\bar q] \restriction C$ a closed subset of $[\bar p] \restriction C$ and $[\bar q] = ([\bar q] \restriction C) \times (2^\omega)^{\lambda \setminus C}$, a compact $E$-independent set $Y \subseteq X$, $N \in \omega$ and continuous functions $\phi \colon [\bar q] \restriction C \to [Y]^{<N}$, $w \colon [\bar q] \restriction C \to \omega^\omega $, so that 
	
	\begin{enumerate}[(i)]
		\item either $f''([\bar q] \restriction C) \subseteq Y$, thus $\bar q \Vdash \dot y \in Y$,
		\item or $\forall \bar x \in [\bar q]\restriction C ( (\phi(\bar x) \cup \{ f(\bar x) \}, w(\bar x)) \in F)$, thus $\bar q \Vdash {\{\dot y\} \cup Y}$ is not $E\text{-independent}$.
	\end{enumerate}
\end{prop}

\begin{proof}
 On $(2^\omega)^\alpha$ let us define the analytic hypergraph $\tilde E$, where $$\{ \bar y_0, \dots, \bar y_{n-1} \} \in \tilde E \leftrightarrow \{ f(\Phi^{-1}(\bar y_0), \dots,f(\Phi^{-1}(\bar y_{n-1})) \} \in E.$$
	By Main Lemma~\ref{lem:mainlemmainf}, there is a countable model $M$ and $\bar s \in \bigotimes_{i<\alpha} 2^{<\omega}$ so that either 
	\begin{enumerate}
		\item for any $\bar y_0,\dots, \bar y_{n-1} \in (2^\omega)^\alpha \cap [\bar s]$ that are strongly $\langle 2^\omega : i < \alpha \rangle$-mCg over $M$, $\{\bar y_0,\dots, \bar y_{n-1}\} \text{ is } E\text{-independent}$,
	\end{enumerate}
	
	or for some $N\in \omega$,
	\begin{enumerate}
		\setcounter{enumi}{1}
		\item there are $\phi_0, \dots, \phi_{N-1} \colon (2^\omega)^\alpha \to (2^\omega)^\alpha$ continuous so that for any $\bar y_0,\dots, \bar y_{n-1} \in (2^\omega)^\alpha \cap [\bar s]$ that are strongly mCg over $M$,  $\{\phi_j(\bar y_i) : j < N, i<n\}$ is $E$-independent but $\{ \bar y_0\} \cup \{ \phi_j(\bar y_0) : j < N \} \in E.$
	\end{enumerate} 
	
Let $M_1$ be a countable elementary model with $M_0, M, \bar p, \mathbb{P}_\lambda \in M_{1}$ and apply Lemma~\ref{lem:mCgforiteration} to get the condition $\bar q \leq \bar p$. In case (1), let $Y := f''([\bar q] \restriction C)$. Then (i) is satisfied. To see that $Y$ is $E$-independent let $\bar x_0, \dots, \bar x_{n-1} \in [\bar q]$ be arbitrary and suppose that $\{{f(\bar x_0 \restriction C)}, \dots,f(\bar x_{n-1} \restriction C)  \} \in E$. By definition of $\tilde E$ this implies that $\{ {\Phi(\bar x_0 \restriction C)}, \dots, \Phi(\bar x_{n-1} \restriction C) \} \in \tilde E$ but this is a contradiction to (1) and the conclusion of Lemma~\ref{lem:mCgforiteration}. In case (2), by elementarity, the $\phi_j$ are in $M_1$ and there is a continuous function $\tilde w \in M_1$, with domain some dense $G_\delta$ subset of $(2^\omega)^\alpha$, so that $\bar s \Vdash (\{f(\bar z), \phi_{j}(\bar z ) : j < N \}, \tilde w(\bar z)) \in F$, where $\bar z$ is a name for the Cohen generic. Let $\phi(\bar x) = \{ f(\Phi^{-1}(\phi_j(\Phi(\bar x)))) : j < N \}$, $w(\bar x) = \tilde w(\Phi(\bar x))$ for $\bar x \in [\bar q] \restriction C$ and $Y := \bigcup_{\bar x \in [\bar q] \restriction C} \phi(\bar x)$. Since $\Phi(\bar x)$ is generic over $M_1$, we indeed have that $(\phi(\bar x), w(\bar x)) \in F$ for every $\bar x \in [\bar q] \restriction C$. Seeing that $Y$ is $E$-independent is as before. 
\end{proof}

\section{Main results and applications}

\begin{thm}\label{thm:mainmaintheorem}
	(V=L) Let $\mathbb{P}$ be a countable support iteration of Sacks or splitting forcing of arbitrary length. Let $X$ be a Polish space and $E \subseteq [X]^{<\omega} $ be an analytic hypergraph. Then there is a $\mathbf{\Delta}^1_2$ maximal $E$-independent set in $V^\mathbb{P}$. If $X = 2^\omega$ or $X = \omega^\omega$, $r \in 2^\omega$ and $E$ is $\Sigma^1_1(r)$, then we can find a $\Delta^1_2(r)$ such set. 
\end{thm}

\begin{proof}
		We will only concentrate on the case $X = 2^\omega$ since the rest follows easily from the fact that there is a Borel isomorphism from $2^\omega$ to any uncountable Polish space $X$, and if $X = \omega^\omega$ that isomorphism is (lightface) $\Delta^1_1$. If $X$ is countable, then the statement is trivial. Also, let us only consider splitting forcing. The proof for Sacks forcing is the same. 
		
First let us us mention some well-known facts and introduce some notation. Recall that a set $Y \subseteq 2^\omega$ is $\Sigma^1_2(x)$-definable if and only if it is $\Sigma_1(x)$-definable over $H(\omega_1)$ (see e.g. \cite[Lemma 25.25]{Jech2013}). Also recall that there is a $\Sigma^1_1$ set $A \subseteq 2^{\omega} \times 2^{\omega}$ that is universal for analytic sets, i.e. for every analytic $B \subseteq 2^{\omega}$, there is some $x\in 2^\omega$ so that $B = A_x$, where $A_x = \{ y \in 2^\omega: (x,y) \in A \}$. In the same way, there is a universal $\Pi^0_1$ set $F \subseteq 2^\omega \times [2^\omega]^{<\omega} \times \omega^\omega$ (\cite[22.3, 26.1]{Kechris1995}). For any $x \in 2^\omega$, let $E_x$ be the analytic hypergraph on $2^\omega$ consisting of $a \in [2^{\omega}]^{<\omega} \setminus \{\emptyset\}$ so that there is $b \in [A_x]^{<\omega}$ with $a \cup b \in E$. Then there is $y \in 2^\omega$ so that $E_x$ is the projection of $F_y$. Moreover, it is standard to note, from the way $A$ and $F$ are defined, that for every $x$, $y= e(x,r)$ for some fixed recursive function $e$. Whenever $\alpha < \omega_1$ and $Z \subseteq (2^\omega)^\alpha$ is closed, it can be coded naturally by the set $S \subseteq \bigotimes_{i < \alpha} 2^{<\omega}$, where $$S = \{ (\bar x \restriction a) \restriction n : \bar x \in Z, a \in [\alpha]^{<\omega}, n \in \omega\}$$ and we write $Z = Z_S$. Similarly, any continuous function $f \colon Z \to \omega^\omega$ can be coded by a function $\zeta \colon S \to \omega^{<\omega}$, where $$f(\bar x) = \bigcup_{\bar s \in S, \bar x \in [\bar s]} \zeta(\bar s)$$ and we write $f = f_\zeta$.  For any $\beta < \alpha$ and $\bar x \in Z \restriction \beta$, let us write $T_{\bar x, Z} = \{ s \in 2^{<\omega} : \exists \bar z \in Z (\bar z \restriction \delta = \bar x \wedge s\subseteq z(\delta))\}$. The set $\Psi_0$ of pairs $(\alpha,S)$, where $S$ codes a closed set $Z \subseteq (2^\omega)^\alpha$ so that for every $\beta < \alpha$ and $\bar x \in Z \restriction \beta$, $T_{\bar x,Z} \in \mathbb{SP}$ is then $\Delta_1$ over $H(\omega_1)$. This follows since the set of such $S$ is $\Pi^1_1$, seen as a subset of $\mathcal{P}(\bigotimes_{i < \alpha} 2^{<\omega})$, uniformly on $\alpha$. Similarly, the set $\Psi_1$ of triples $(\alpha, S, \zeta)$, where $(\alpha, S) \in \Psi_0$ and $\zeta$ codes a continuous function $f \colon Z_S \to \omega^\omega$, is $\Delta_1$. 
		
Now let $\langle \alpha_\xi, S_\xi, \zeta_\xi : \xi < \omega_1 \rangle$ be a $\Delta_1$-definable enumeration of all triples $(\alpha, S, \zeta)\in \Psi_1$. This is possible since we assume $V=L$ (cf. \cite[Theorem 25.26]{Jech2013}). Let us recursively construct a sequence $\langle x_\xi, y_\xi, T_\xi, \bar \eta_\xi, \theta_\xi : \xi < \omega_1 \rangle$, where for each $\xi < \omega_1$,

\begin{enumerate}
	\item $\bigcup_{\xi' < \xi } A_{x_{\xi'}}  = A_{y_\xi}$ and $A_{y_\xi} \cup A_{x_\xi}$ is $E$-independent, 
	\item $\bar \eta_\xi = \langle \eta_{\xi,j} : j <N \rangle$ for some $N \in \omega$,
	\item $T_\xi \subseteq S_\xi$, $(\alpha_\xi, T_\xi, \eta_{\xi,j}) \in \Psi_1$ for every $j < N$ and $(\alpha_\xi, T_\xi, \theta_{\xi}) \in \Psi_1$,
	\item either $\forall \bar x \in Z_{T_\xi} (f_{\zeta_\xi}(\bar x) \in A_{x_\xi})$ or $\forall \bar x \in Z_{T_\xi} \big(\forall n < N (f_{\eta_{\xi,n}}(\bar x) \in A_{x_\xi}) \wedge (\{ f_{\eta_{\xi,n}}(\bar x), f_{\zeta_\xi}(\bar x): n < N\},f_{\theta_\xi}(\bar x) ) \in F_{e(y_\xi,r)}\big)$,
\end{enumerate}

and $(x_\xi, y_\xi, T_\xi, \bar \eta_\xi, \theta_\xi)$ is $<_L$-least such that (1)-(4), where $<_L$ is the $\Delta_1$-good global well-order of $L$. That $<_L$ is $\Delta_1$-good means that for every $z \in L$, the set $\{ z' : z' <_L z \}$ is $\Delta_1(z)$ uniformly on the parameter $z$. In particular, quantifying over this set does not increase the complexity of a $\Sigma_n$-formula. Note that (1)-(4) are all $\Delta_1(r)$ in the given variables. E.g. the second part of (1) is uniformly $\Pi^1_1(r)$ in the variables $x_\xi, y_\xi$, similarly for (4).

\begin{claim}
For every $\xi < \omega_1$, $(x_\xi, y_\xi, T_\xi, \bar \eta_\xi, \theta_\xi)$ exists. 
\end{claim}
\begin{proof}
Assume we succeeded in constructing the sequence up to $\xi$. Then there is $y_\xi$ so that $\bigcup_{\xi' < \xi } A_{x_{\xi'}}  = A_{y_\xi}$. By Lemma~\ref{lem:gettingacondition}, there is a good master condition $\bar r \in \mathbb{P}_{\alpha_\xi}$ so that $[\bar r] \subseteq Z_{S_\xi}$, where $\mathbb{P}_{\alpha_\xi}$ is the $\alpha_\xi$-long csi of splitting forcing. Then $f_{\zeta_\xi}$ corresponds to a $\mathbb{P}_{\alpha_\xi}$-name $\dot y$ so that $\bar r \Vdash \dot y = f_{\zeta_\xi}(\bar x_G)$. Let $M_0$ be a countable elementary model with $\dot y,\mathbb{P}_{\alpha_\xi}, \bar r \in M_0$ and $\bar p \leq \bar r$ a good master condition over $M_0$. Let $C := \alpha_\xi$ and consider the results of the last section. By Proposition~\ref{prop:propheart} applied to $E_{y_\xi}$ and $X = 2^\omega$, there is $\bar q \leq \bar p$, a compact $E_{y_\xi}$ independent set $Y_\xi$, $N \in \omega$ and continuous functions $\phi \colon [\bar q] \to [Y_\xi]^{<N}$, $w \colon [\bar q] \to \omega^\omega$ such that 
	\begin{enumerate}[(i)]
		\item either $f_{\zeta_\xi}''([\bar q]) \subseteq Y_\xi$,
		\item or $\forall \bar x \in [\bar q] ( (\phi(\bar x) \cup \{ f_{\zeta_\xi}(\bar x) \}, w(\bar x)) \in F)$.
	\end{enumerate}

Let $x_\xi$, $T_\xi$, $\bar \eta_\xi = \langle \eta_{\xi,j} : j < N\rangle$ and $\theta_\xi$ be such that $A_{x_\xi} = Y_\xi$, $Z_{T_\xi} = [\bar q]$, $\{ f_{\eta_{\xi,j}}(\bar x) : j < N \} = \phi(\bar x)$ for every $\bar x \in [\bar q]$, and $f_{\theta_\xi} = w$. Then $(x_\xi, y_\xi, T_\xi, \bar \eta_\xi, \theta_\xi)$ is as required.
\end{proof}

Let $Y = \bigcup_{\xi< \omega_1} A_{x_\xi}$. Then $Y$ is $\Sigma_1(r)$-definable over $H(\omega_1)$, namely $x \in Y$ iff  there is a sequence $\langle x_\xi, y_\xi, T_\xi, \bar \eta_\xi, \theta_\xi : \xi \leq \alpha < \omega_1 \rangle$  so that for every $\xi \leq \alpha$, (1)-(4), for every $(x,y, T, \bar \eta, \theta) <_L (x_\xi,y_\xi, T_\xi, \bar \eta_\xi, \theta_\xi)$, not (1)-(4), and $x \in A_{x_\alpha}$.

\begin{claim}
In $V^{\mathbb{P}}$, the reinterpretation of $Y$ is maximal $E$-independent.  
\end{claim}

\begin{proof}
Let $\bar p \in \mathbb{P}$ and $\dot y \in M_0$ be a $\mathbb{P}$-name for an element of $2^\omega$, $M_0 \ni \mathbb{P}, \bar p$ a countable elementary model. Then let $\bar q \leq \bar p$ be a good master condition over $M_0$ and $C$ countable, $f \colon [\bar q] \restriction C \to 2^\omega$ continuous according to Lemma~\ref{lem:intermediategoodanalytic}. Now $(2^\omega)^{C}$ is canonically homeomorphic to $(2^\omega)^\alpha$, $\alpha = \otp(C)$,  via the map $\Phi \colon (2^\omega)^{C} \to (2^\omega)^\alpha$. Then we find some $\xi < \omega_1$ so that $\alpha_\xi = \alpha$, $\Phi''([\bar q] \restriction C) = Z_{S_\xi}$ and $f_{\zeta_\xi} \circ \Phi = f$. On the other hand, $\Phi^{-1} (Z_{T_\xi})$ is a subset of $[\bar q] \restriction C$ conforming to the assumptions of Lemma~\ref{lem:gettingacondition}. Thus we get $\bar r \leq \bar q$ so that $[\bar r] \restriction C \subseteq \Phi^{-1} (Z_{T_\xi})$. According to (4), either $\bar r \Vdash \dot y \in A_{x_{\xi}}$ or $\bar r \Vdash \{ \dot y\} \cup A_{x_\xi} \cup A_{y_\xi}$ is not $E$-independent. Thus we can not have that $\bar p \Vdash \dot y \notin Y \wedge \{\dot y\} \cup Y$ is $E$-independent. This finishes the proof of the claim, as $\bar p$ and $\dot y$ were arbitrary.
\end{proof}

To see that $Y$ is $\Delta^1_2(r)$ in $V^{\mathbb{P}}$ it suffices to observe that any $\Sigma^1_2(r)$ set that is maximal $E$-independent is already $\Pi^1_2(r)$. \end{proof}

A priori, Theorem~\ref{thm:mainmaintheorem} only works for hypergraphs that are defined in the ground model. But note that there is a universal analytic hypergraph on $2^\omega \times 2^\omega$, whereby we can follow the more general statement of Theorem~\ref{thm:maintheorem}.

\begin{thm}
After forcing with the $\omega_2$-length countable support iteration of $\mathbb{SP}$ over $L$, there is a $\Delta^1_2$ ultrafilter, a $\Pi^1_1$ maximal independent family and a $\Delta^1_2$ Hamel basis, and in particular, $\mathfrak{i}_{B} = \mathfrak{i}_{cl} = \mathfrak{u}_B = \mathfrak{u}_{cl} = \omega_1 < \mathfrak{r} = \mathfrak{i} = \mathfrak{u} = \omega_2$.
\end{thm}

\begin{proof}
Apply Theorem~\ref{thm:mainmaintheorem} to $E_u$, $E_i$ and $E_h$ from the introduction. To see that $\mathfrak{i}_{cl} = \mathfrak{u}_{cl} = \omega_1$ note that every analytic set is the union of $\mathfrak{d}$ many compact sets and that $\mathfrak{d} = \omega_1$, since $\mathbb{SP}$ is $\omega^\omega$-bounding.
\end{proof}

\begin{thm}\label{thm:finiteprod}
(V=L) Let $\mathbb{P}$ be either Sacks or splitting forcing and $k \in \omega$. Let $X$ be a Polish space and $E \subseteq [X]^{<\omega}$ be an analytic hypergraph. Then there is a $\mathbf{\Delta}^1_2$ maximal $E$-independent set in $V^{\mathbb{P}^k}$.
\end{thm}

\begin{proof}
This is similar to the proof of Theorem~\ref{thm:mainmaintheorem}, using Main Lemma~\ref{thm:mainlemma} and Proposition~\ref{prop:weightedmain} to get an analogue of Proposition~\ref{prop:propheart}.
\end{proof}

Lastly, we are going to prove Theorem~\ref{thm:icleqd}. 

\begin{lemma}\label{lem:sigmacompind}
Let $X \subseteq [\omega]^\omega$ be closed so that $\forall x,y \in X ( \vert x \cap y \vert = \omega)$. Then $X$ is $\sigma$-compact. 
\end{lemma}

\begin{proof}
If not, then by Hurewicz's Theorem (see \cite[7.10]{Kechris1995}), there is a superperfect tree $T \subseteq \omega^{<\omega}$ so that $[T] \subseteq X$, identifying elements of $[\omega]^\omega$ with their increasing enumeration, as usual.  But then it is easy to recursively construct increasing sequences $\langle s_n : n \in \omega \rangle$, $\langle t_n : n \in \omega \rangle$ in $T$ so that $s_0 = t_0 = \operatorname{stem}(T)$, for every $n \in \omega$, $t_n$ and $s_n$ are infinite-splitting nodes in $T$ and $s_{2n+1}(\vert s_{2n} \vert ) > t_{2n+1}(\vert t_{2n+1} \vert -1)$, $t_{2n+2}(\vert t_{2n} \vert ) > s_{2n+1}(\vert s_{2n+1} \vert -1)$. Then, letting $x = \bigcup_{n \in \omega} s_n$ and $y = \bigcup_{n \in \omega} t_n$, $x \cap y \subseteq \vert s_0 \vert$, viewing $x,y$ as elements of $[\omega]^\omega$. This contradicts that $x,y \in X$. 
\end{proof}

The proof of Theorem~\ref{thm:icleqd} is a modification of Shelah's proof that $\mathfrak{d} \leq \mathfrak{i}$.

\begin{proof}[Proof of Theorem~\ref{thm:icleqd}]
Let $\seq{C_\alpha: \alpha < \kappa}$ be compact independent families so that $ \mathcal{I} = \bigcup_{\alpha < \kappa} C_\alpha$ is maximal independent and $\kappa < \mathfrak d$ and assume without loss of generality that $\{C_\alpha: \alpha < \kappa\}$ is closed under finite unions. Here, we will identify elements of $[\omega]^\omega$ with their characteristic function in $2^\omega$ at several places and it should always be clear from context which representation we consider at the moment. 

\begin{claim}
There are $\seq{x_n : n \in \omega}$ pairwise distinct in $\mathcal{I}$ so that $\{x_n : n \in \omega \} \cap C_\alpha$ is finite for every $\alpha < \kappa$. 
\end{claim}

\begin{proof}
The closure of $\mathcal{I}$ is not independent. Thus there is $x \in \bar{\mathcal{I}} \setminus \mathcal{I}$. Now we pick $\seq{x_n : n \in \omega} \subseteq \mathcal{I}$ converging to $x$. Since $C_\alpha$ is closed, whenever for infinitely many $n$, $x_n \in C_\alpha$, then also $x \in C_\alpha$ which is impossible. 
\end{proof}

Fix a sequence $\seq{x_n : n \in \omega}$ as above. And let $a_\alpha = \{ n \in \omega : x_n \in C_\alpha \} \in [\omega]^{<\omega}$. We will say that $x$ is a Boolean combination of a set $X \subseteq [\omega]^\omega$, if there are finite disjoint $Y,Z \subseteq X$ so that $x = (\bigcap_{y \in Y} y) \cap (\bigcap_{z \in Z} \omega \setminus z)$.

\begin{claim}
 For any $\alpha < \kappa$ there is $f_{\alpha} \colon \omega \to \omega$ so that for any $K \in [C_\alpha \setminus \{x_n : n \in a_\alpha \} ]^{<\omega}$, for all but finitely many $k \in \omega$ and any Boolean combination $x$ of $K \cup \{x_0, \dots, x_k \}$, $x \cap [k,f_{\alpha}(k)) \neq \emptyset$.
\end{claim}

\begin{proof}
We define $f_{\alpha}(k)$ as follows. For every $l \leq k$, we define a collection of basic open subsets of $(2^\omega)^{l}$, $ \mathcal{O}_{0,l} := \{[\bar s] : \bar s \in (2^{<\omega})^l \wedge \forall i < l (\vert s_i\vert> k) \wedge (\exists i < l, n \in a_\alpha (s_i \subseteq x_n) \vee \exists i<j<l (s_i \not\perp s_j))\}$. Further we call any $[\bar s] \notin \mathcal{O}_{0,l}$ good if for any $F,G \subseteq l$ with $F \cap G = \emptyset$ and for any Boolean combination $x$ of $\{ x_0,\dots x_k \}$, there is $ k'> k$ so that for every $i \in F$, $s_i(k') = 1$, for every $ i \in G$, $s_i(k')=0$ and $x(k')=1$. Let $\mathcal{O}_{1,l}$ be the collection of all good $[\bar s]$. We see that $\bigcup_{l \leq k}(\mathcal{O}_{0,l} \cup \mathcal{O}_{1,l})$ is an open cover of $C_\alpha \cup (C_\alpha)^2 \cup \dots \cup(C_\alpha)^k$. Thus it has a finite subcover $\mathcal{O}'$. Now let $f_{\alpha}(k) := \max \{ \vert t \vert : \exists [\bar s] \in \mathcal{O}' \exists i < k (t = s_i) \}$. 

Now we want to show that $f_{\alpha}$ is as required. Let $(y_0,\dots,y_{l-1}) \in (C_\alpha \setminus \{ x_n : n \in a_\alpha \} )^{l}$ be arbitrary, $y_0, \dots, y_{l-1}$ pairwise distinct and $k \geq l$ so that $y_i \restriction k \neq x_n \restriction k$ for all $i < l$, $n \in a_\alpha$ and $y_i \restriction k \neq y_j \restriction k$ for all $i < j < l$. In the definition of $f_{\alpha}(k)$, we have the finite cover $\mathcal{O}'$ of $(C_\alpha)^l$ and thus $(y_0,\dots,y_{l-1}) \in [\bar s]$ for some $[\bar s] \in \mathcal{O}'$. We see that $[\bar s] \in \mathcal{O}_{0,l}$ is impossible as we chose $k$ large enough so that for no $i < l, n \in a_\alpha$, $s_i \subseteq x_n$ and for every $i < j<l$, $s_i \perp s_j$. Thus $[\bar s] \in \mathcal{O}_{1,l}$. But then, by the definition of $\mathcal{O}_{1,l}$, $f_{\alpha}(k)$ is as required. 
\end{proof}

As $\kappa < \mathfrak d$ we find $f \in \omega^\omega$ so that $f$ is unbounded over $\{f_{\alpha} : \alpha < \kappa\}$. Let $x_n^0 := x_n$ and $x_n^1 := \omega \setminus x_n$ for every $n \in \omega$. For any $g \in 2^\omega$ and $n \in \omega$ we define $y_{n,g} := \bigcap_{m\leq n} x_m^{g(m)}$. Further define $y_g = \bigcup_{n \in \omega} y_{n,g} \cap f(n)$. Note $y_{n,g} \subseteq y_{m,g}$ for $m\leq n$ and that $y_g \subseteq^* y_{n,g}$ for all $n \in \omega$. 

\begin{claim}
For any $g \in 2^\omega$, $y_g$ has infinite intersection with any Boolean combination of $\bigcup_{\alpha < \kappa} C_\alpha \setminus \{x_n : n \in \omega\}$.
\end{claim}

\begin{proof}
Let $\{y_0,\dots,y_{l-1}\} \in [C_\alpha\setminus \{x_n : n \in a_\alpha\}]^l$ for some $l \in \omega$, $\alpha < \kappa$ be arbitrary. Here, recall that $\{C_\alpha: \alpha < \kappa\}$ is closed under finite unions. We have that there is some $k_0 \in \omega$ so that for every $k \geq k_0$, any Boolean combination $y$ of $\{y_0,\dots,y_{l-1}\}$ and $x$ of  $\{x_n : n \leq k \}$, $x\cap y \cap [k,f_{\alpha}(k)) \neq \emptyset$. Let $y$ be an arbitrary Boolean combination of $\{y_0,\dots,y_{l-1}\}$ and $m \in \omega$. Then there is $k > m,k_0$ so that $f(k) > f_{\alpha}(k)$. But then we have that $y_{k,g}$ is a Boolean combination of $\{x_0, \dots, x_{k} \}$ and thus $y_{k,g} \cap y \cap [k,f(k)) \neq \emptyset$. In particular, this shows that $y \cap y_g \not\subseteq m$ and unfixing $m$, $\vert y \cap y_g \vert = \omega$.
\end{proof}

Now let $Q_0,Q_1$ be disjoint countable dense subsets of $2^{\omega}$. We see that $\card{y_g \cap y_h} < \omega$ for $h \neq g \in 2^\omega$. Thus the family $\{ y_g : g \in Q_0 \cup Q_1 \}$ is countable almost disjoint and we can find $y'_g =^* y_g$, for every $g \in Q_0 \cup Q_1$, so that $\{y'_g : g \in Q_0 \cup Q_1 \}$ is pairwise disjoint.
Let $y = \bigcup_{g \in Q_0} y'_g$. We claim that any Boolean combination $x$ of sets in $\mathcal{I}$ has infinite intersection with $y$ and $\omega \setminus y$. To see this, assume without loss of generality that $x$ is of the from $\tilde x \cap x_0^{g(0)} \cap \dots \cap x_k^{g(k)}$, where $\tilde x$ is a Boolean combination of sets in $\mathcal{I} \setminus \{ x_n : n \in \omega\}$ and $g \in 2^\omega$. As $Q_0$ is dense there is some $h \in Q_0$ such that $h \restriction (k+1) = g \restriction (k+1)$. Thus we have that $y'_h \subseteq^* x_0^{g(0)} \cap \dots \cap x_k^{g(k)}$ but also $y'_h \cap \tilde x$ is infinite by the claim above. In particular we have that $y \cap x$ is infinite. The complement of $y$ is handled by replacing $Q_0$ with $Q_1$. 
We now have a contradiction to $\mathcal{I}$ being maximal. 
\end{proof}

\section{Concluding remarks}

Our focus in this paper was on Sacks and splitting forcing but it is clear that the method presented is more general. We mostly used that our forcing has Axiom A with continuous reading of names and that it is a weighted tree forcing (Definition~\ref{def:weightedtreeforcing}), both in a definable way. For instance, the more general versions of splitting forcing given by Shelah in \cite{Shelah1992} fall into this class. It would be interesting to know for what other tree forcings Theorem~\ref{thm:mainmaintheorem} holds true. In \cite{Schrittesser2018}, the authors showed that after adding a single Miller real over $L$, every (2-dimensional) graph on a Polish space has a $\mathbf \Delta^1_2$ maximal independent set. It is very plausible that this can be extended to the countable support iteration. One line of attack might be to use a similar method to ours, where Cohen genericity is replaced by other kinds of genericity. For instance, the following was shown by Spinas in \cite{Spinas2001} (compare this with Proposition~\ref{prop:weightedmain}):

\begin{fact}
Let $M$ be a countable model, then there is a superperfect tree $T$ so that for any $x \neq y \in [T]$, $(x,y)$ is $\mathbb{M}^2$ generic over $M$, where $\mathbb{M}$ denotes Miller forcing.
\end{fact}

On the other hand, it is impossible to have that any three $x,y,z \in [T]$ are mutually generic. This follows from a fact due to Velickovic and Woodin (see \cite[Theorem 1]{Velickovic1998}) that there is a Borel function $h \colon (\omega^\omega)^3 \to 2^\omega$, such that for any superperfect $T$, $h''([T]^3) = 2^\omega$.
Also, $\mathbb{M}^3$ always adds a Cohen real (see e.g. the last paragraph in \cite{Brendle1999}). This means that Theorem~\ref{thm:finiteprod} can't hold for Miller forcing and $k \geq 3$, even for just equivalence relations. Namely, after adding a Cohen real, there can't be any $E_0$-transversal that is definable with parameters from the. This doesn't rule out though that the iteration might work. Let us ask the following question. 

\begin{quest}
Does Theorem~\ref{thm:mainmaintheorem} hold true for Miller forcing?
\end{quest}

A positive result would yield a model in which $\mathfrak{i}_{B} < \mathfrak{i}_{cl}$, as per $\mathfrak{d} \leq \mathfrak{i}_{cl}$. No result of this kind has been obtained so far. 

Another common way to iterate Sacks or splitting forcing is to use the countable support product. The argument in Remark~\ref{rem:E1} can be used to see that no definable $E_1$-transversals can exist in an extension by a (uncountably long) countable support product of Sacks or splitting forcing. This raises the question for which hypergraphs Theorem~\ref{thm:mainmaintheorem} applies to countable support products. 

\begin{quest}
Is there a nice characterization of hypergraphs for which Theorem~\ref{thm:mainmaintheorem} holds when using countable support products? 
\end{quest}

An interesting application of our method that appears in the authors thesis, is that $P$-points exist after iterating splitting forcing over a model of CH. To our knowledge this is different to any other method of $P$-point existence in the literature. 

Other applications are related to questions about families of reals and the existence of a well-order of the continuum. For example, it is known that after adding $\omega_1$ many Sacks reals, there is no well-order of the reals in $L(\mathbb{R})$. On the other hand, if we start with $V = L$, a $\Delta^1_2$ Hamel base exists in the extension. In particular, a Hamel base will exist in $L(\mathbb{R})$. This gives a new solution to a question by Pincus and Prikry \cite{Pincus1975}, which asks whether a Hamel basis can exist without a well-order of the reals. This has only been solved recently (see \cite{Schindler2018}). Our results solve this problem not just for Hamel bases but for a big class of families of reals. In an upcoming paper \cite{Schilhan2022} we will consider further applications to questions related to the Axiom of Choice. 

\bibliographystyle{plain}

\end{document}